\documentclass{amsart}
\usepackage[latin1]{inputenc}
\usepackage{epsf,graphicx}
\usepackage{amssymb,latexsym}
\usepackage[T1]{fontenc}
\usepackage{psfrag}
\usepackage{subfigure}

\theoremstyle{plain}
\newtheorem{lemma}{Lemma}[section]
\newtheorem{proposition}[lemma]{Proposition}
\newtheorem{example}[lemma]{Example}
\newtheorem{remark}[lemma]{Remark}
\newtheorem{theorem}[lemma]{Theorem}
\newtheorem{corollary}[lemma]{Corollary}

\theoremstyle{definition}
\newtheorem{definition}[lemma]{Definition}

\theoremstyle{remark}
\newtheorem{notation}[lemma]{Notation}

\begin{document}
\title{Complete system of analytic invariants for unfolded differential linear systems with an irregular singularity of Poincaré rank 1}
\author{Caroline Lambert, Christiane Rousseau}
\address{Département de mathématiques et de statistique\\Université de Montréal\\C.P. 6128, Succursale Centre-ville, Montréal (Qc), H3C 3J7, Canada}
\email{lambert@dms.umontreal.ca, rousseac@dms.umontreal.ca}
\thanks{Research supported by NSERC in Canada}
\date{\today}
\begin{abstract}
In this, paper, we give a complete system of analytic invariants for the unfoldings of nonresonant linear differential systems with an irregular singularity of Poincaré rank $1$ at the origin over a fixed neighborhood $\mathbb{D}_r$. The unfolding parameter $\epsilon$ is taken in a sector $S$ pointed at the origin of opening larger than $2 \pi$ in the complex plane, thus covering a whole neighborhood of the origin. For each parameter value $\epsilon \in S$, we cover $\mathbb{D}_r$ with two sectors and, over each sector, we construct a well chosen basis of solutions of the unfolded linear differential systems. This basis is used to find the analytic invariants linked to the monodromy of the chosen basis around the singular points. The analytic invariants give a complete geometric interpretation to the well-known Stokes matrices at $\epsilon=0$: this includes the link (existing at least for the generic cases) between the divergence of the solutions at $\epsilon=0$ and the presence of logarithmic terms in the solutions for resonance values of the unfolding parameter. Finally, we give a realization theorem for a given complete system of analytic invariants satisfying a necessary and sufficient condition, thus identifying the set of modules.
\end{abstract}

\maketitle

\pagestyle{myheadings}\markboth{C.Lambert, C.Rousseau}{Complete system of analytic invariants for unfolded differential linear systems}

\section{Introduction}

In this paper, we are interested in the unfolding of linear differential systems written as
\begin{equation}\label{E: syst confluent general}
y'=\frac{A(x)}{x^{k+1}} y,
\end{equation}
with $A$ a matrix of germs of analytic functions in $x$ at the origin such that $A(0)$ has distinct eigenvalues (nonresonant case), $x \in (\mathbb{C},0)$, $y \in \mathbb{C}^n$, and $k$ is a strictly positive integer called the Poincaré rank. We investigate the case of Poincaré rank $k=1$, but a prenormal form, from which formal invariants can be calculated, is obtained in the general case $k \in \mathbb{N}^*$ (Section \ref{S:prenormal}).

Most of the time, the formal solutions of the differential systems (\ref{E: syst confluent general}) at the irregular singular point $x=0$ are divergent and the Stokes phenomenon is observed. To understand this phenomenon, the irregular singular point can be split into regular singular points by a deformation depending on a parameter~$\epsilon$. A.~Glutsyuk~\cite{aG99} showed that the Stokes multipliers related to the system (\ref{E: syst confluent general}) can be obtained from the limits of transition operators of a perturbed system. In the generic deformations of the system (\ref{E: syst confluent general}) he considered, the parameter~$\epsilon$ is taken in sectors that do not cover a whole neighborhood of $\epsilon=0$. In particular, he restricts his study to parameter values for which the bases of solutions of the perturbed system around the regular singular points never contain logarithmic terms. In our previous paper \cite{cLcR}, we studied the confluence of two regular singular points of the hypergeometric equation into an irregular one. Our approach allowed us to cover a full neighborhood of the origin in the parameter space, the occurrence of logarithmic terms being embedded into a continuous phenomenon. Our description of the geometry however was not uniform in the parameter space. In this paper, we use the same approach for the unfolding of the systems (\ref{E: syst confluent general}): a whole neighborhood of $\epsilon=0$ is covered, in a ramified way.

One of the main questions of the field is the equivalence problem for systems of the form (\ref{E: syst confluent general}): under which conditions does there exist an invertible matrix of germs of analytic functions at the origin, $P(x)$, giving an equivalence between two arbitrary systems of the form (\ref{E: syst confluent general}) with $y_1=P(x)y_2$? The complete system of invariants for this equivalence relation contains formal invariants and an equivalence class of Stokes matrices. Many people have worked on it, and a final statement can be found in the paper of W.~Balser, W.B.~Jurkat and D.A.~Lutz~\cite{BJL79}. In this paper, we give the analog of this complete system of invariants for $1$-parameter families of systems that unfold generically the systems (\ref{E: syst confluent general}), with $k=1$. Over a fixed neighborhood $\mathbb{D}_r$ in $x$-space, the complete system of invariants for the unfolded systems consists of formal and analytic invariants. Formal invariants are obtained from the polynomial part of degree $k$ of a prenormal form. The system composed of this polynomial part is a formal normal form which we call the "model system". When $\epsilon$ tends to $0$, it converges to the usual polynomial formal normal form. $\mathbb{D}_r$ is covered with two sectorial domains converging to sectors when $\epsilon \to 0$. These sectorial domains are chosen so that, on their intersection, solutions of the model have the same behavior when $x$ tends to the singular points as solutions of the formal normal form at $\epsilon=0$. Analytic invariants are given by an equivalence class of unfolded Stokes matrices (defined in Section \ref{S:unfolded Stokes matrices}), obtained from the monodromy of a well chosen basis of solutions that is the unique basis having the same asymptotic behavior, over the intersection of the sectorial domains and near the singular points, as the "diagonal" basis of the model system. In dimension $n=2$ and $k=1$, the well chosen basis corresponds to a "mixed basis" composed of two solutions that are eigenvectors of the monodromy operator at the two different singular points.

Furthermore, we give a geometric interpretation to the Stokes matrices in the unfolded systems: in particular, we link the Stokes matrices to the presence of logarithmic terms in the general solution of the unfolded system for resonance values of the parameter. We also relate these analytic invariants to the monodromy of first integrals of associated Riccati systems. Unfolded Stokes matrices depend analytically on $\hat{\epsilon}$ over a ramified sector around the origin and we show that there exists a representative in their equivalence class which is $\frac{1}{2}$-summable in $\epsilon$.

Finally, we describe the moduli space. We give a necessary and sufficient condition for a given set of invariants to be realizable as the modulus of an equivalence class of differential systems.

The paper is organized as follows. In Section \ref{S:The Stokes
phenomenon}, we recall the known results for $\epsilon=0$ and $k=1$.
In Section \ref{S:prenormal}, we define the genericity of the
unfoldings of the systems we consider. We then state the equivalence
relations under which we classify these unfoldings (Definition
\ref{D:analytic equivalence of systems}) and we obtain the prenormal
form (Theorem \ref{T:prenormal form}) from which the formal
invariants can be calculated. Section~\ref{S:prenormal} is the only
section where the results are given for the general codimension $k
\in \mathbb{N}^*$ case instead of $k=1$. Section \ref{S:Complete
System} contains the proof of the theorem of analytic classification
of the unfolded systems (Theorem \ref{T:analy equiv}). We begin with
the identification of the formal invariants in Theorem \ref{T:formal
equiv}.  Sections \ref{S:inv man e0} to \ref{S:sectors x} contain a
description of the sectorial domains in $x$ and $\epsilon$ spaces.
We then give, in Theorem \ref{T:fundamental matrix e}, the well
chosen fundamental matrix of solutions of the unfolded systems,
which is valid over the domains previously defined. The definition
of the unfolded Stokes matrices and their equivalence classes is
given in Theorem \ref{T:unfolded Stokes marices} and Definition
\ref{D:equiv}. In Sections \ref{S:Unfolded Stokes matrices and
monodromy} to \ref{S:trivial rows columns}, we obtain: the relation
between the analytic invariants and the number of independent
solutions that are eigenvectors of the monodromy (Theorem
\ref{T:tout sur bases vp}); the monodromy of first integrals of the
associated Riccati systems (Theorem \ref{T:monodromy first
integral}); the auto-intersection relation (Definition
\ref{D:auto-intersection}); the existence of a representative in the
equivalence class of unfolded Stokes matrices that is
$\frac{1}{2}$-summable in $\epsilon$ (Theorem \ref{T:summ}); the
analytic equivalence to a simpler form for systems caracterized by
unfolded Stokes matrices in a (possibly permuted) block diagonal
form (Theorem \ref{T:decomp sum systems}) or with identically zero
(nondiagonal) entries in a column or row (Theorems \ref{T:trivial
columns theorem} and \ref{T:trivial rows theorem}). Finally,
Section~\ref{S:realization} contains the proof of the realization
theorem (Theorem \ref{T:reala globa}).

\section{The Stokes phenomenon and invariants, $\epsilon=0$}\label{S:The Stokes phenomenon}

We consider the system (\ref{E: syst confluent general}) and we denote by $\lambda_{1,0},...,\lambda_{n,0}$ the distinct eigenvalues of the matrix $A(0)$ that we can assume diagonal after a constant linear change of coordinates in the $y$ variable. There exists a formal transformation $\hat{H}(x)$ such that $\hat{H}(0)=I$ and such that $y=\hat{H}(x)z$ conjugates (\ref{E: syst confluent general}) with its \emph{formal normal form}
\begin{equation}\label{E: syst confluent normal}
z'=\frac{\Lambda_0+\Lambda_1 x+...+\Lambda_k x^{k}}{x^{k+1}} z,
\end{equation}
with
\begin{equation}
\Lambda_q=diag\{\lambda_{1,q},...,\lambda_{n,q}\}, \quad q=0,1,...,k.
\end{equation}
Generally, elements of the matrix $\hat{H}(x)$ are not analytic around $x=0$. But, there exists a covering of a punctured neighborhood of the origin in $x$-space by $2k$ sectors $\Omega_s$ such that on each of them there exists a unique invertible analytic transformation $H_s(x)$ conjugating (\ref{E: syst confluent general}) with (\ref{E: syst confluent normal}) and having the asymptotic series $\hat{H}(x)$ in $\Omega_s$. The comparison of these transformations on the intersections of the sectors $\Omega_s$ leads to the analytic invariants of the system (\ref{E: syst confluent general}). In this section, we recall these known results (for instance \cite{yIsY} pp.~351--372) in the case $k=1$, since they will organize our study of the unfolding.

Let us take the system (\ref{E: syst confluent general}) and its formal normal form (\ref{E: syst confluent normal}) which are written in the case $k=1$ as
\begin{equation}\label{E: syst confluent general k1}
y'=\frac{A(x)}{x^2} y
\end{equation}
and
\begin{equation}\label{E: syst confluent normal k1}
z'=\frac{\Lambda_0+\Lambda_1 x}{x^{2}} z,
\end{equation}
with the above assumptions on $A(0)$. We permute the coordinates of $y \in \mathbb{C}^n$ in order to have
\begin{equation}\label{E:inequalities}
\Re(\lambda_{1,0}) \geqslant \Re(\lambda_{2,0}) \geqslant ... \geqslant \Re(\lambda_{n,0})
\end{equation}
and, if $\Re(\lambda_{q,0})=\Re(\lambda_{j,0})$,
\begin{equation}
\Im(\lambda_{q,0})<\Im(\lambda_{j,0}), \quad q<j.
\end{equation}
Then, we have $\arg(\lambda_{q,0}-\lambda_{j,0}) \in ]-\frac{\pi}{2},\frac{\pi}{2}]$ for $q<j$. We rotate slightly the $x$-plane in the positive direction such that
\begin{equation}\label{E:inequalities strict}
\Re(\lambda_{q,0}-\lambda_{j,0})>0, \quad q<j.
\end{equation}
From now on, the order of the coordinates of $y$ and the $x$-coordinate (for $\epsilon=0$) are fixed. We are now ready to choose the covering sectors in $x$ using the notion of separation rays.
\begin{definition}\label{D:separation rays}
When $k=1$, the \emph{separation rays} in the $x$-plane corresponding to $\lambda_{q,0}, \lambda_{j,0} \in \mathbb{C}$, $\lambda_{q,0} \ne\lambda_{j,0}$, are the two rays such that
\begin{equation}
\Re\left( \frac{\lambda_{q,0} - \lambda_{j,0}}{x}\right)=0.
\end{equation}
\end{definition}
\begin{definition}\label{D:defsectors}
We define two open sectors $\Omega_D$ and $\Omega_U$ as
\begin{equation}\label{E:defsectors}
\begin{array}{lll}
\Omega_D=\{x \in \mathbb{C} : |x|<r, -(\pi+\delta) < \arg(x)< \delta  \},\\
\Omega_U=\{x \in \mathbb{C} : |x|<r, -\delta < \arg(x)< \pi+ \delta \},
\end{array}
\end{equation}
with $\delta>0$ chosen sufficiently small so that the closure of $\Omega_D$ (respectively $\Omega_U$) does not contain any separation rays located in the upper (respectively lower) half plane. Several restrictions on the radius of these sectors will be discussed later. The sectors are illustrated in Figure \ref{fig:Art2 4} with their intersection $\Omega_L \cap \Omega_R$.
\end{definition}
\begin{figure}[h!]
\begin{center}
{\psfrag{E}{$\Omega_L$}
\psfrag{D}{$\Omega_R$}
\psfrag{B}{$\Omega_U$}
\psfrag{C}{$\Omega_D$}
\includegraphics[width=5cm]{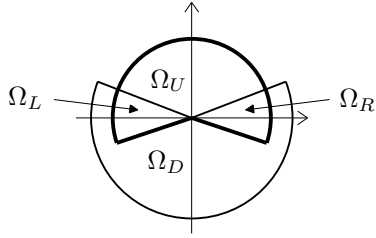}}
    \caption{Sectors $\Omega_D$ and $\Omega_U$ and their intersection $\Omega_L \cup \Omega_R$.}
    \label{fig:Art2 4}
\end{center}
\end{figure}

By the sectorial normalization theorem of Y.~Sibuya \cite{yS90} (p.~144), if $r$ is chosen sufficiently small, there exists over each sector $\Omega_s$ ($s=D,U$) a unique invertible matrix of analytic functions $H_s(x)$, asymptotic at the origin in $\Omega_s$ to a power series $\hat{H}(x)$ independent of $s$, such that $y=H_s(x)z$ conjugates (\ref{E: syst confluent general k1}) with its formal normal form (\ref{E: syst confluent normal k1}).

The Stokes phenomenon appears when considering the intersection of the sectors $\Omega_U$ and $\Omega_D$. Let $F(x)$ be the diagonal fundamental matrix solution of the formal normal form (\ref{E: syst confluent normal k1}) in the ramified domain $\{x \in \mathbb{C}: -(\pi+\delta) <\arg(x)<\pi+\delta\}$ given by
\begin{equation}
F(x)=x^{\Lambda_1}e^{-\frac{1}{x}\Lambda_0}.
\end{equation}
Let $F_s(x)$ be the restriction of $F(x)$ to $\Omega_s$, $s=D,U$. On each connected component of the intersection $\Omega_D \cap \Omega_U$ (Figure \ref{fig:Art2 4}), we have two bases of solutions of (\ref{E: syst confluent general k1}) given by $H_D(x)F_D(x)$ and $H_U(x)F_U(x)$, with
\begin{equation}
F_U(x)=\begin{cases} \begin{array}{lll}F_D(x), \quad  &\mbox{\rm{on} }  \Omega_R,   \\
                  F_D(x)e^{2 \pi i \Lambda_1}, \quad  &\mbox{\rm{on} }  \Omega_L.
\end{array}\end{cases}
\end{equation}
Each element of one basis may be expressed as a linear combination of elements of the other basis, giving the existence of matrices $C_R$ and $C_L$, such that
\begin{equation}\label{E:Stokes S2}
H_D(x)^{-1}H_U(x)=\begin{cases}F_D(x) C_R (F_D(x))^{-1}, \quad \mbox{on } \Omega_R, \\
F_D(x) C_L (F_D(x))^{-1}, \quad \mbox{on } \Omega_L.
\end{cases}
\end{equation}
The matrices $C_R$ and $C_L$ are unipotent, respectively upper and lower triangular, and they are called the \emph{Stokes matrices}. The \emph{Stokes phenomenon} occurs when at least one of these Stokes matrices is different from the identity matrix and it reflects the divergence of the formal transformation $\hat{H}(x)$.

As $F(x)K$ is also a fundamental matrix of the normal system (\ref{E: syst confluent normal k1}) for any nonsingular constant diagonal matrix $K$, two \emph{Stokes collections} $\{C_R,C_L\}$ and $\{C'_R,C'_L\}$ are said to be \emph{equivalent} if there exists a nonsingular constant diagonal matrix $K$ such that
\begin{equation}\label{E:equiv}
C'_l=K C_l K^{-1}, \quad l=L,R.
\end{equation}
The equivalence classes of Stokes collections are analytic invariants for the classification of the systems (\ref{E: syst confluent general k1}). The next two theorems are now standard in the literature.

\begin{definition}
Two systems are locally \emph{analytically equivalent} if there exists an invertible matrix $P$ of germs of analytic functions in $x$ at the origin such that the substitution $y_1=P(x)y_2$ transforms the system $y_1'=A_1(x)y_1$ into $y_2'=A_2(x)y_2$.
\end{definition}

\begin{theorem}
Two systems (\ref{E: syst confluent general k1}) with the same formal normal form (\ref{E: syst confluent normal k1}) are locally analytically equivalent in the neighborhood of $x=0$ if and only if their Stokes collections are equivalent in the sense (\ref{E:equiv}).
\end{theorem}
Related to a system (\ref{E: syst confluent general k1}), we thus have formal invariants, which are the coefficients of the matrices $\Lambda_0$ and $\Lambda_1$ in the formal normal form (\ref{E: syst confluent normal k1}), and analytic invariants, given by the equivalence class of the Stokes collections. The moduli space corresponding to these invariants has been completely described:
\begin{theorem}\label{T:realisation eps 0}
Any collection consisting of two unipotent matrices, an upper triangular one and a lower triangular one, can be realized as the Stokes collection of a nonresonant irregular singularity with a preassigned formal normal form.
\end{theorem}

Where do these invariants come from? What do they mean? The answer appears when unfolding.

\section{The prenormal form, $k \in \mathbb{N}^*$}\label{S:prenormal}

In this section, we unfold the systems (\ref{E: syst confluent general}), with $k \in \mathbb{N}^*$, and introduce a prenormal form in which formal invariants can be calculated from a polynomial part. The transformation from a system (\ref{E: syst confluent general}) to its prenormal form is analytic.

\subsection{Generic unfolding}\label{S:generic unfold}

We consider an unfolding of a system (\ref{E: syst confluent general}) of the form
\begin{equation}\label{E:unfolded syst 1}
f(\eta,x) y'=A(\eta,x) y,
\end{equation}
where $\eta=(\eta_0,...,\eta_{k-1}) \in \mathbb{C}^k$, $f(\eta,x)$ are germs of analytic functions at the origin such that $f(0,x)=x^{k+1}$ and $A$ is a matrix of germs of analytic functions at the origin satisfying $A(0,x)=A(x)$. We will restrict ourselves to functions $f(\eta,x)$ such that the unfolding is \emph{generic}. To define this term, we need the following proposition.
\begin{proposition}\label{P:generic}
After a translation $X=x+b(\eta)$, with $b$ a germ of analytic map such that $b(0)=0$, any linear differential system (\ref{E:unfolded syst 1}) may be written as
\begin{equation}\label{E:unfolded syst 2}
q^*(\eta,X) y'=A^*(\eta,X) y,
\end{equation}
with $A^*$ a matrix of germs of analytic functions in $(\eta,X)$ at
the origin satisfying $A^*(0,X)=A(0,x)$ and with
$q^*(\eta,X)=X^{k+1}+\epsilon_{k-1}(\eta)X^{k-1}+\epsilon_{k-2}(\eta)X^{k-2}...+\epsilon_{0}(\eta)$,
where $\epsilon_j(\eta)$ are germs of holomorphic functions at the
origin such that $\epsilon_j(0)=0$, $j=0,1,...,k-1$.
\end{proposition}

\begin{proof}
Given a particular $f(\eta,x)$, there exist, from Weierstrass preparation theorem, a unique invertible germ of analytic functions at the origin $u(\eta,x)$ and a unique Weierstrass polynomial $q(\eta,x)=x^{k+1}+\alpha_k(\eta)x^{k}+\alpha_{k-1}(\eta)x^{k-1}...+\alpha_{0}(\eta)$ such that $f(\eta,x)=u(\eta,x)q(\eta,x)$, where $\alpha_j(\eta)$ are germs of analytic functions at the origin satisfying $\alpha_j(0)=0$ for $j=0,1,...,k$. This yields the system
\begin{equation}
q(\eta,x) y'=\frac{A(\eta,x)}{u(\eta,x)} y.
\end{equation}
The change of variable $X=x+\frac{\alpha_k(\eta)}{k+1}$ yields the result.
\end{proof}

\begin{definition}
An unfolding is \emph{generic} if the germ of analytic map
\begin{equation}
\eta=(\eta_{0},...,\eta_{k-1}) \mapsto \epsilon=(\epsilon_0(\eta),...,\epsilon_{k-1}(\eta))
\end{equation} 
defined in Proposition \ref{P:generic} is invertible.
\end{definition}

We restrict our study to generic unfoldings of systems (\ref{E: syst confluent general}). From the equation (\ref{E:unfolded syst 2}), the genericity condition allows us to take $\epsilon=(\epsilon_{0},...,\epsilon_{k-1})$ as our new parameter. Let us change the notation of the variable $X$ by $x$ and from now on we do not make any more coordinate change on $x$. We write the generic unfoldings of the differential linear systems (\ref{E: syst confluent general}) as
\begin{equation}\label{E:unfolded syst 3}
p(\epsilon,x) y'=B(\epsilon,x) y,
\end{equation}
with
\begin{equation}\label{Weierstrass polynomial}
p(\epsilon,x)=x^{k+1}+\epsilon_{k-1}x^{k-1}+...+\epsilon_0,
\end{equation}
$\epsilon=(\epsilon_0,...,\epsilon_{k-1}) \in \mathbb{C}^k$ and $B(\epsilon,x)$ a matrix of germs of analytic functions at the origin satisfying $B(0,x)=A(x)$ as in (\ref{E: syst confluent general}).

\subsection{Equivalence classes of generic families of linear systems unfolding (\ref{E: syst confluent general})}\label{S:equivalence classes}

In this paper, we are interested in equivalence classes of systems (\ref{E:unfolded syst 3}). We use the same terminology as the one used for the classification of the systems (\ref{E: syst confluent general}), since it agrees with it when $\epsilon=0$:
\begin{definition}\label{D:analytic equivalence of systems}
Two systems $y'=A(\epsilon,x)y$ and $z'=B(\epsilon,x)z$ are locally \emph{analytically equivalent} (respectively \emph{formally equivalent}) if there exists an invertible matrix $P$ of germs of analytic functions of $(\epsilon,x)$ at the origin (respectively an invertible matrix of formal series in $(\epsilon,x)$) such that the substitution $y=P(\epsilon,x)z$ transforms one system into the other.
\end{definition}

We search for a complete system of analytic invariants for the systems (\ref{E:unfolded syst 3}) under analytic equivalence. First, we choose a representative of each equivalence class called the \emph{prenormal form}.

\subsection{Prenormal form}

The families of systems (\ref{E:unfolded syst 3}) have singularities at $x=x_l$, for $x_l$ such that $p(\epsilon,x_l)=0$. When looking at solutions around these singularities, we need to evaluate the eigenvalues of $B(\epsilon,x_l)$. With the next theorem, we express them as the values at $x_l$ of polynomials of degree less than or equal to $k$.
\begin{theorem}\label{T:prenormal form}
The family of systems (\ref{E:unfolded syst 3}) is analytically equivalent to a family in the prenormal form
\begin{equation}\label{E:prenormal system}
p(\epsilon,x) y'=B(\epsilon,x)y,
\end{equation}
where
\begin{equation}\label{E:form of B prenormal}
B(\epsilon,x)=\Lambda(\epsilon,x)+p(\epsilon,x)R(\epsilon,x),
\end{equation}
\begin{equation}\label{E: lamda(e,x) model}
\Lambda(\epsilon,x)=diag \{ \lambda_{1}(\epsilon,x),...,\lambda_{n}(\epsilon,x)\},
\end{equation}
\begin{equation}\label{E:aij}
\lambda_i(\epsilon,x)=\lambda_{i,0}(\epsilon)+\lambda_{i,1}(\epsilon)x+...+\lambda_{i,k}(\epsilon)x^k,
\end{equation}
$\lambda_{j,q}$ are germs of analytic functions in $\epsilon$ at the origin, $p(\epsilon,x)$ is given by (\ref{Weierstrass polynomial}) and $R(\epsilon,x)$ is a matrix of germs of analytic functions at the origin.
\end{theorem}

\begin{proof}
As $A(0)$ in (\ref{E: syst confluent general}) is a diagonal matrix, $B(0,0)=A(0)$ is also diagonal with distinct eigenvalues. We take $x$ in a neighborhood $\mathbb{D}_r$ of the origin such that the eigenvalues of $A(x)$ are distinct. Let us prove that there exists $P(\epsilon,x)$ a matrix of germs of analytic functions at the origin that diagonalizes $B(\epsilon,x)$ for $x \in \mathbb{D}_r$ and for $\epsilon$ sufficiently small. $P(0,0)$ can be any nonsingular diagonal matrix, let us take $P(0,0)=I$. For $\epsilon$ small and $x \in \mathbb{D}_r$, the eigenvalues of $B(\epsilon,x)$ are distinct and are analytic functions $\nu_i(\epsilon,x)$ of $(\epsilon,x)$ by the implicit function theorem. Also, there exists a unique analytic eigenvector $v_i(\epsilon,x)$ relative to the eigenvalue $\nu_i(\epsilon,x)$ having the $i^{th}$ component equal to one (this is obtained with the implicit function theorem, taking $F_i(w,\epsilon,x)=0$, where $F_i(w,\epsilon,x)=B_i(\epsilon,x)v_i$, $w=(w_1,...,w_{n-1})$, $v_i=(w_1,...,w_{i-1},1,w_{i},...,w_{n-1})$ and where $B_i(\epsilon,x)$ is the matrix obtained by removing the $i^{th}$ line of $(B(\epsilon,x)-\nu_i(\epsilon,x)I)$). We then take the $i^{th}$ column of $P(\epsilon,x)$ equal to $v_i(\epsilon,x)$.

Finally, by taking $z=P(\epsilon,x)^{-1}y$, the new system $p(\epsilon,x) z'=B^*(\epsilon,x) z$ satisfies $B^*(\epsilon,x)=diag\{\nu_1(\epsilon,x),...,\nu_n(\epsilon,x)\}+p(\epsilon,x)P(\epsilon,x)^{-1} \frac{\partial P(\epsilon,x)}{\partial x}$ and is analytically equivalent to the original system. Dividing $\nu_i(\epsilon,x)$ by $p(\epsilon,x)$, we get $\nu_i(\epsilon,x)= c_i(\epsilon,x)p(\epsilon,x)+\lambda_{i,0}(\epsilon)+\lambda_{i,1}(\epsilon)x+...+\lambda_{i,k}(\epsilon)x^k$, from which the result follows.
\end{proof}

\begin{remark}
The polynomial part $\Lambda(\epsilon,x)$ of the prenormal form is completely characterized by $n(k+1)$ quantities $\lambda_{j,q}(\epsilon)$ (with $q=0,1,...,k$ and $j=1,2,...,n$). For $\epsilon$ fixed such that the singular points are nonresonant, the collection of the well-known formal invariants at all singular points contains also $n(k+1)$ elements (for instance the collection of the eigenvalues of the residue matrices if the singular points are all distinct).
\end{remark}

For the rest of the paper, we only discuss systems in prenormal form (\ref{E:prenormal system}).

\section{Complete system of invariants in the case $k=1$}\label{S:Complete System}

This section leads to the complete description of the analytic equivalence classes of generic families of systems in the prenormal form (\ref{E:prenormal system}), limiting ourselves to the case $k=1$. Let us write these systems as
\begin{equation}\label{E:prenormal system k1}
(x^2-\epsilon) y'=B(\epsilon,x)y,
\end{equation}
where
\begin{equation}\label{E:form of B prenormal k1}
B(\epsilon,x)=\Lambda(\epsilon,x)+(x^2-\epsilon)R(\epsilon,x),
\end{equation}
with
\begin{equation}\label{E: C(e,x) model k1}
\begin{array}{lll}
\Lambda(\epsilon,x)&=diag \{ \lambda_{1}(\epsilon,x),...,\lambda_{n}(\epsilon,x)\}, \\
&=\Lambda_{0}(\epsilon)+\Lambda_{1}(\epsilon)x,
\end{array}
\end{equation}
and
\begin{equation}\label{E:D def}
\Lambda_q(\epsilon)=diag \{ \lambda_{1,q}(\epsilon), ...,\lambda_{n,q}(\epsilon) \}, \quad q=0,1.
\end{equation}
The quantity $\lambda_{j,0}(0)=\lambda_{j}(0,0)$ correspond to $\lambda_{j,0}$ defined in Section \ref{S:The Stokes phenomenon}. Hence, relation (\ref{E:inequalities strict}) may be written as
\begin{equation}\label{E:inequalities strict k1}
\Re(\lambda_{q}(0,0)-\lambda_{j}(0,0))>0, \quad q<j.
\end{equation}
This ordering on the eigenvalues of $\Lambda(\epsilon,x)$ at $(\epsilon,x)=0$ will be kept for $\epsilon \ne 0$ and $|x| \leq \sqrt{|\epsilon|}$ by taking $\epsilon$ sufficiently small (see Remark \ref{R:order ai}).

We like to call
\begin{equation}\label{E:model system k1}
(x^2 -\epsilon) z'=\Lambda(\epsilon,x)z
\end{equation}
the \emph{model system}. When $\epsilon=0$, it corresponds to the formal normal form.

In the systems (\ref{E:prenormal system k1}) and (\ref{E:model system k1}), the irregular singular point at $\epsilon=0$ splits into two regular singular points when $\epsilon \ne 0$ (in the present context, these points are Fuchsian).

\begin{notation}
We denote the zeros of $x^2-\epsilon$ by
\begin{equation}\label{E:xLR}
x_L=\sqrt{\epsilon} \quad  \mbox {and } \quad x_R=-\sqrt{\epsilon}.
\end{equation}
These points are respectively at the left and at the right of the origin when $\sqrt{\epsilon} \in \mathbb{R}_{-}$ (this will make sense with Definition \ref{D:sector S}).
\end{notation}

The model system has a fundamental matrix of solutions given by
\begin{equation}\label{E:matrix of model}
F(\epsilon,x)=diag \{f_1(\epsilon,x),...,f_n(\epsilon,x) \}
=\begin{cases}\begin{array}{lll}&(x-x_R)^{\mathcal{U}_R}(x-x_L)^{\mathcal{U}_L},  &\epsilon \ne 0, \\
&x^{\Lambda_{1}(0)} \exp{(-\frac{\Lambda_{0}(0)}{x})}, &\epsilon = 0,
\end{array}
\end{cases}
\end{equation}
with
\begin{equation}\label{E:mu j l k1}
\mathcal{U}_{l}=\frac{1}{2x_l}\Lambda(\epsilon,x_l)=\frac{1}{2x_l}\Lambda_{0}(\epsilon )+\frac{1}{2}\Lambda_{1}(\epsilon)=diag \{\mu_{1,l},...,\mu_{n,l} \}, \quad l=L,R.
\end{equation}

The functions $f_j(\epsilon,x)$ will be at the core of the construction of the sectorial domains in the $x$-space done in Section \ref{S:sectors x}.
\begin{remark}\label{R:solutions fj}
The solutions $f_j(\epsilon,x)$ of the model system given by (\ref{E:matrix of model}) are analytic in $(\epsilon,x)$ for $\epsilon$ in a punctured neighborhood of $\epsilon=0$ and for $x$ in a simply connected domain that does not contain any singular point $x=x_l$, for $l=L,R$. These functions converge uniformly on simply connected compact sets of the punctured disc $\mathbb D_r^*$ of radius $r$ to $f_j(0,x)$ when $\epsilon \to 0$.
\end{remark}

Let us immediately state notations related to formal invariants that we will frequently use in this paper. \begin{notation}
We define
\begin{equation}\label{E:Dx1}
 D_R=  e^{-2 \pi i \mathcal{U}_R}, \qquad D_L=e^{2 \pi i \mathcal{U}_L}.
\end{equation}
and
\begin{equation}\label{E:Delta}
\begin{array}{lll}
\Delta_{sj,l}&=(D_l)_{ss}(D_l^{-1})_{jj},  \quad l=L,R, \\
&=\begin{cases}e^{2 \pi i (\mu_{s,l}-\mu_{j,l})},\quad &l=L, \\
e^{2 \pi i (\mu_{j,l}-\mu_{s,l})}, \quad &l=R,
\end{cases}
\end{array}
\end{equation}
with ${U}_l$ and $\mu_{j,l}$ given by (\ref{E:mu j l k1}). We have
\begin{equation} \label{E:prod DrDl}
D_R^{-1}D_L=e^{2 \pi i \Lambda_1(\epsilon)},
\end{equation}
with $\Lambda_1(\epsilon)$ given by (\ref{E:D def}). We will see that $D_L$ (respectively $D_R$) is the matrix representing the monodromy around $x=x_L$ in the positive direction (respectively around $x=x_R$ in the negative direction) when acting on the fundamental matrix of solutions (\ref{E:matrix of model}) of the model system. $e^{2 \pi i \Lambda_1(\epsilon)}$ represents the monodromy around both singular points, in the positive direction.
\end{notation}

The model system (\ref{E:model system k1}) corresponding to a system (\ref{E:prenormal system k1}) contains all the information on the formal invariants:

\begin{theorem}\label{T:formal equiv}
Two systems (\ref{E:prenormal system k1}) are formally equivalent if and only if they have the same model system. Hence, the complete system of formal invariants of the systems (\ref{E:prenormal system k1}) is given by the $n$ (degree $1$) polynomials $\lambda_{i}(\epsilon,x)$ in the polynomial part of the prenormal form.
\end{theorem}

\begin{proof}
By the Poincaré-Dulac Theorem (see \cite{yIsY} p. 45) applied to the nonlinear system
\begin{equation}
\begin{cases}
\dot{y}=B(\epsilon,x)y,\\
\dot{x}=x^2-\epsilon,\\
\dot{\epsilon}=0,\end{cases}
\end{equation}
there exists an invertible formal transformation $Y=T(\epsilon,x)y$  at $(\epsilon,x)=(0,0)$ eliminating nondiagonal terms in (\ref{E:prenormal system k1}) and yielding a diagonal $R(\epsilon,x)$ in (\ref{E:form of B prenormal k1}). Then, the transformation $z=e^{-\int_0^x R(\epsilon,x)dx}Y$ leads to the model. Hence, letting $J(\epsilon,x)=e^{-\int_0^x R(\epsilon,x)dx} T(\epsilon,x)$, the invertible transformation $z=J(\epsilon,x)y$ conjugates formally a system (\ref{E:prenormal system k1}) to its model.

Let us take two systems of the form (\ref{E:prenormal system k1}) with the same model system, each of them formally conjugated to the model with $J^i(\epsilon,x)$. The transformation $\mathcal{Q}(\epsilon,x)=(J^1(\epsilon,x))^{-1}J^2(\epsilon,x)$ leads a formal equivalence between the two systems.

On the other hand, let us suppose that two systems $(x^2-\epsilon)y_1'=B^1(\epsilon,x)y_1$ and $(x^2-\epsilon)y_2'=B^2(\epsilon,x)y_2$, with $B^i(\epsilon,x)=\Lambda^i(\epsilon,x)+(x^2-\epsilon)R^i(\epsilon,x)$, are formally equivalent via $y_1=\mathcal{Q}(\epsilon,x)y_2$, each of them formally conjugated to its model with $z_i=J^i(\epsilon,x)y_i$. We obtain that $P(\epsilon,x)=J^1(\epsilon,x)\mathcal{Q}(\epsilon,x) (J^2(\epsilon,x))^{-1}$ is an invertible formal transformation from the second model system $(x^2-\epsilon)z_2'=\Lambda^2(\epsilon,x) z_2$ to the first model system $(x^2-\epsilon)z_1'=\Lambda^1(\epsilon,x) z_1$. Formally, we thus have
\begin{equation}
(x^2-\epsilon)\frac{\partial}{\partial x}P(\epsilon,x)+P(\epsilon,x)\Lambda^2(\epsilon,x)=\Lambda^1(\epsilon,x)P(\epsilon,x).
\end{equation}
By considering this equality for each power of $\epsilon^p x^q$, we obtain that $\Lambda^1(\epsilon,x)=\Lambda^2(\epsilon,x)$ (and that $P(\epsilon,x)$ is a diagonal matrix depending only on $\epsilon$). Hence, the two systems have the same model system.
\end{proof}

Around each singular point, the system (\ref{E:prenormal system k1}) has a well-known basis of solutions (given by eigenvectors of the monodromy operator) that we present in Theorem \ref{T:tout sur bases vp}, but the problem with this basis is that it is not defined for an infinite set of resonance values of $\epsilon$ which accumulate at $\epsilon=0$. We want to give a unified treatment which highlights the fact that the Stokes phenomenon at $\epsilon=0$ organizes, in the unfolding, the form of solutions at the resonance values of the parameter. Thus, we rather use a new basis that is defined for all parameter values in a sector of opening greater than $2 \pi$ in the universal covering of the $\epsilon$-space punctured at $\epsilon=0$. To find this particular basis, we choose to consider the solutions of the linear systems in the complex projective space.

\subsection{The projective space}\label{S:systeme de Riccati}

The system (\ref{E:prenormal system k1}) is invariant under $y \to c y$, with $c \in \mathbb{C^*}$. Taking charts in the complex projective space, it gives $n$ particular Riccati matrix differential equations. We introduce $t$ by $\frac{dx}{dt}= \dot{x}=x^2-\epsilon$ and replace them by $n$ systems of ordinary differential equations (indexed by $j$)
\begin{equation}\label{E:Riccati system expli}
\begin{cases}
\begin{array}{lll}
\frac{dx}{dt} &=& x^2-\epsilon, \\
\frac{d}{dt}\frac{(y)_q}{(y)_j}&=&\left( \lambda_q(\epsilon,x) - \lambda_j(\epsilon,x) \right) \frac{(y)_q}{(y)_j} \\&&+ (x^2-\epsilon) \sum_{i=1}^n  \frac{(y)_i}{(y)_j} \left( (R(\epsilon,x))_{qi} - (R(\epsilon,x))_{ji}\frac{(y)_q}{(y)_j} \right), \quad q \ne j,
\end{array}
\end{cases}
\end{equation}
that we call the \emph{Riccati systems}.
\begin{notation}\label{N:projective not}
Let $v$ be a $n$-dimensional column vector. We define
\begin{equation}\label{E:change of variable Riccati}
[v]_j=\left(-\frac{(v)_1}{(v)_j},...,-\frac{(v)_{j-1}}{(v)_j},-\widehat{\frac{(v)_j}{(v)_j}},-\frac{(v)_{j+1}}{(v)_j},...,-\frac{(v)_n}{(v)_j}\right)^T,
\end{equation}
where $(v)_i$ is the $i^{th}$ component of the column vector $v$ and where the hat denotes omission.
\end{notation}

\begin{remark}
Following Notation \ref{N:projective not}, the $j^{th}$ Riccati system associated to the linear system (\ref{E:prenormal system k1}) may be written as
\begin{equation}\label{Riccati system}
\begin{cases}\begin{array}{lll}
\frac{d}{dt}x &= x^2-\epsilon, \\
\frac{d}{dt}[y]_j & =-B^0_j (\epsilon,x) +B^1_j(\epsilon,x)[y]_j+\left(B^2_j(\epsilon,x)[y]_j\right) [y]_j ,
\end{array}\end{cases}
\end{equation}
with, denoting $I$ the $(n-1) \times (n-1)$ identity matrix,
\begin{equation}
\begin{array}{lll}
B^0_j (\epsilon,x) \, : j^{th} \mbox{ column of } B(\epsilon,x)  \mbox{ except } \left(B(\epsilon,x)\right)_{jj};\\
B^1_j (\epsilon,x) \,: \left(B(\epsilon,x) \mbox{ without $j^{th}$ column and $j^{th}$ line} \right) -\left(B(\epsilon,x)\right)_{jj} I;\\
B^2_j (\epsilon,x) \, : j^{th} \mbox{ line of } B(\epsilon,x)  \mbox{ except } \left(B(\epsilon,x)\right)_{jj}.
\end{array}
\end{equation}
\end{remark}

\subsection{Radius of the sectors in the $x$-space when $\epsilon=0$}\label{S:inv man e0}
In order to obtain a basis of solutions of the linear system (\ref{E:prenormal system k1}), we will find in Section \ref{S:invar man e} particular solutions (defined for $\epsilon$ in a ramified sector and for $x$ in sectorial domains) of the Riccati systems (\ref{Riccati system}). To ensure that these solutions will converge uniformly on compact sets to solutions $[y]_j=G_{j,s}(0,x)$ (defined over the sectors $\Omega_s$ given by (\ref{E:defsectors}) for $s=D,U$), we choose in this section the radius of $\Omega_s$.

Let us first define the solution $[y]_j=G_{j,s}(0,x)$. When $\epsilon=0$, if the radius $r$ of $\Omega_s$ is chosen sufficiently small, there exists a unique fundamental matrix of solutions of the system (\ref{E:prenormal system k1}) that can be written as
\begin{equation}\label{E:confluent solution}
W_s(0,x)=H_s(0,x)F_s(0,x), \quad \mbox{on } \Omega_s, \, s=D,U,
\end{equation}
where $F_s(0,x)$ is the restriction of $F(0,x)$ given by (\ref{E:matrix of model}) to the sectorial domain $\Omega_s$, and where $H_s(0,x)$ is an invertible matrix of functions which are analytic on $\Omega_s$ and continuous on its closure, satisfying $H_s(0,0)=I$. $H_s(0,x)$ links the system to its formal normal form and is obtained by the sectorial normalization theorem of Y.~Sibuya \cite{yS90} (p.~144), as mentioned in Section \ref{S:The Stokes phenomenon}.

\begin{notation}\label{N:def G0x}
The solution corresponding to the $j^{th}$ column of $W_s(0,x)$ in the $j^{th}$ Riccati system passes through $(x,[y]_j)=(0,0)$ and is tangent to the $x$ direction, we denote it as $[y]_j=G_{j,s}(0,x)$.
\end{notation}

Let us now specify how we restrict the radius of $\Omega_s$.
\begin{proposition}\label{P:graph confined e non zero}
Let us define the region
\begin{equation}\label{E:hypersurf}
\mathcal{V}^j=\left\{ (x,[y]_j) \in \mathbb{C} \times \mathbb{CP}^{n-1}:\left|\frac{(y)_i}{(y)_j}\right| \leq  |x|, \, \forall i \in \{1,...,n \} \backslash \{j\} \right\}.
\end{equation}
The boundary of $\mathcal{V}^j$ is $\bigcup_{\begin{subarray}{lll}i=1\\i\ne j \end{subarray}}^{n} \mathcal{V}^j_i$, with
\begin{equation}\label{E:hypersurf bord}
\mathcal{V}^j_i=\left\{(x,[y]_j) \in \mathbb{C} \times \mathbb{CP}^{n-1}:\left|\frac{(y)_i}{(y)_j}\right|=|x|, \, \left|\frac{(y)_k}{(y)_j}\right| \leq |x| \mbox{ if }k \ne i,j\right\}, \quad i \ne j.
\end{equation}
The radius $r$ of $\Omega_s$, $s=D,U$, can be chosen sufficiently small so that the graph $[y]_j=G_{j,s}(0,x)$ is confined inside $\mathcal{V}^j$, for all $j \in \{1,...,n \}$.
\end{proposition}

\begin{proof}
We consider (\ref{E:Riccati system expli}) for $\epsilon=0$. We have
\begin{equation}
\left|\frac{d}{dt}|x |  \right|=\frac{\left|\Re(\bar{x}\dot{x})\right|}{|x|}=\frac{\left|\Re(\bar{x}x^2)\right|}{|x|}\leq |x |^2
\end{equation}
and
\begin{equation}
\frac{1}{|x|}\left|\frac{d}{dt}\left|\frac{(y)_i}{(y)_j}\right| \, \right| \geq   |\Re( \lambda_{i,0}(0)-\lambda_{j,0}(0))|-v_{ij}(x),
\end{equation}
with
\begin{equation}\label{E:vij}
\begin{array}{lll}
v_{ij}(x)=&|\lambda_{i,1}(0)-\lambda_{j,1}(0)||x|
+  |x|^2 \sum_{\begin{subarray}{lll}k=1\\k \ne j \end{subarray}}^n|(R(0,x))_{ik}|\\
&+|x|\left(|(R(0,x))_{ij}|+|x|²\sum_{\begin{subarray}{lll}k=1\\k \ne j \end{subarray}}^n|(R(0,x))_{jk}|+|x|(R(0,x))_{jj}\right).
\end{array}
\end{equation}
Let us choose $0<\eta<1$. As $|\Re( \lambda_{i,0}(0)-\lambda_{j,0}(0))|$>0, we can take the radius $r$ of $\Omega_D$ and $\Omega_U$ sufficiently small so that
\begin{equation}\label{E:inega vij}
v_{ij}(x)+|x|<(1-\eta)|\Re( \lambda_{i,0}(0)-\lambda_{j,0}(0))| , \quad x \in \Omega_D \cup \Omega_U, \, i,j \in \{1,...,n\}, \, i \ne j.
\end{equation}
This implies
\begin{equation}\label{E:cond cone}
\left|\frac{d}{dt}|x |  \right|<\left|\frac{d}{dt}\left|\frac{(y)_i}{(y)_j}\right| \, \right|, \quad \mbox{for }\begin{cases}\begin{array}{lll}(x,[y]_j) \in \mathcal{V}^j_i, \\ x \in \Omega_s, \quad &s=D,U, \\  i,j \in \{1,...,n\}, \quad &i \ne j.\end{array}\end{cases}
\end{equation}
Since the graph $[y]_j=G_{j,s}(0,x)$ contains the point $(x,[y]_j)=(0,0)$ and is tangent to the $x$-plane, it is confined inside $\mathcal{V}^j$ (if a solution parametrized by a curve in complex time living on the graph $[y]_j=G_{j,s}(0,x)$ were to intersect a boundary component of $\mathcal{V}^j$, then (\ref{E:cond cone}) would not be satisfied). We introduced the parameter $\eta$ in order to have in the unfolding a similar property (see Proposition \ref{P:sous le cone}).
\end{proof}

\subsection{Sector in the parameter space}\label{S:sectors e}

Let us specify the sector on the universal covering of the $\epsilon$-space punctured at the origin with which we will work.
\begin{remark}\label{R:order ai}
We take $\epsilon$ sufficiently small in order to have:
\begin{equation}\label{E:order ai}
\Re((\lambda_q(\epsilon,x)-\lambda_j(\epsilon,x))>0, \quad |x|\leq \sqrt{|\epsilon|}, \, q<j.
\end{equation}
Hence, we have the same ordering of the eigenvalues of $\Lambda(\epsilon,x_l)$ as the one for $\Lambda(0,0)$ given by (\ref{E:inequalities strict k1}).
\end{remark}
\begin{definition} \label{D:sector S}
We define the sector $S$, of opening larger than $2\pi$ and covering completely a punctured neighborhood of $\epsilon=0$, as
\begin{equation}\label{D:def S}
S=\{\hat{\epsilon} \in \mathbb{C} \, : \, 0<|\hat{\epsilon}|<\rho, \, \arg(\hat{\epsilon}) \in (\pi-2\gamma,3\pi +2\gamma ) \}
\end{equation}
(see Figure \ref{fig:Art2 1}).
\begin{figure}[h!]
\begin{center}
{\psfrag{A}{$\Re(\sqrt{\epsilon})$}
\psfrag{B}{$\Im(\sqrt{\epsilon})$}
\psfrag{D}{$\Re(\epsilon)$}
\psfrag{E}{$\Im(\epsilon)$}\psfrag{C}{$S$}\psfrag{F}{$\sqrt{S}$}
\includegraphics[width=8cm]{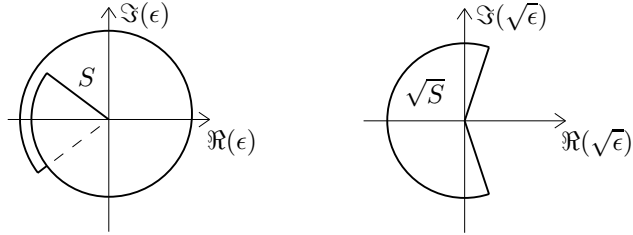}}
    \caption{Sector $S$ in terms of the parameters $\epsilon$ and $\sqrt{\epsilon}$.}
    \label{fig:Art2 1}
\end{center}
\end{figure}
In (\ref{D:def S}), any $\gamma>0$ such that $\gamma(1+2\frac{\gamma}{\pi}) < \theta_0$ can be chosen, with $\theta_0$ the maximum angle in $(0,\frac{\pi}{2})$ such that
\begin{equation}
\Re(e^{\pm i \theta_0}(\lambda_q(0,0)-\lambda_j(0,0)))\geq 0, \quad q<j,
\end{equation}
with $\lambda_{j}(\epsilon,x)$ as in (\ref{E: C(e,x) model k1}) ($\theta_0$ exists because of (\ref{E:inequalities strict k1})). Once $\gamma$ is chosen, the radius $\rho$ is chosen to ensure that there exists $C>0$ for which
\begin{equation}\label{E:theta1 satisfy}
\Re(e^{\pm i \gamma(1+2\frac{\gamma}{\pi})}(\lambda_q(\epsilon,\hat{x}_l)-\lambda_j(\epsilon,\hat{x}_l))) > C > 0, \quad q<j, \quad l=L,R, \quad \hat{\epsilon} \in S.
\end{equation}
We will restrict a few other times the value of $\rho$ (in particular, to construct the sectorial domains in the $x$-variable in Section \ref{S:sectors x} and to ensure that Proposition \ref{P:sous le cone} is true).
\end{definition}

\begin{remark}
We did not include the real values of the parameter such that $\arg(\hat{\epsilon})=0$ into $S$, but rather the values of the parameter such that $\arg(\hat{\epsilon})=2 \pi$. This results from defining first the sector in terms of $\sqrt{\epsilon}$ in order to avoid decreasing resonance values of the parameter (see Definition \ref{D:resonant values}) that are approching the ray $\arg(\sqrt{\epsilon})=\arg(\lambda_{q,0}(0)-\lambda_{j,0}(0))$ for $q<j$ (and to include those approaching it for $q>j$). By a change of parameter (see Remark \ref{R:other half resonance}), we can choose to include values such that $\arg(\hat{\epsilon})=0$ into $S$.
\end{remark}

\begin{notation}\label{N:autointersec}
We denote the auto-intersection of $S$ as $S_\cap$. For values of the parameter in $S_\cap$, we denote
\begin{equation}\label{E:def tilde et bar eps}
\tilde{\epsilon}=\bar{\epsilon} e^{2 \pi i}  \in S_\cap
\end{equation}
(see Figure \ref{fig:Art2 16}).
\end{notation}

\begin{figure}[h!]
\begin{center}
{\psfrag{A}{$\Re(\epsilon)$}
\psfrag{D}{$\Im(\epsilon)$}
\psfrag{G}{$S$}
\psfrag{K}{$\sqrt{S}$}
\psfrag{F}{$\bar{\epsilon}$}
\psfrag{E}{$\tilde{\epsilon}$}
\psfrag{I}{$\Re(\sqrt{\epsilon})$}
\psfrag{J}{$\Im(\sqrt{\epsilon})$}
\psfrag{L}{$\sqrt{\bar{\epsilon}}$}
\psfrag{M}{$\sqrt{\tilde{\epsilon}}$}
\includegraphics[width=9cm]{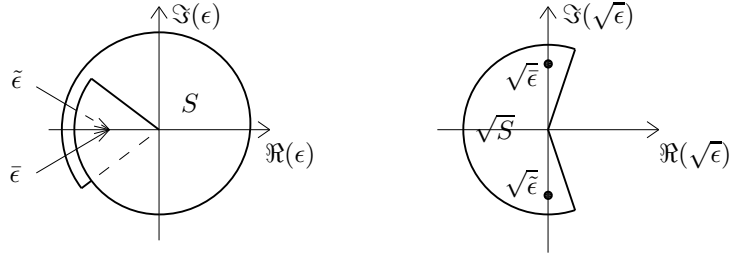}}
    \caption{Example of values of $\bar{\epsilon}$ and $\tilde{\epsilon}$ in $S_\cap$ (in terms of $\epsilon$ and $\sqrt{\epsilon}$).}
    \label{fig:Art2 16}
\end{center}
\end{figure}

\begin{notation}
We will write the hat symbol over quantities that depend on $\hat{\epsilon} \in S$ (for example $\hat{x}_L$). We will write the symbols $\tilde{\,}$ (respectively $\bar{\,}$ ) over quantities that depend on $\tilde{\epsilon} \in S_\cap$ (respectively $\bar{\epsilon} \in S_\cap$). When we use the hat symbol for values of the parameter in $S_\cap$, we mean that $\hat{\epsilon}$ could either be $\bar{\epsilon}$ or $\tilde{\epsilon}$.
\end{notation}

\subsection{Sectorial domains in $x$}\label{S:sectors x}
For the rest of Section \ref{S:Complete System}, $x$ belongs to a disk of radius $r$ determined by Proposition \ref{P:graph confined e non zero}. Let us now explain the construction of the sectorial domains in the complex plane for the $x$-variable. The boundary of these domains will be defined from solutions of the equation
\begin{equation}
\dot{x}=(x^2-\epsilon),
\end{equation}
allowing complex time. More precisely, passing to the $t$-variable, we have
\begin{equation}\label{E:t of x}
t(x)=\begin{cases}\begin{array}{lll}
&\frac{1}{2\sqrt{\epsilon}} \log \left(\frac{x-\sqrt{\epsilon}}{x+\sqrt{\epsilon}}\right), &\epsilon \ne 0,\\
&-\frac{1}{x}, &\epsilon=0.
\end{array}
\end{cases}
\end{equation}
For $\epsilon=0$, we cover the disk of radius $r$ with two sectorial domains $\Omega^{0}_D$ and $\Omega^{0}_U$ (see Figure \ref{fig:Art2 13}) included respectively inside the sectors $\Omega_D$ and $\Omega_U$ defined by (\ref{E:defsectors}). The sectorial domains $\Omega^{0}_D$ and $\Omega^{0}_U$ correspond respectively, in the $t$-variable, to the sectorial domains $\Gamma^{0}_D$ and $\Gamma^{0}_U$ illustrated in Figure \ref{fig:Art2 9}.

\begin{figure}[h!]
\begin{center}
{\psfrag{E}{$\Omega_L^{0}$}
\psfrag{D}{$\Omega_R^{0}$}
\psfrag{B}{$\Omega^{0}_D$}
\psfrag{C}{$\Omega^{0}_U$}
\includegraphics[width=6cm]{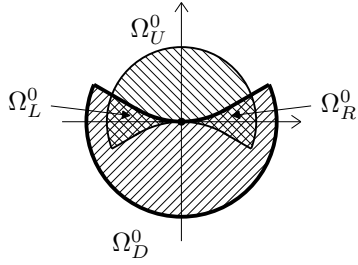}}
    \caption{Sectorial domains in the $x$-variable when $\epsilon=0$.}
    \label{fig:Art2 13}
\end{center}
\end{figure}

\begin{figure}[h!]
\begin{center}
{\psfrag{J}{$\Gamma^{0}_D$}
\psfrag{I}{$\Gamma^{0}_U$}
\psfrag{E}{$\Gamma^{0}_L$}
\psfrag{F}{$\Gamma^{0}_R$}
\includegraphics[width=6cm]{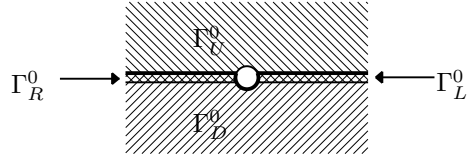}}
    \caption{Sectorial domains in the $t$-variable when $\epsilon=0$.}
    \label{fig:Art2 9}
\end{center}
\end{figure}

When $\epsilon \ne 0$, as the function $t(x)$ given by (\ref{E:t of x}) is multivalued, its inverse function $x(t)$ is periodic of period $p=\frac{\pi i}{\sqrt{\epsilon}}$. Hence, the disk of radius $r$ is sent to the exterior of a sequence of deformed circles (of initial radius $r^{-1}$ for $\epsilon=0$) repeated with period $p$. To cover the disk, we take two strips ($\Gamma^{\hat{\epsilon}}_D$ and $\Gamma^{\hat{\epsilon}}_U$, see Figure \ref{fig:Art2 10}) in the direction of $p$ of width larger than $\frac{p}{2}$, such that their union is a strip (with a hole) of width $w$, $p<w<2p$, containing $\frac{\pm \pi i}{2 \sqrt{\epsilon}}$. The singular points in the $t$-variable are located at infinity in the direction perpendicular to the line of holes. The intersection of the two domains $\Gamma^{\hat{\epsilon}}_D$ and $\Gamma^{\hat{\epsilon}}_U$ consists of three connected sets: $\Gamma^{\hat{\epsilon}}_L$ and $\Gamma^{\hat{\epsilon}}_R$ linking a part of the boundary to a singular point, and $\Gamma^{\hat{\epsilon}}_C$ linking the two singular points (coming from the periodicity).

\begin{figure}[h!]
\begin{center}
{\psfrag{B}{$\Gamma^{\hat{\epsilon}}_D$}
\psfrag{C}{$\Gamma^{\hat{\epsilon}}_U$}
\psfrag{E}{$\Gamma^{\hat{\epsilon}}_L$}
\psfrag{F}{$\Gamma^{\hat{\epsilon}}_R$}
\psfrag{G}{$\Gamma^{\hat{\epsilon}}_C$}
\includegraphics[width=6cm]{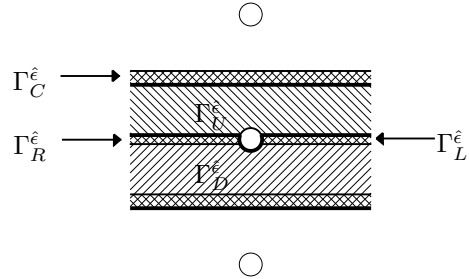}}
    \caption{Sectorial domains in the $t$-variable when $\sqrt{\epsilon} \in \mathbb{R}^*$.}
    \label{fig:Art2 10}
\end{center}
\end{figure}

For most values of $\hat{\epsilon} \in S$, the line of holes is slanted and we need to slant the strips. If we take pure slanted strips as in Figure \ref{fig:Art2 11}, we get domains that do not converge when $\hat{\epsilon} \to 0$ to the sectorial domains at $\epsilon=0$ (Figure \ref{fig:Art2 9}). Hence, we take a part of the boundary horizontal on a length $\frac{c}{\sqrt{|\epsilon|}}$ for some fixed $c > 0$ independent of $\hat{\epsilon}$, as illustrated in Figure \ref{fig:Art2 12}.

\begin{figure}[h!]
\begin{center}
{\includegraphics[width=4cm]{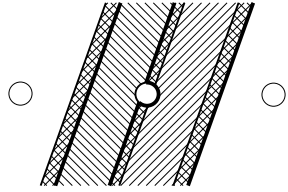}}
    \caption{Incorrectly slanted sectorial domains in the $t$-variable.}
    \label{fig:Art2 11}
\end{center}
\end{figure}

\begin{figure}[h!]
\begin{center}
{\psfrag{B}{\small{$\Gamma^{\hat{\epsilon}}_D$}}
\psfrag{C}{\small{$\Gamma^{\hat{\epsilon}}_U$}}
\psfrag{D}{\small{$\frac{c}{\sqrt{|\epsilon|}}$}}
\psfrag{E}{\small{$T=\frac{\pi i}{\sqrt{|\epsilon}|}$}}
\includegraphics[width=9cm]{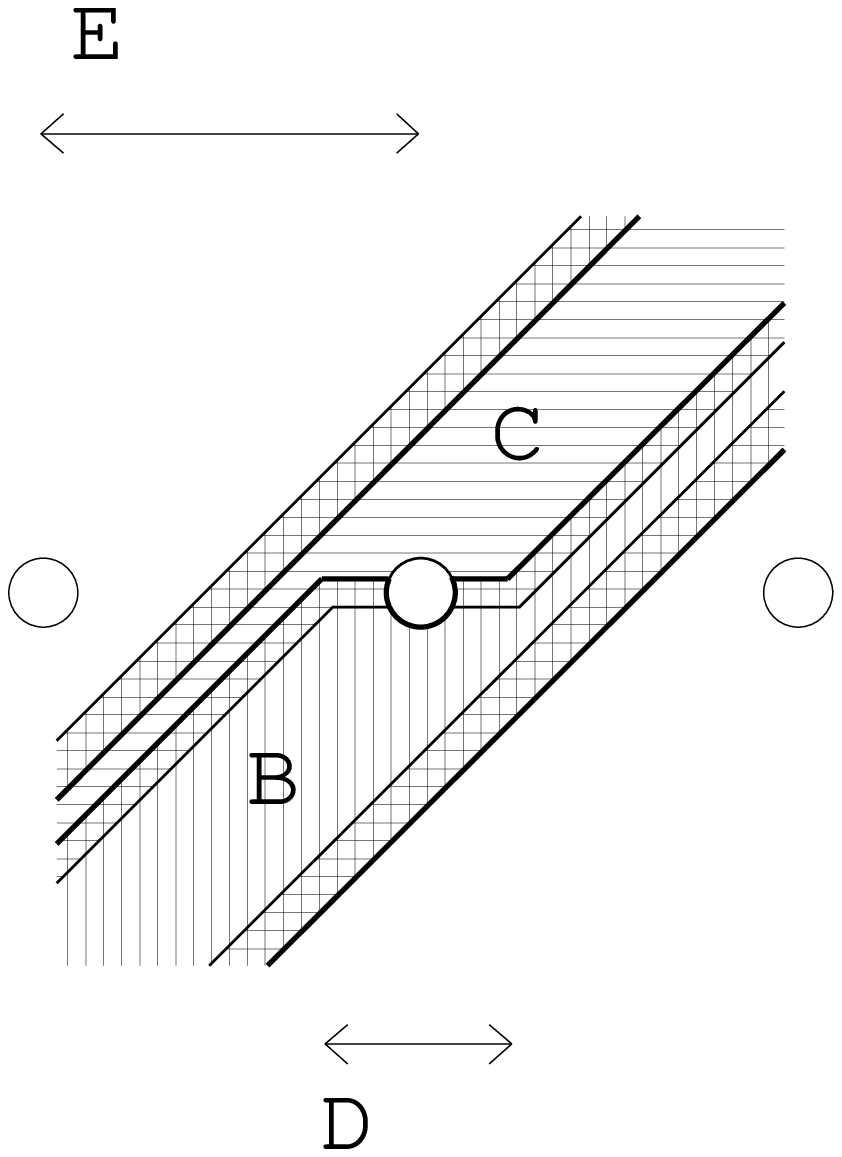}}
    \caption{Correctly slanted sectorial domains in the $t$-variable.}
    \label{fig:Art2 12}
\end{center}
\end{figure}

\begin{figure}[h!]
\begin{center}
{\psfrag{A}{\tiny{$\Re(\sqrt{\epsilon})$}}
\psfrag{D}{\tiny{$\Im(\sqrt{\epsilon})$}}
\psfrag{B}{\small{$\Gamma^{\hat{\epsilon}}_D$}}
\psfrag{C}{\small{$\Gamma^{\hat{\epsilon}}_U$}}
\psfrag{J}{\small{$\Gamma^{0}_D$}}
\psfrag{I}{\small{$\Gamma^{0}_U$}}
\psfrag{E}{\small{$\Gamma^{\hat{\epsilon}}_L$}}
\psfrag{F}{\small{$\Gamma^{\hat{\epsilon}}_R$}}
\psfrag{L}{\small{$\Gamma^{0}_L$}}
\psfrag{K}{\small{$\Gamma^{0}_R$}}
\psfrag{G}{\small{$\Gamma^{\hat{\epsilon}}_C$}}
\psfrag{H}{\tiny{$S$}}
\includegraphics[width=12cm]{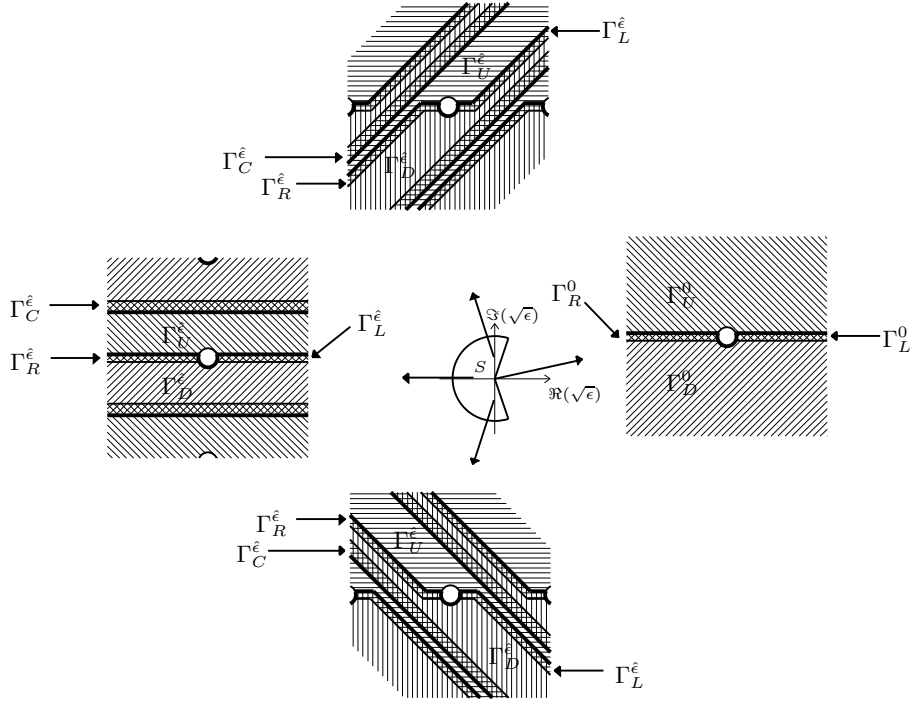}}
    \caption{Sectorial domains in the $t$-variable for some values of $\hat{\epsilon} \in S \cup \{0\}$, with $\gamma=\frac{\pi}{4}$.}
    \label{fig:Art2 7}
\end{center}
\end{figure}

Then, we define the sectorial domain $\Omega^{\hat{\epsilon}}_s$ in the $x$-variable as the one corresponding, via (\ref{E:t of x}), to the sectorial domain in the $t$-variable $\Gamma^{\hat{\epsilon}}_s$, $s \in \{U,D,L,R,C\}$ (Figures \ref{fig:Art2 2} and \ref{fig:Art2 6}). The points $\hat{x}_R$ and $\hat{x}_L$ are not in the sectorial domains $\Omega^{\hat{\epsilon}}_s$ but in their closure. The region $\Omega_L^{\hat{\epsilon}}$ (respectively $\Omega_R^{\hat{\epsilon}}$) has the singular point $\hat{x}_L$ (respectively $\hat{x}_R$) in its closure and $\Omega_C^{\hat{\epsilon}}$ has both (Figure \ref{fig:Art2 6}). Note that the point $x=0$ belongs to $\Omega_C^{\hat{\epsilon}}$.

\begin{figure}[h!]
\begin{center}
{\psfrag{A}{\tiny{$\Re(\sqrt{\epsilon})$}}
\psfrag{D}{\tiny{$\Im(\sqrt{\epsilon})$}}
\psfrag{B}{\small{$\Omega^{\hat{\epsilon}}_D$}}
\psfrag{C}{\small{$\Omega^{\hat{\epsilon}}_U$}}
\psfrag{G}{\small{$S$}}
\psfrag{E}{\small{$\hat{x}_L$}}
\psfrag{F}{\small{$\hat{x}_R$}}
\psfrag{I}{\small{$\Omega^{0}_D$}}
\psfrag{J}{\small{$\Omega^{0}_U$}}
\includegraphics[width=12cm]{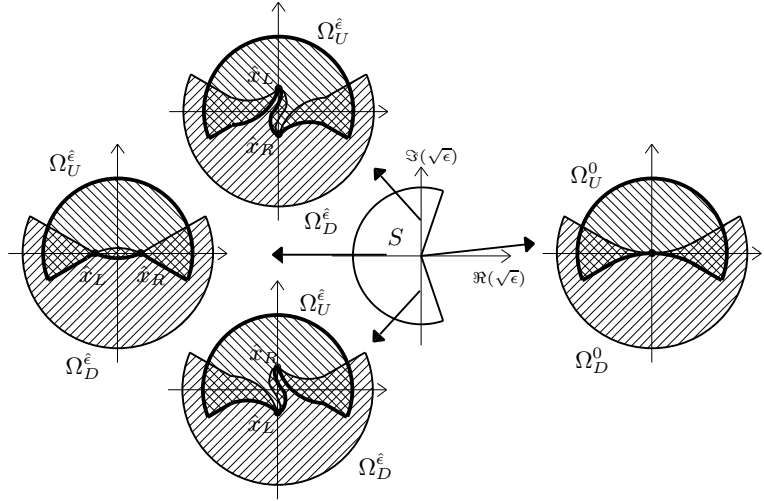}}
    \caption{Sectorial domains in the $x$-variable for some values of $\hat{\epsilon} \in S \cup \{0\}$.}
    \label{fig:Art2 2}
\end{center}
\end{figure}

\begin{figure}[h!]
\begin{center}
{\psfrag{H}{$\Omega_L^{\hat{\epsilon}}$}
\psfrag{D}{$\Omega_R^{\hat{\epsilon}}$}
\psfrag{B}{$\Omega_D^{\hat{\epsilon}}$}
\psfrag{C}{$\Omega_U^{\hat{\epsilon}}$}
\psfrag{G}{$\Omega_C^{\hat{\epsilon}}$}
\psfrag{F}{\small{$\hat{x}_R$}}
\psfrag{E}{\small{$\hat{x}_L$}}
\includegraphics[width=7cm]{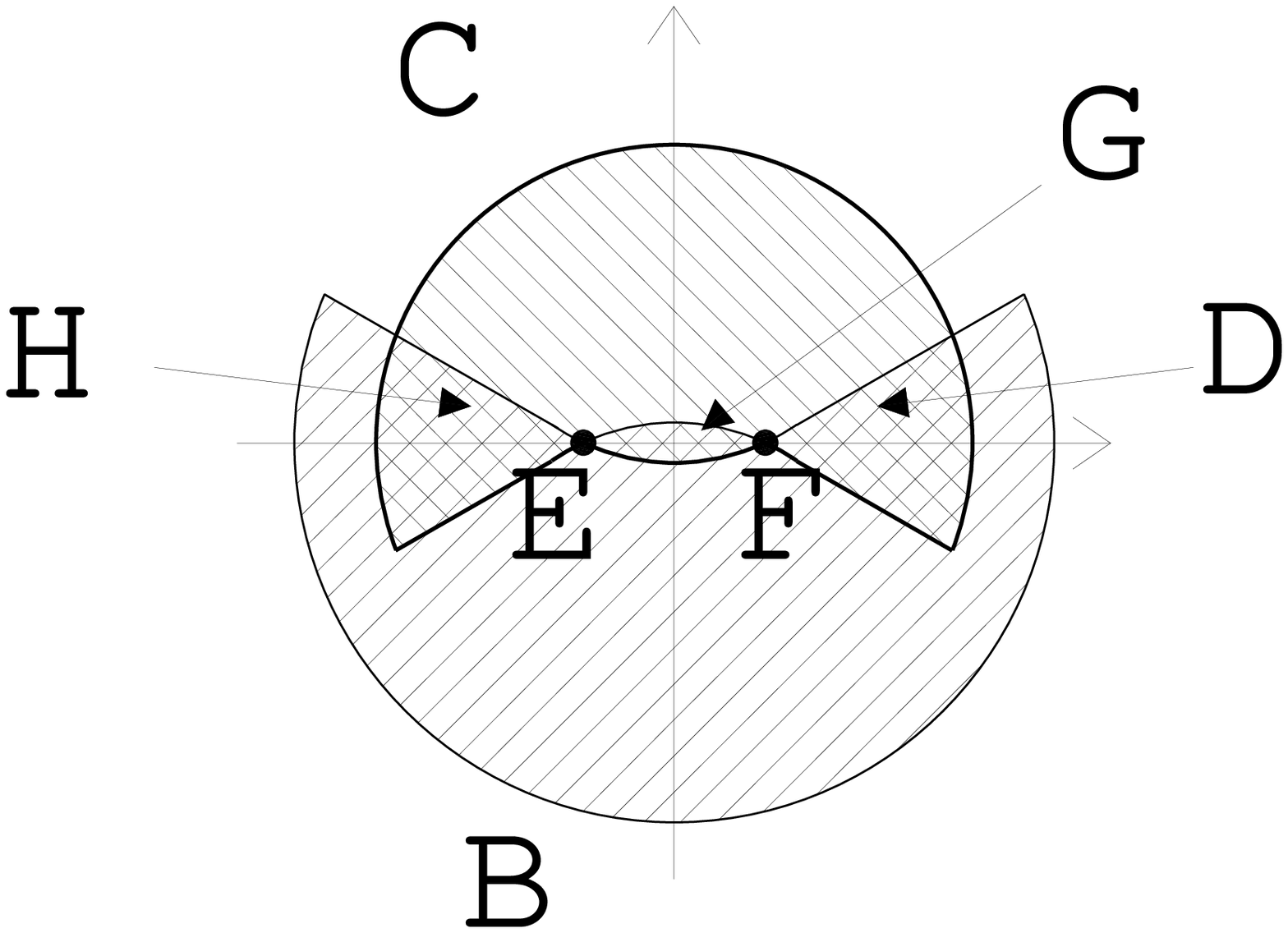}}
    \caption{The connected components of the intersection of the sectorial domains $\Omega_D^{\hat{\epsilon}}$ and $\Omega_U^{\hat{\epsilon}}$, case $\sqrt{\hat{\epsilon}}\in \mathbb{R}_{-}^*$.}
    \label{fig:Art2 6}
\end{center}
\end{figure}

In the $x$-variable, the difference between $\Omega^{\hat{\epsilon}}_s$ and $\Omega^{0}_s$ ($s=D,U$) is mainly located inside a disk of radius $c'\sqrt{|\epsilon|}$ (Figure \ref{fig:Art2 14}), due to the non-horizontal part of the boundary of the sectorial domains in the $t$-variable. Indeed, a slanted half-line going to infinity in the $t$-variable corresponds to logarithmic spirals approaching a singular point in the $x$-variable. Quantitative details and proofs can be found in \cite{cRlT}. The construction is possible for all $\hat{\epsilon} \in S$, provided the radius $\rho$ of $S$ is sufficiently small. Indeed, reducing $\rho$ amounts to increase the distance between the holes.
\begin{figure}[h!]
\begin{center}
{\includegraphics[width=6cm]{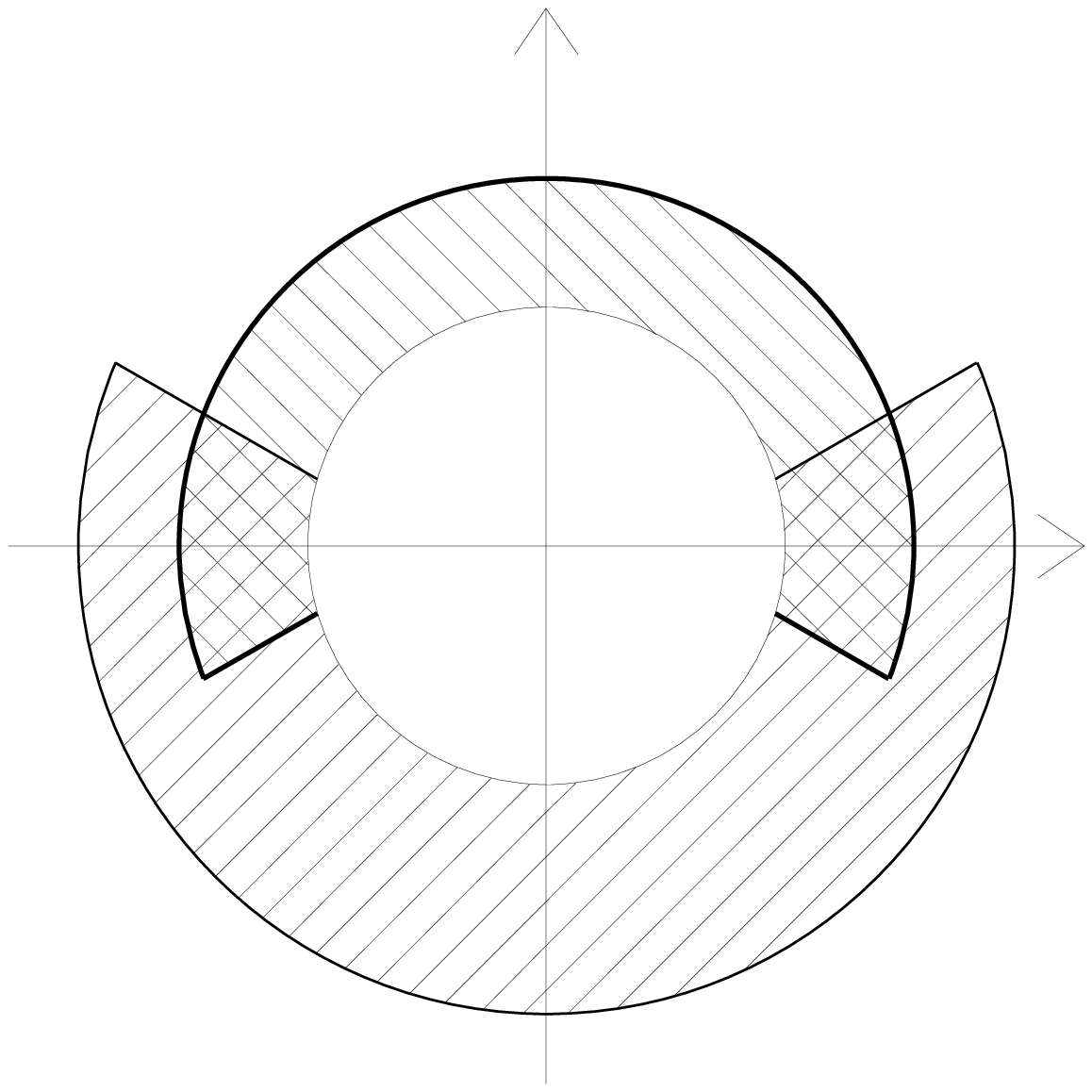}}
    \caption{Difference between the sectorial domains $\Omega^{\hat{\epsilon}}_s$ and $\Omega^{0}_s$ mainly located inside a small disk of radius $c'\sqrt{|\epsilon|}$.}
    \label{fig:Art2 14}
\end{center}
\end{figure}

The angle of the slope is chosen as follows. We take
\begin{equation}
\hat{\theta}=\frac{2\gamma}{\pi}(\pi-\arg(\sqrt{\hat{\epsilon}})),
\end{equation}
with $\gamma$ as chosen in Definition \ref{D:sector S}. Then, on the trajectories in the $x$-plane corresponding to $t=Ce^{i\hat{\theta}}+C'$ near the singular points, with $C' \in \mathbb{C}$ fixed for each trajectory and $C \in \mathbb{R}$, we have
\begin{equation}\label{E: f limit sur De 1}
\lim_{\begin{subarray}{lll}x(t) \to \hat{x}_l\\ t=Ce^{i\hat{\theta}}+C' \end{subarray}}(x-\hat{x}_R)^{\hat{\mu}_{j,R}-\hat{\mu}_{q,R}}(x-\hat{x}_L)^{\hat{\mu}_{j,L}-\hat{\mu}_{q,L}}=0,  \quad \mbox{for } \begin{cases} q>j, \quad \mbox{if } l=R, \\ q<j, \quad \mbox{if } l=L, \end{cases}
\end{equation}
(this is obtained from the fact that $\Re(e^{i\hat{\theta}}\sqrt{\hat{\epsilon}})<0$ and that $|\hat{\theta}|<\gamma(1+2\frac{\gamma}{\pi})$ with $\gamma$ satisfying (\ref{E:theta1 satisfy})). The limits (\ref{E: f limit sur De 1}) yield, with $f_j(\epsilon,x)$ given by (\ref{E:matrix of model}),
\begin{equation}\label{E: f limit sur De}
\lim_{\begin{subarray}{lll}x \to \hat{x}_l\\x \in \Omega_s^{\hat{\epsilon}}\end{subarray}} \frac{f_j(\epsilon,x)}{f_q(\epsilon,x)}=0, \quad \mbox{for }\begin{cases} q>j,\quad \mbox{if } l=R, \\ q<j, \quad \mbox{if } l=L. \end{cases}
\end{equation}
Note that we have a similar behavior when $\epsilon=0$:
\begin{equation}
\lim_{\begin{subarray}{lll}x \to 0\\x \in \Omega_l^{0}\end{subarray}} \frac{f_j(0,x)}{f_q(0,x)}=0, \quad \mbox{for }\begin{cases} q>j,\quad \mbox{if } l=R, \\ q<j, \quad \mbox{if } l=L. \end{cases}
\end{equation}

\subsection{Invariant manifolds in the projective space}\label{S:invar man e}

In this section, we find an invariant manifold $[y]_j=G_{j,s}(\hat{\epsilon},x)$ of the $j^{th}$ Riccati system (\ref{Riccati system}) that converges when $\hat{\epsilon} \to 0$ (in $S$) to the invariant manifold $[y]_j=G_{j,s}(0,x)$ (Notation \ref{N:def G0x}).

The Jacobian of the $j^{th}$ Riccati system at the singular point $(\hat{x}_l,0)$, $l=L,R$, has eigenvalues
\begin{equation}\label{E:eigenvalues}
2 \hat{x}_l; \, \lambda_1(\epsilon,\hat{x}_l)-\lambda_j(\epsilon,\hat{x}_l);\, ...\,; \,\widehat{\left(\lambda_j(\epsilon,\hat{x}_l)-\lambda_j(\epsilon,\hat{x}_l)\right)};\,...\,;\,\lambda_n(\epsilon,\hat{x}_l)-\lambda_j(\epsilon,\hat{x}_l).
\end{equation}
For $q \ne j$, the quotient of the eigenvalue in $-\frac{(y)_q}{(y)_j}$ (see Notation \ref{N:projective not}) over the one in $x$ gives $\hat{\mu}_{q,l}-\hat{\mu}_{j,l}$, with $\hat{\mu}_{j,l}$ given by (\ref{E:mu j l k1}).

\begin{definition}\label{D:resonant values}
We define the \emph{resonance values} of $\hat{\epsilon}$ as those for which $\hat{\mu}_{q,l}-\hat{\mu}_{j,l} \in \mathbb{N}^*$ for $q \ne j$, $l=L,R$. These are true resonances of the nonlinear Riccati system: they are exactly the values for which there is an obstruction to eliminate the terms $(y)_j(x-\hat{x}_l)^m \frac{\partial}{\partial (y)_q}$ in (\ref{E:prenormal system k1}) when localizing the system at $x=\hat{x}_l$. The parameter $\hat{\epsilon}$ has been taken inside a sector which avoids half of these resonances.
\end{definition}

\begin{remark}\label{R:other half resonance}
All resonance values of the unfolding parameter $\epsilon$ can be integrated in a continuous study: the consideration of half of them on the sector $S$ is sufficient since the change of parameter $\hat{\varepsilon}=\hat{\epsilon} e^{-2 \pi i}$, under which the unfolded systems are invariant, gives the new parameter $\hat{\varepsilon}$ in a sector including the other half of the resonance values.
\end{remark}

When $\hat{\epsilon} \in S$, the eigenvalues of the Jacobian, listed in (\ref{E:eigenvalues}), are separated in two distinct groups by a real line passing through the origin. In the proof of Theorem \ref{T:existence G(e,x)}, we will see that this separation of eigenvalues gives, locally, the existence of invariant manifolds that are tangent to the invariant subspaces of the linearization operator of the vector field at the singular points $(\hat{x}_l,0)$. We will need the following proposition to extend these local invariant manifolds.

\begin{proposition}\label{P:sous le cone}
For $\hat{\epsilon} \in S$, let us define the region
\begin{equation}\label{E:region Ve}
\mathcal{V}^j_{\hat{\epsilon}}=\mathcal{V}_{\hat{\epsilon},+}^j \cap \mathcal{V}_{\hat{\epsilon},-}^j,
\end{equation}
with
\begin{equation}
\mathcal{V}_{\hat{\epsilon},\pm}^j=\left\{(x,[y]_j) \in \mathbb{C} \times \mathbb{CP}^{n-1}:\left|\frac{(y)_i}{(y)_j}\right|\leq |x ± \sqrt{\hat{\epsilon}}|, \, i \in \{1,2,...,n \} \backslash \{j\} \right\}.
\end{equation}
The boundary of $\mathcal{V}_{\hat{\epsilon},\pm}^j$ is  $\bigcup_{\begin{subarray}{lll}i=1\\i \ne j\end{subarray}}^{n}\mathcal{V}_{\hat{\epsilon},±,i}^j$, with, for $i \ne j$,
\begin{equation}\label{E:bound vei}
\mathcal{V}_{\hat{\epsilon},±,i}^j=\{(x,[y]_j) \in \mathbb{C} \times \mathbb{CP}^{n-1}:\left|\frac{(y)_i}{(y)_j}\right|=|x ± \sqrt{\hat{\epsilon}}|, \, \left|\frac{(y)_k}{(y)_j}\right| \leq |x ± \sqrt{\hat{\epsilon}}| \mbox{ if }k \ne i,j\}.
\end{equation}
We can take the radius of $S$ sufficiently small so that
\begin{equation}\label{E:cond flow entre sort}
\left|\frac{d}{dt}|x\pm \sqrt{\hat{\epsilon}} |^2  \right|<\left|\frac{d}{dt}\left|\frac{(y)_i}{(y)_j}\right|^2  \right|,  \quad \mbox{\rm{for} }\begin{cases}\begin{array}{lll}(x,[y]_j) \in \mathcal{V}_{\hat{\epsilon},±,i}^j, \\ x \in \Omega^{\hat{\epsilon}}_s, \quad &s=D,U, \\ \hat{\epsilon} \in S,\\ i,j \in \{1,...,n\}, \quad &i \ne j.\end{array}\end{cases}
\end{equation}
\end{proposition}

\begin{proof}
Similarly to the proof of Proposition \ref{P:graph confined e non zero}, we consider (\ref{E:Riccati system expli}) and we have, either with the upper or the lower sign,
\begin{equation}
\left|\frac{1}{2}\frac{d}{dt}|x \pm \sqrt{\hat{\epsilon}}|^2  \right| \leq |x \pm \sqrt{\hat{\epsilon}}|^2|x \mp \sqrt{\hat{\epsilon}}|.
\end{equation}
On $\mathcal{V}_{\hat{\epsilon},±,i}^j$, we have
\begin{equation}
\frac{1}{2|x\pm \sqrt{\hat{\epsilon}}|^2}\left|\frac{d}{dt}\left|\frac{(y)_i}{(y)_j}\right|^2  \right| \geq  |\Re( \lambda_{i,0}(\epsilon)-\lambda_{j,0}(\epsilon))|-v^{±}_{ij}(\hat{\epsilon},x),
\end{equation}
with
\begin{equation}
\begin{array}{lll}
v^{±}_{ij}(\hat{\epsilon},x)=|\lambda_{i,1}(\epsilon)-\lambda_{j,1}(\epsilon)||x|+|x \mp \sqrt{\hat{\epsilon}} | |(R(\epsilon,x))_{ij}|\\ \qquad \qquad +|x² -\epsilon| \left(|(R(\epsilon,x))_{jj}|+ \sum_{\begin{subarray}{lll}k=1\\k \ne j \end{subarray}}^n (|(R(\epsilon,x))_{ik}|+|x \pm \sqrt{\hat{\epsilon}}||(R(\epsilon,x))_{jk}|)  \right).
\end{array}
\end{equation}
Let us take $\alpha$ such that
\begin{equation}
\alpha \leq \eta |\Re( \lambda_{i,0}(0)-\lambda_{j,0}(0))|, \quad \forall i \ne j,
\end{equation}
with $\eta$ as chosen in Proposition \ref{P:graph confined e non zero}. We restrict the radius of $S$ to $\rho>0$ such that
\begin{equation}
\big| |\Re( \lambda_{i,0}(\epsilon)-\lambda_{j,0}(\epsilon))|- |\Re( \lambda_{i,0}(0)-\lambda_{j,0}(0))| \big| < \frac{\alpha}{2}
\end{equation}
and such that
\begin{equation}
\big| v^{±}_{ij}(\hat{\epsilon},x)+|x \mp \sqrt{\hat{\epsilon}}|- |x|- v_{ij}(0,x)\big| < \frac{\alpha}{2}, \quad \forall i \ne j,
\end{equation}
implying
\begin{equation}
v^{±}_{ij}(\hat{\epsilon},x)+|x \mp \sqrt{\hat{\epsilon}}|<|\Re( \lambda_{i,0}(\epsilon)-\lambda_{j,0}(\epsilon))|, \quad \forall \hat{\epsilon} \in S, \quad \forall i \ne j.
\end{equation}
This yields (\ref{E:cond flow entre sort}).
\end{proof}

Using Proposition \ref{P:sous le cone}, we now define the graph $x \mapsto G_{j,s}(\hat{\epsilon},x)$ as consisting of the union of all solutions, parametrized by curves in complex time of the $j^{th}$ Riccati system, that are confined inside the region $\mathcal{V}_{\hat{\epsilon}}^j$ when restricted to the sectors $\Omega_s^{\hat{\epsilon}}$:

\begin{theorem}\label{T:existence G(e,x)}
In the $j^{th}$ Riccati system, there exists, for $\hat{\epsilon}\in S$, a unique one-dimensional invariant manifold $[y]_j=\hat{G}_{j,s}(x)=G_{j,s}(\hat{\epsilon},x)$, given as the graph of an analytic function $\hat{G}_{j,s}(x)$ with a continuous extension at $(\hat{x}_l,0)$, $l=L,R$,  passing through the two singular points $(x,[y]_j)=(\hat{x}_l,0)$, $l=L,R$, and located inside the region $\mathcal{V}_{\hat{\epsilon}}^j$ over the sector $\Omega_s^{\hat{\epsilon}}$. $G_{j,s}(\hat{\epsilon},x)$ depends analytically on $(\hat{\epsilon},x)$ for $\hat{\epsilon} \in S$ and $x \in \Omega_s^{\hat{\epsilon}}$.
\end{theorem}

\begin{proof}
We always take $x$ inside the sectorial domain $\Omega_s^{\hat{\epsilon}}$ and we omit the lower index $s$ within the proof : we write simply $G_{j}(\hat{\epsilon},x)$.

Let us take the first Riccati system and fix $\epsilon_0 \in S$. The choice of $S$ allows to separate, by a real line passing through the origin, the eigenvalue $2 \hat{x}_R$ from the other eigenvalues at $(\hat{x}_R,0)$ given by (\ref{E:eigenvalues}). From the Hadamard-Perron theorem for holomorphic flows (see \cite{yIsY} p.~106), there exist holomorphic invariant manifolds $\mathcal{W}_{\hat{x}_R,1}^+$ and $\mathcal{W}_{\hat{x}_R,1}^-$ tangent to the invariant subspaces of the linearization operator of the vector field at $(\hat{x}_R,0)$. We denote by $[y]_1=G_{1}(\epsilon_0,x)$ the unique one-dimensional invariant manifold $\mathcal{W}_{\hat{x}_R,1}^+$. Near $x=\hat{x}_R$, it is the unique invariant manifold contained inside the region $\mathcal{V}_{\epsilon_0}^j$ (defined by (\ref{E:region Ve})) and its extension cannot escape from $\mathcal{V}_{\epsilon_0}^j$, by Proposition \ref{P:sous le cone}.

Similarly, in the $n^{th}$ Riccati system, we take $[y]_n=G_{n}(\epsilon_0,x)$ as the extension of the unique holomorphic one-dimensional invariant manifold $\mathcal{W}_{\hat{x}_L,n}^-$ passing through $(\hat{x}_L,0)$.

Now, let us take the $j^{th}$ Riccati system, with $1<j<n$. Around $x=\hat{x}_R$ (respectively $x=\hat{x}_L$), we have two invariant manifolds $\mathcal{W}_{\hat{x}_R,j}^+$ and $\mathcal{W}_{\hat{x}_R,j}^-$ of dimension $j$ and $n-j$ (respectively $\mathcal{W}_{\hat{x}_L,j}^+$ and $\mathcal{W}_{\hat{x}_L,j}^-$ of dimension $j-1$ and $n-j+1$) tangent to the corresponding invariant subspaces of the linearization operator of the vector field. We analytically extend the invariant manifold $\mathcal{W}_{\hat{x}_R,j}^+$ tangent to $(x,\frac{(y)_1}{(y)_j},...,\frac{(y)_{j-1}}{(y)_j})$ at $(\hat{x}_R,0)$ towards the singular point $x=\hat{x}_L$. Proposition \ref{P:sous le cone} implies that any solution (with complex time) of this extended invariant manifold cannot exit $\mathcal{V}_{\epsilon_0}^j$ by the part of its boundary consisting of the $\mathcal{V}_{\epsilon_0,±,i}^j$ for $i \geq j+1$. Near $x=\hat{x}_L$, the extension of $\mathcal{W}_{\hat{x}_R,j}^+$ must then intersect the invariant manifold $\mathcal{W}_{\hat{x}_L,j}^-$, which is tangent to $(x,\frac{(y)_{j+1}}{(y)_j},...,\frac{(y)_n}{(y)_j})$. From the form of the manifolds as graphs, the intersection is transversal, and hence a one-dimensional invariant manifold denoted $[y]_j=G_{j}(\epsilon_0,x)$ passing through the singular points and nonsingular outside the singular points.

In each Riccati system, we thus have one-dimensional invariant manifolds $[y]_j=G_{j}(\epsilon_0,x)$ confined inside $\mathcal{V}_{\epsilon_0}^j$. Near $\epsilon_0 \ne 0$, $\mathcal{W}_{\hat{x}_l,j}^\pm$ depends analytically on $\epsilon$, implying that the unique solution $[y]_j=G_{j}(\hat{\epsilon},x)$ is analytic in $\hat{\epsilon}$ for $\hat{\epsilon} \in S$.
\end{proof}

\begin{remark}\label{R:coincide centre}
The invariant manifolds $[y]_1=G_{1,s}(\hat{\epsilon},x)$ and $[y]_n=G_{n,s}(\hat{\epsilon},x)$ are uniform respectively near $\hat{x}_R$ and near $\hat{x}_L$, whereas $[y]_j=G_{j,s}(\hat{\epsilon},x)$ is ramified at the two singular points. More precisely, $G_{j,U}(\hat{\epsilon},x)=G_{j,D}(\hat{\epsilon},x)$ over $\Omega_C^{\hat{\epsilon}}$ (Figure \ref{fig:Art2 6}) for $j=1,2,...,n$, $G_{1,U}(\hat{\epsilon},x)=G_{1,D}(\hat{\epsilon},x)$ over $\Omega_R^{\hat{\epsilon}}$ and $G_{n,U}(\hat{\epsilon},x)=G_{n,D}(\hat{\epsilon},x)$ over $\Omega_L^{\hat{\epsilon}}$.
\end{remark}

Solutions in the invariant manifold $[y]_j=G_{j,s}(\hat{\epsilon},x)$ behave differently from the other solutions of the $j^{th}$ Riccati system, since they are the only ones that are bounded when $x \to \hat{x}_R$ and $x \to \hat{x}_L$ over $\Omega_s^{\hat{\epsilon}}$. The fact that an invariant manifold $[y]_j=G_{j,s}(\hat{\epsilon},x)$ is bounded over the region $\mathcal{V}_{\hat{\epsilon}}^j$ leads to its uniform convergence on compact sets of $\Omega_s^0$:

\begin{theorem}\label{T:uniform convergence G(e,x)}
The invariant manifold $[y]_j=G_{j,s}(\hat{\epsilon},x)$ converges uniformly on compact subsets of $\Omega_s^0$, when $\hat{\epsilon} \to 0$, $\hat{\epsilon} \in S$, to the invariant manifold at $\epsilon=0$ $[y]_j=G_{j,s}(0,x)$ (see Notation \ref{N:def G0x}), for $s=D,U$.
\end{theorem}

\begin{proof}
Let us take a simply connected compact subset of $\Omega_s^0$. For $|\epsilon|$ sufficiently small, it does not contain neither $\hat{x}_R$ nor $\hat{x}_L$, nor the spiraling part of $\Omega_s^{\hat{\epsilon}}$. Proposition \ref{P:sous le cone} implies that the invariant manifold $[y]_j=G_{j,s}(\hat{\epsilon},x)$ satisfies
\begin{equation}\label{E:g borne}
|(G_{j,s}(\hat{\epsilon},x))_i|<\min\{|x-\hat{x}_R|,|x-\hat{x}_L|\}, \quad \mbox{with }\begin{cases} x \in \Omega_s^{\hat{\epsilon}}, \quad s=D,U,\\
\hat{\epsilon} \in S, \\i=1,2,...,n-1.\end{cases}
\end{equation}
This implies the desired convergence to a bounded solution of the system for $\epsilon=0$ that can only be $[y]_j=G_{j,s}(x,0)$.
\end{proof}

\subsection{Basis of the linear system (\ref{E:prenormal system k1})}\label{S:basis of linear syst}
In this section, we establish the correspondence between the invariant manifold $[y]_j=G_{j,s}(\hat{\epsilon},x)$ of the $j^{th}$ Riccati system (\ref{Riccati system}) and multiples (by a complex constant) of a particular solution of the linear system (\ref{E:prenormal system k1}). We show that these $n$ particular solutions form a basis of solutions of the linear system which is valid for all values of $\hat{\epsilon} \in S$ and $x \in \Omega_s^{\hat{\epsilon}}$.

\begin{notation}\label{N: def F_s}
Let $F_D(\hat{\epsilon},x)$ be the restriction to $\Omega_D^{\hat{\epsilon}}$ of the fundamental matrix of solutions of the model system $F(\epsilon,x)$ (given by (\ref{E:matrix of model})), and let $F_U(\hat{\epsilon},x)$ be its analytic continuation to $\Omega_U^{\hat{\epsilon}}$, passing through $\Omega_R^{\hat{\epsilon}}$.
\end{notation}

\begin{remark}\label{R:ramif F}
The matrix $F_s(\hat{\epsilon},x)$ is uniform over $\Omega_s^{\hat{\epsilon}}$, $s=D,U$, and according to Notation \ref{N: def F_s}, we have
\begin{equation}\label{E:F ramif}
F_U(\hat{\epsilon} ,x)=\begin{cases} \begin{array}{lll}F_D(\hat{\epsilon},x), \quad  &\mbox{\rm{on} }   \Omega_R^{\hat{\epsilon}},  \\
                  F_D(\hat{\epsilon},x) e^{2 \pi i \Lambda_1(\epsilon)}, \quad  &\mbox{\rm{on} }    \Omega_L^{\hat{\epsilon}},\\
F_D(\hat{\epsilon},x) \hat{D}_R^{-1}, \quad  &\mbox{\rm{on} } \Omega_C^{\hat{\epsilon}},\end{array}\end{cases}
\end{equation}
with $\hat{D}_R$ given by (\ref{E:Dx1}) and $\Lambda_1(\epsilon)$ by (\ref{E:D def}), satisfying (\ref{E:prod DrDl}).
\end{remark}

\begin{theorem}\label{T:fundamental matrix e}
Let $s=D,U$. There exists a fundamental matrix of solutions of (\ref{E:prenormal system k1}) that can be written as
\begin{equation}\label{E:equation W,H,F}
W_s(\hat{\epsilon},x)=H_s(\hat{\epsilon},x) F_s(\hat{\epsilon},x), \quad (\hat{\epsilon},x) \in S \times \Omega_s^{\hat{\epsilon}},
\end{equation}
with $H_s$ analytic in $(\hat{\epsilon},x)$ on $S \times \Omega_s^{\hat{\epsilon}}$, satisfying
\begin{equation}\label{E:diff Hs0bartilde}
|H_{s}(\bar{\epsilon} ,0)-H_{s}(\tilde{\epsilon} ,0)| \leq c |\bar{\epsilon}|, \quad \mbox{\rm{for some} } c \in \mathbb{R}_+,  \quad \bar{\epsilon}, \,\tilde{\epsilon}=\bar{\epsilon} e^{2 \pi i} \in S_\cap \cup \{0\},
\end{equation}
\begin{equation}\label{E:Hs0bounded}
|H_{s}(\hat{\epsilon},0)| \mbox{ \rm{and} } |H_{s}(\hat{\epsilon},0)^{-1}| \mbox{\rm{ are bounded}},\quad \hat{\epsilon} \in S_\cap,
\end{equation}
and
\begin{equation}\label{E:limit Hs}
\lim_{\begin{subarray}{lll}x \to \hat{x}_l\\x \in \Omega_s^{\hat{\epsilon}} \end{subarray}}H_s(\hat{\epsilon},x)=\hat{\mathcal{K}}_{l}, \quad \hat{\epsilon} \in S, \, l=L,R,
\end{equation}
where $\hat{\mathcal{K}}_{l}$ is an invertible diagonal matrix depending analytically on $\hat{\epsilon} \in S$ with a nonsingular limit at $\epsilon=0$ (independent of $s$).
\end{theorem}

\begin{proof}
The proof is valid for $s=D$ or $s=U$. For our needs, we write $[y]_j=G_{j,s}(\hat{\epsilon},x)$ on $\Omega_s^{\hat{\epsilon}}$ as
\begin{equation}\label{E:ecriture gjj}
\begin{cases}
\begin{array}{lll}
\frac{-(y)_k}{(y)_j}&=g_{kj,s}(\hat{\epsilon},x),\\
-1&=g_{jj,s}(\hat{\epsilon},x).
\end{array}
\end{cases}
\end{equation}
With (\ref{E:prenormal system k1}), we can write
\begin{equation}
(y')_j=\frac{\lambda_j(\epsilon,x)}{x^2 - \epsilon} (y)_j+ \sum_{k=1}^{n}(R(\epsilon,x))_{jk} (y)_k.
\end{equation}
Dividing by $(y)_j$, the known solutions of the $j^{th}$ Riccati system appear in the right hand side:
\begin{equation}\label{E:derivee y_j}
\frac{(y')_j}{(y)_j}=-\frac{\lambda_j(\epsilon,x)}{x^2 - \epsilon}g_{jj,s}(\hat{\epsilon},x) - \sum_{k=1}^{n}(R(\epsilon,x))_{jk} g_{kj,s}(\hat{\epsilon},x).
\end{equation}
The integration of equation (\ref{E:derivee y_j}) allows to recover $(y)_j$ and relation (\ref{E:ecriture gjj}) leads to the other $(y)_k$, thus yielding a solution $w_{j,s}(\hat{\epsilon},x)$ of the linear system (\ref{E:prenormal system k1}) (and all its multiples by a complex constant) that can be written as
\begin{equation}\label{E:w_p}
w_{j,s}(\hat{\epsilon},x)=f_{j,s}(\hat{\epsilon},x)h_{j,s}(\hat{\epsilon},x),
\end{equation}
with $f_{j,s}(\hat{\epsilon},x)$ the $j^{th}$ diagonal element of $F_s(\hat{\epsilon},x)$ (see Notation \ref{N: def F_s}), and with
\begin{equation}\label{E:h_j}
(h_{j,s}(\hat{\epsilon},x))_k=-
e^{-\int_{0}^x\sum_{p=1}^n\left(R(\epsilon,x)\right)_{jp}g_{pj,s}(\hat{\epsilon},x)dx} g_{kj,s}(\hat{\epsilon},x),
\end{equation}
where the integration path is taken inside $\Omega_s^{\hat{\epsilon}}$. Such a path can be found in the $t$-variable (see Section \ref{S:sectors x}) since $t(0) \in \Gamma_C^{\hat{\epsilon}}$. With the $n$ Riccati systems, we obtain in this way $n$ solutions $w_{j,s}(\hat{\epsilon},x)$ of the linear system (\ref{E:prenormal system k1}) defined for $\hat{\epsilon} \in S$ and $x \in \Omega_s^{\hat{\epsilon}}$.
We take
\begin{equation}
W_s(\hat{\epsilon},x)=[w_{1,s}(\hat{\epsilon},x) \, ... w_{n,s}(\hat{\epsilon},x)]
\end{equation}
and
\begin{equation}
H_s(\hat{\epsilon},x)=[h_{1,s}(\hat{\epsilon},x) \, ... h_{n,s}(\hat{\epsilon},x)]
\end{equation}
to obtain (\ref{E:equation W,H,F}) from (\ref{E:w_p}). The limit (\ref{E:limit Hs}) follows from
\begin{equation}\label{E:limit G}
\lim_{\begin{subarray}{lll}x \to \hat{x}_l\\x \in \Omega_s^{\hat{\epsilon}} \end{subarray}}g_{kj,s}(\hat{\epsilon},x)=0, \quad k \ne j, \quad l=L,R,
\end{equation}
and
\begin{equation}\label{E:limit Hxl Kl}
(\hat{\mathcal{K}}_{l})_{jj}=\lim_{x \to \hat{x}_{l}}(h_{j,s}(\hat{\epsilon},x))_j=
e^{-\int_{0}^{\hat{x}_l}\sum_{p=1}^n\left(R(\epsilon,x)\right)_{jp}g_{pj,s}(\hat{\epsilon},x)dx},
\end{equation}
which is independent of $s$ since the integration path in (\ref{E:limit Hxl Kl}) may be taken inside $\Omega_C^{\hat{\epsilon}}$ (see Remark \ref{R:coincide centre}).

The solutions $w_{1,s}(\hat{\epsilon},x),...,w_{n,s}(\hat{\epsilon},x)$ form a basis of solutions since the columns of $F_s(\hat{\epsilon},x)$ are linearly independent and since $\hat{\mathcal{K}}_{l}$ in (\ref{E:limit Hs}) is invertible.

The property (\ref{E:Hs0bounded}) comes from (\ref{E:g borne}) and (\ref{E:limit Hxl Kl}). Let us now prove (\ref{E:diff Hs0bartilde}). From its definition, $H_{s}(\hat{\epsilon} ,0)F_s(\hat{\epsilon},0)$ is a solution of (\ref{E:prenormal system k1}) at $x=0$, so
\begin{equation}\label{E:automo D 6 pp}
\begin{array}{lll}
\Lambda(\epsilon,0)H_{s}(\hat{\epsilon} ,0)-H_{s}(\hat{\epsilon} ,0)\Lambda(\epsilon,0) =- \epsilon \left(H'_{s}(\hat{\epsilon} ,0)-R(\epsilon,0)H_{s}(\hat{\epsilon} ,0)\right).
\end{array}
\end{equation}
With $\bar{\epsilon}$ and $\tilde{\epsilon}$ in $S_\cap$ (see Notation \ref{N:autointersec}), we thus have
\begin{equation}\label{E:automo D 6}
\begin{array}{lll}
\Lambda(\epsilon,0)(H_{s}(\bar{\epsilon} ,0)-H_{s}(\tilde{\epsilon} ,0))-(H_{s}(\bar{\epsilon} ,0)-H_{s}(\tilde{\epsilon} ,0))\Lambda(\epsilon,0) \\
\quad = -\epsilon \left(H'_{s}(\bar{\epsilon} ,0)-H'_{s}(\tilde{\epsilon} ,0)-R(\epsilon,0)(H_{s}(\bar{\epsilon} ,0)-H_{s}(\tilde{\epsilon} ,0))\right),
\end{array}
\end{equation}
yielding, for some $k \in \mathbb{R}_+$,
\begin{equation}\label{E:automo D 7}
|\left(H_{s}(\bar{\epsilon} ,0)-H_{s}(\tilde{\epsilon} ,0)\right)_{jq}| \leq k |\bar{\epsilon}|, \quad j \ne q, \, \bar{\epsilon} \in S_\cap\cup \{0\}, \, i=1,2,
\end{equation}
by the boundedness of $|H'_{s}(\bar{\epsilon} ,0)-H'_{s}(\tilde{\epsilon} ,0)|$, $|R(\epsilon,0)|$ and $|H_{s}(\bar{\epsilon} ,0)-H_{s}(\tilde{\epsilon} ,0)|$ over $S_\cap \cup \{0\}$ (recall that $\Lambda(\epsilon,0)$ has distinct eigenvalues for $\epsilon \in S \cup \{0\}$). Relation (\ref{E:diff Hs0bartilde}) comes from (\ref{E:automo D 7}) and from the fact that the diagonal entries of $H_{s}(\bar{\epsilon},0)-H_{s}(\tilde{\epsilon} ,0)$ are zeros (since $(H_{s}(\hat{\epsilon},0))_{jj}=1$).
\end{proof}

\begin{corollary}
Let $s=D,U$. There exists a transformation $y=H_s(\hat{\epsilon},x)z$, with $H_s(\hat{\epsilon},x)$ given by Theorem \ref{T:fundamental matrix e}, conjugating the system (\ref{E:prenormal system k1}) to its model (\ref{E:model system k1}) over $\Omega_s^{\hat{\epsilon}}$, for $\hat{\epsilon} \in S \cup \{ 0 \}$.
\end{corollary}

We have seen that the solutions in the invariant manifold $[y]_j=G_{j,s}(\hat{\epsilon},x)$ converge uniformly on compact sets of $\Omega_s^{0}$ when $\hat{\epsilon}\to 0$. This property remains for the corresponding solutions of the linear system:

\begin{corollary}[of Theorem \ref{T:uniform convergence G(e,x)}]\label{C:uniform convergence}
 When $\hat{\epsilon}\to 0$, the fundamental matrix $W_s(\hat{\epsilon},x)$ converges (uniformly on compact sets of $\Omega_s^0$) to the fundamental matrix $W_s(0,x)$ defined in (\ref{E:confluent solution}), $s=D,U$.
\end{corollary}
\begin{proof}
From (\ref{E:equation W,H,F}) and the convergence of $F(\hat{\epsilon},x)$ to $F(0,x)$, it suffices to prove the desired convergence of $H_s(\hat{\epsilon},x)$. This is immediate, since each column has an expression in terms of the solution $[y]_j=G_{j,s}(\hat{\epsilon},x)$ as in (\ref{E:h_j}), using the notation (\ref{E:ecriture gjj}).
\end{proof}

The bases $W_D(\hat{\epsilon},x)$ and $W_U(\hat{\epsilon},x)$ defined respectively on $\Omega_D^{\hat{\epsilon}}$ and $\Omega_U^{\hat{\epsilon}}$ will allow the calculation of the analytic invariants of the linear system.

\subsection{Definition of the unfolded Stokes matrices}\label{S:unfolded Stokes matrices}

In this section, we define the unfolded Stokes matrices by comparing the fundamental matrices of solutions $W_D(\hat{\epsilon},x)$ and $W_U(\hat{\epsilon},x)$ on the connected components of the intersection of $\Omega_D^{\hat{\epsilon}}$ and $\Omega_U^{\hat{\epsilon}}$ (Figure \ref{fig:Art2 6}).

\begin{theorem}\label{T:unfolded Stokes marices}
There exist matrices $\hat{C}_R$ and $\hat{C}_L$ such that
\begin{equation}\label{E: def C_Re}
H_D(\hat{\epsilon}, x)^{-1}H_U(\hat{\epsilon}, x)=\begin{cases}\begin{array}{lll}F_D(\hat{\epsilon},x) \hat{C}_R(F_D(\hat{\epsilon},x))^{-1}, &\quad  \mbox{\rm{ on} } \Omega_R^{\hat{\epsilon}},
\\ F_D(\hat{\epsilon},x) \hat{C}_L(F_D(\hat{\epsilon},x))^{-1}, &\quad  \mbox{\rm{ on} } \Omega_L^{\hat{\epsilon}},\\I, &\quad \mbox{\rm{ on} } \Omega_C^{\hat{\epsilon}}.\end{array}\end{cases}
\end{equation}
$\hat{C}_R$ and $\hat{C}_L$ are respectively an upper triangular and a lower triangular unipotent matrix. They depend analytically on $\hat{\epsilon} \in S$ and converge when $\hat{\epsilon} \to 0$ ($\hat{\epsilon} \in S$) to the Stokes matrices defined by (\ref{E:Stokes S2}).
\end{theorem}

\begin{proof}
As $W_D(\hat{\epsilon},x)$ and $W_U(\hat{\epsilon},x)$ are two fundamental matrices of solutions on the intersection of $\Omega_D^{\hat{\epsilon}}$ and $\Omega_U^{\hat{\epsilon}}$ (see Theorem \ref{T:fundamental matrix e}), there exist matrices expressing the fact that columns of $W_U(\hat{\epsilon},x)$ are linear combinations of columns of $W_D(\hat{\epsilon},x)$ on the intersection parts $\Omega_L^{\hat{\epsilon}}$, $\Omega_R^{\hat{\epsilon}}$ and $\Omega_C^{\hat{\epsilon}}$. With (\ref{E:F ramif}) and (\ref{E:equation W,H,F}), these relations become equivalent to
\begin{equation}
H_D(\hat{\epsilon}, x)^{-1}H_U(\hat{\epsilon}, x)=\begin{cases}\begin{array}{lll}F_D(\hat{\epsilon},x) \hat{C}_R(F_D(\hat{\epsilon},x))^{-1} &\quad  \mbox{\rm{ on} } \Omega_R^{\hat{\epsilon}}
\\ F_D(\hat{\epsilon},x) \hat{C}_L(F_D(\hat{\epsilon},x))^{-1}&\quad  \mbox{\rm{ on} } \Omega_L^{\hat{\epsilon}}\\F_D (\hat{\epsilon},x)\hat{C}_0(F_D(\hat{\epsilon},x))^{-1}  &\quad \mbox{\rm{ on} } \Omega_C^{\hat{\epsilon}}.\end{array}\end{cases}
\end{equation}
Then, taking the limit $x \to \hat{x}_L$ on $\Omega_L^{\hat{\epsilon}}$, $x \to \hat{x}_R$ on $\Omega_R^{\hat{\epsilon}}$ and both limits on $\Omega_C^{\hat{\epsilon}}$ leads, with (\ref{E: f limit sur De}) and (\ref{E:limit Hs}), to $\hat{C}_0 =I$ and to the unipotent triangular form of the matrices $\hat{C}_R$ and $\hat{C}_L$. Since $W_s(\hat{\epsilon},x)$ and $F_s(\hat{\epsilon},x)$ converge uniformly on compact sets of $\Omega_s^{0}$ (see Corollary \ref{T:uniform convergence G(e,x)} and Remark \ref{R:solutions fj}), so does $H_s(\hat{\epsilon},x)$ . Then, the matrices $\hat{C}_R$ and $\hat{C}_L$ must converge to the Stokes matrices when $\hat{\epsilon} \to 0$, $\hat{\epsilon} \in S$.
\end{proof}

\begin{definition}\label{D:unfolded Stokes matrices}
We call $\hat{C}_R$ and $\hat{C}_L$ (defined by (\ref{E: def C_Re})) the \emph{unfolded Stokes matrices} and $\{\hat{C}_R,\hat{C}_L\}$ an \emph{unfolded Stokes collection}.
\end{definition}

\begin{proposition}\label{P:fundamental matrix}
A fundamental matrix of solutions of (\ref{E:prenormal system k1}) that can be written as (\ref{E:equation W,H,F}), with $H_s(\hat{\epsilon},x)$ analytic on $S \times \Omega_s^{\hat{\epsilon}}$, satisfying (\ref{E:diff Hs0bartilde}), (\ref{E:Hs0bounded}) and with a limit when $x \to \hat{x}_l$, $x \in \Omega_s^{\hat{\epsilon}}$ that is bounded, invertible and independent of $s$, is unique up to right multiplication by any nonsingular diagonal matrix $\hat{K}$ depending analytically on $\hat{\epsilon} \in S$ with a nonsingular limit at $\epsilon=0$ and such that
\begin{equation}\label{E:condition Ktildebar}
|\bar{K}-\tilde{K}| \leq c |\bar{\epsilon}| \quad \mbox{\rm{over} } S_\cap, \, \mbox{\rm{for some} } c \in \mathbb{R}_+.
\end{equation}
\end{proposition}

\begin{proof}
Let us suppose that we have two fundamental matrices of solutions that can be written as $H_s(\hat{\epsilon},x)F_s(\hat{\epsilon},x)$ and $H^*_s(\hat{\epsilon},x)F_s(\hat{\epsilon},x)$ with properties listed in the proposition. Since we have two bases of solutions over $\Omega_C^{\hat{\epsilon}}$, there exists a matrix $\hat{K}$ such that
\begin{equation}
H^*_s(\hat{\epsilon},x)F_s(\hat{\epsilon},x)=H_s(\hat{\epsilon},x)F_s(\hat{\epsilon},x)\hat{K}, \quad x \in \Omega_C^{\hat{\epsilon}}.
\end{equation}
Since the limits when $x \to \hat{x}_l$, $l=L,R$, of $H_s(\hat{\epsilon},x)$ and of $H^*_s(\hat{\epsilon},x)$ are bounded and invertible, $\hat{K}$ must be a diagonal matrix. Then, we have
\begin{equation}
H_s(\hat{\epsilon},x)^{-1}H^*_s(\hat{\epsilon},x)=\hat{K}, \quad x \in \Omega_C^{\hat{\epsilon}},
\end{equation}
and in particular
\begin{equation}\label{E:K0H}
H_s(\hat{\epsilon},0)^{-1}H^*_s(\hat{\epsilon},0)=\hat{K}.
\end{equation}
From (\ref{E:K0H}), (\ref{E:diff Hs0bartilde}) and (\ref{E:Hs0bounded}), we obtain (\ref{E:condition Ktildebar}).
\end{proof}

As the uniqueness of $W_s(\hat{\epsilon},x)$ is ensured by the choice of a nonsingular diagonal matrix $\hat{K}$ having properties listed in Proposition \ref{P:fundamental matrix}, it is natural to adopt the following definition:

\begin{definition}\label{D:equiv}
Two unfolded Stokes collections, $\{\hat{C}_R,\hat{C}_L\}$ and $\{\hat{C}'_R,\hat{C}'_L\}$ (see Definition \ref{D:unfolded Stokes matrices}), are \emph{equivalent} if
\begin{equation}\label{E:equiv e}
\hat{C}'_l =\hat{K} \hat{C}_l \hat{K}^{-1}, \quad l=L,R,
\end{equation}
for some nonsingular diagonal matrix $\hat{K}$ depending analytically on $\hat{\epsilon} \in S$ with a nonsingular limit at $\epsilon=0$ and such that (\ref{E:condition Ktildebar}) is satisfied.
\end{definition}

Using results obtained from the study of the monodromy of the solutions, we will prove in Section \ref{S:analytic invariant} that these equivalence classes of unfolded Stokes collections constitute the analytic part of the complete system of invariants for the systems (\ref{E:prenormal system k1}).

\subsection{Unfolded Stokes matrices and monodromy in the linear system}\label{S:Unfolded Stokes matrices and monodromy}
In this section, we show how the unfolded Stokes matrices are linked to the monodromy operator acting on $W_s(\hat{\epsilon},x)$, how they give information on the existence of the bases of solutions composed of eigenvectors of the monodromy operator, and how they provide a meaning to the Stokes matrices at $\epsilon=0$.

To study the action of the monodromy operator, we consider the ramified domain
\begin{equation}
V^{\hat{\epsilon}}=\Omega_D^{\hat{\epsilon}} \cup \Omega_U^{\hat{\epsilon}},
\end{equation}
illustrated in Figure \ref{fig:Art2 5}, which could have a (non illustrated) spiraling part around $\hat{x}_R$ and $\hat{x}_L$.
\begin{figure}[h!]
\begin{center}
{\psfrag{F}{\small{$\hat{x}_R$}}
\psfrag{E}{\small{$\hat{x}_L$}}
\psfrag{G}{$V^{\hat{\epsilon}}$}
\includegraphics[width=5cm]{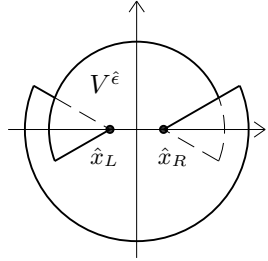}}
    \caption{Domain of $H(\hat{\epsilon} ,x)$, denoted $V^{\hat{\epsilon}}$, case $\sqrt{\hat{\epsilon}} \in \mathbb{R}^*_{-}$.}
    \label{fig:Art2 5}
\end{center}
\end{figure}

\begin{notation}\label{N:Fv}
Let $F_V(\hat{\epsilon},x)$ be the analytic continuation of $F_D(\hat{\epsilon},x)$ from $\Omega_D^{\hat{\epsilon}}$ to $V^{\hat{\epsilon}}$ (through $\Omega_C^{\hat{\epsilon}}$).
\end{notation}

The well chosen basis of solutions we consider on this domain is the analytic continuation of $W_D(\hat{\epsilon},x)$ from $\Omega_D^{\hat{\epsilon}}$ to $V^{\hat{\epsilon}}$, that we write as
\begin{equation}\label{E:basis union}
W_V(\hat{\epsilon},x)=[w_1(\hat{\epsilon},x) \, ... \, w_n(\hat{\epsilon},x)]=H(\hat{\epsilon},x)F_V(\hat{\epsilon},x),
\end{equation}
where
\begin{equation}\label{E:definition H(e,x)}
H(\hat{\epsilon} ,x)=   \begin{cases}H_D(\hat{\epsilon} ,x), \quad \mbox{on } \Omega_D^{\hat{\epsilon}}, \\
H_U(\hat{\epsilon} ,x), \quad  \mbox{on  } \Omega_U^{\hat{\epsilon}},
                \end{cases}
\end{equation}
which is well-defined because of (\ref{E: def C_Re}).

The fundamental group of $\mathbb{D}_r \backslash \{x_R,x_L \}$ based at a nonsingular point acts on a solution (valid at this base point) by giving its analytic continuation at the end of a loop. In this way we have monodromy operators around each singular point $x=x_l$. We can extend this action of the fundamental group to any function of the solutions. When the monodromy operator acts on a fundamental matrix of solutions $W$, its is represented by a matrix acting by right multiplication on $W$.


\begin{notation}\label{N:monodromy operator}
Let us take the fundamental group of $\mathbb{D}_r \backslash \{x_R,x_L \}$ based, independently of $\hat{\epsilon} \in S$, at a point belonging to $\Omega_R^{\hat{\epsilon}}$ (respectively $\Omega_L^{\hat{\epsilon}}$) and taken  on $\Omega_D^{\hat{\epsilon}}$. We denote $M_{\hat{x}_R}$ (respectively $M_{\hat{x}_L}$) the monodromy operator associated to the loop which makes one turn around the singular point $x=\hat{x}_R$ (respectively $x=\hat{x}_L$) in the negative (respectively positive) direction and which does not surround any other singular point (see Figure \ref{fig:Art2 23}).
\end{notation}

\begin{figure}[h!]
\begin{center}
{
\psfrag{G}{$M_{\hat{x}_L}$}
\psfrag{H}{$M_{\hat{x}_R}$}
\includegraphics[width=6cm]{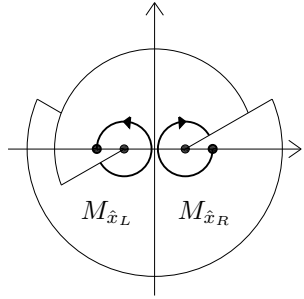}}
    \caption{Illustration of the definition of the monodromy operators $M_{\hat{x}_L}$ and $M_{\hat{x}_R}$, case $\hat{x}_L=\sqrt{\hat{\epsilon}} \in \mathbb{R}^*_{-}$.}
    \label{fig:Art2 23}
\end{center}
\end{figure}

\begin{proposition}\label{P:Stokes and monod}
For $l=L,R$, the action of the monodromy operator $M_{\hat{x}_l}$ on $W_V(\hat{\epsilon},x)$ is represented by the matrix $\hat{m}_l$ satisfying
\begin{equation}\label{E:Mxl in terms of S}
\hat{m}_l= \hat{C}_l \hat{D}_l,
\end{equation}
where $\hat{C}_l$ is the unfolded Stokes matrix defined by (\ref{E: def C_Re}) and $\hat{D}_l$, given by (\ref{E:Dx1}), is the matrix representing the action of the monodromy operator $M_{\hat{x}_l}$ on the fundamental matrix of solutions $F_V(\hat{\epsilon},x)$ of the model system.
\end{proposition}

\begin{proof}
Starting on $\Omega_R^{\hat{\epsilon}}$, the operator $M_{\hat{x}_R}$ acting on $W_V(\hat{\epsilon},x)=H_D(\hat{\epsilon},x)F_D(\hat{\epsilon},x)$ gives $H_U(\hat{\epsilon},x)F_D(\hat{\epsilon},x)\hat{D}_R$. Starting on $\Omega_L^{\hat{\epsilon}}$, the operator $M_{\hat{x}_L}$ acting on $W_V(\hat{\epsilon},x)=H_D(\hat{\epsilon},x)F_D(\hat{\epsilon},x)$ gives $H_U(\hat{\epsilon},x)F_D(\hat{\epsilon},x)\hat{D}_L$. As we have (\ref{E: def C_Re}), equation (\ref{E:Mxl in terms of S}) is verified for $l=L,R$.
\end{proof}

\begin{notation} \label{N:trivial row column}
The $j^{th}$ row (respectively column) of a matrix will be called \emph{trivial} if it corresponds to the $j^{th}$ row (respectively column) of the identity matrix.
\end{notation}

\begin{remark}
Relation (\ref{E:Mxl in terms of S}) gives a geometric meaning to zeros in unfolded Stokes matrices $\hat{C}_l$. For example, if a permutation $P$ is such that $P\hat{C}_lP^{-1}$ is in a block diagonal form, it indicates a decomposition of the solution space into invariant subspaces under the action of the monodromy operator $M_{\hat{x}_l}$. A trivial $j^{th}$ column (see Notation \ref{N:trivial row column}) of $\hat{C}_l$ points out that $w_j(\hat{\epsilon},x)$ is eigenvector of $M_{\hat{x}_l}$. If an unfolded Stokes matrix $\hat{C}_l$ is equal to the identity, it would imply that all the columns of $W_V(\hat{\epsilon},x)$ are eigenvectors of $M_{\hat{x}_l}$.
\end{remark}

Via the Jordan normal form of the monodromy matrix $\hat{C}_l \hat{D}_l$, we will now express how the entries of the unfolded Stokes matrices are linked to the existence of the solutions that are eigenvectors of the monodromy operator around the singular points. This will give a geometric interpretation to the entries of $\hat{C}_l=C_l(\hat{\epsilon})$ and, in particular, of their limits, the entries of $C_l=C_l(0)$.

\begin{theorem}\label{T:tout sur bases vp}
$t \in \mathbb{C}^n$ is an eigenvector of the monodromy matrix $\hat{m}_l$ if and only if $W_V(\hat{\epsilon},x)t$ is a solution eigenvector of the monodromy operator $M_{\hat{x}_l}$ with the same eigenvalue. Hence, for $l=L,R$, the number of independent solutions which are eigenvectors of $M_{\hat{x}_l}$ is equal to the number of Jordan blocks in the Jordan matrix associated to $\hat{m}_l=\hat{C}_l\hat{D}_l$. The values for which the monodromy matrix $\hat{m}_l$ may not be diagonalizable, are the resonance values of $\hat{\epsilon}$ specified in Definition \ref{D:resonant values} (which exactly correspond to multiple eigenvalues of $\hat{m}_l$).

For nonresonance values of $\hat{\epsilon}$, let $\hat{T}_l$ be the unipotent triangular matrix diagonalizing the monodromy matrix $\hat{m}_l=\hat{C}_l\hat{D}_l$:
\begin{equation}\label{E:monodromy matrix}
(\hat{T}_l)^{-1}\hat{m}_l\hat{T}_l= \hat{D}_l.
\end{equation}
The fundamental matrix of solutions
\begin{equation}\label{E:def Wxl}
W_{\hat{x}_l}(x)=W_V(\hat{\epsilon},x)\hat{T}_l
\end{equation}
is composed of eigenvectors of the monodromy operator around $x=\hat{x}_l$. A fundamental matrix having this property is unique up to its normalization: the $j^{th}$ column of $W_{\hat{x}_l}$ is a nonzero multiple of the Floquet solution (for example~\cite{wa87} p.~25) given by
\begin{equation}\label{E:floquet solution}
\hat{w}_{j,l}(x)=(x-\hat{x}_l)^{\hat{\mu}_{j,l}}\hat{g}_{j,l}(x),
\end{equation}
with $\hat{\mu}_{j,l}$ given by (\ref{E:mu j l k1}) and $\hat{g}_{j,l}(x)=e_j +O(|x-\hat{x}_l|)$ an analytic function of $x$ in a region containing $x=\hat{x}_l$ but no other singular point.

When $\hat{\epsilon}$ is a resonance value, the matrix $\hat{m}_l=\hat{C}_l\hat{D}_l$ is no more diagonalizable with no $j^{th}$ eigenvector if and only it the $j^{th}$ Floquet solution $\hat{w}_{j,l}(x)$ does not exist and has to be replaced, in the basis of solutions around $x=\hat{x}_l$, by a solution containing logarithmic terms.
\end{theorem}
\begin{proof}
By Proposition \ref{P:Stokes and monod}, we have $M_{\hat{x}_l} W_V(\hat{\epsilon},x)=W_V(\hat{\epsilon},x)\hat{m}_l$. Let $t \in \mathbb{C}^n$ and $\beta \in \mathbb{C}$. The first assertion of the theorem is obtained from
\begin{equation}
\begin{array}{lll}
\hat{m}_lt=\beta t &\iff W_V(\hat{\epsilon},x)\hat{m}_lt=\beta W_V(\hat{\epsilon},x)t \\&\iff M_{\hat{x}_l} W_V(\hat{\epsilon},x) t=\beta W_V(\hat{\epsilon},x)t.
\end{array}
\end{equation}

To prove the uniqueness (up to normalization) of $W_{\hat{x}_l}(x)$, let us suppose that $W^*$ is such that $M_{\hat{x}_l}W^*=W^*\hat{D}_l$. Since we have two bases of solutions, there exists a nonsingular matrix $K$ such that $W_{\hat{x}_l}=W^*K$. Since $M_{\hat{x}_l}W_{\hat{x}_l}=W_{\hat{x}_l}\hat{D}_l$, we must have $\hat{D}_lK=K\hat{D}_l$. Since $\hat{\epsilon}$ is not a resonance value, the eigenvalues of $\hat{D}_l$ are distinct and $K$ can only be diagonal.
\end{proof}

\begin{remark}\label{R:Vx1 en fonction de S}
For nonresonance values of $\hat{\epsilon}$, (\ref{E:monodromy matrix}) implies that the unfolded Stokes matrices are equal to the multiplicative commutator of the matrices $\hat{T}_l$ and $\hat{D}_l$:
\begin{equation}\label{E:Stokes and Txl}
\hat{C}_{l}= \hat{T}_l \hat{D}_l\hat{T}_l^{-1}\hat{D}_l^{-1}=[\hat{T}_l,\hat{D}_l].
\end{equation}
\end{remark}

The unfolded Stokes matrix $\hat{C}_R$ (respectively $\hat{C}_L$) is linked to the presence of logarithmic terms in solutions around $x=\hat{x}_R$ (respectively $x=\hat{x}_L$):

\begin{corollary}\label{C:Stokes and log}
There exist polynomials in terms of the entries of the unfolded Stokes matrices $\hat{C}_l=C_l(\hat{\epsilon})$ and the entries of $\hat{D}_l$ indicating, when they are nonzero at a resonance value, the nonexistence of a Floquet solution $\hat{w}_{j,l}(x)$ at the resonance.

In generic cases, the obstruction to the existence of Floquet solutions can be forced by the special form of the Stokes matrix $C_l=C_l(0)$. This is the case when
\begin{itemize}
\item $(C_R)_{12}\ne 0$: $\hat{w}_{2,R}(x)$ does not exist at the resonance $\hat{\mu}_{1,R}-\hat{\mu}_{2,R} \in \mathbb{N}^*$;
\item $(C_L)_{n(n-1)}\ne 0$: $\hat{w}_{n-1,L}(x)$ does not exist at the resonance $\hat{\mu}_{n,L}-\hat{\mu}_{n-1,L} \in \mathbb{N}^*$;
\item $\arg(\lambda_{s,0}-\lambda_{j,0})$ are distinct for all $s \ne j$: a nonvanishing $s^{th}$ polynomial in terms of the entries of the Stokes matrices $C_l$ with integer coefficients yields an obstruction to the existence of $\hat{w}_{j,l}(x)$ at the resonance $\hat{\mu}_{s,l}-\hat{\mu}_{j,l} \in \mathbb{N}^*$, with $s>j$ if $l=L$ and $s<j$ if $l=R$.
\end{itemize}
\end{corollary}

\begin{proof}
The polynomials of the corollary could be obtained by analytic or algebraic arguments, by counting the number of eigenvectors of $\hat{C}_l\hat{D}_l$. We present the proof in the analytic way. Recall that the matrices $\hat{T}_l$ are triangular and unipotent. Since $\hat{T}_l=\hat{C}_{l}\hat{D}_l\hat{T}_l\hat{D}_l^{-1}$(see (\ref{E:Stokes and Txl})), elements $(\hat{T}_l)_{ij}$, for $i \ne j$, can be calculated from the recurrent equations
\begin{equation}\label{E: Tij recurs}
(\hat{T}_l)_{ij}(1-\hat{\Delta}_{ij,l})=(\hat{C}_l)_{ij}+\sum_{\begin{subarray}{lll}i<k <j, \, l=R\\ j<k<i, \, l=L\end{subarray}} (\hat{C}_l)_{ik}(\hat{T}_l)_{kj}\hat{\Delta}_{kj,l},
\end{equation}
with $\hat{\Delta}_{sj,l}$ given by (\ref{E:Delta}). At the resonance, $\hat{\Delta}_{sj,l}=1$ for some $s,j,l$. Conditions to the nonexistence of the $j^{th}$ column of $\hat{T}_l$ at the resonance can be calculated from (\ref{E: Tij recurs}): they are given by polynomials in terms of entries of $\hat{D}_l$ and of entries of the unfolded Stokes matrices. For generic cases, these polynomials have a limit at $\epsilon = 0$ and the conditions can be formulated with polynomials in the entries of the Stokes matrices at $\epsilon=0$: the nonvanishing of the polynomials for small $\hat{\epsilon}$ is ensured by the nonvanishing of the limit polynomial at $\epsilon=0$ which depends on $C_l$. This is the case for the conditions to the existence of
\begin{itemize}
\item  the second column of $\hat{T}_R$;
\item the $(n-1)^{th}$ column of $\hat{T}_L$;
\item all columns if the Stokes lines are distinct (i.e. $\arg(\lambda_{s,0}-\lambda_{j,0})$ are distinct for all $s \ne j$). In that case, the resonance $\hat{\Delta}_{ij,l}= 1$ is distinct from the resonance $\hat{\Delta}_{kj,l} = 1$ for $k \ne i$. On the sequence $\hat{\epsilon}_n \to 0$ corresponding to the resonance $\hat{\Delta}_{ij,l}= 1$, the limit of $\left( \frac{\hat{\Delta}_{kj,l}}{1-\hat{\Delta}_{kj,l}} \right)$ is $0$ or $-1$, hence the polynomials at the limit have integer coefficients (independent of $\hat{\epsilon}$).
\end{itemize}
\end{proof}

\begin{example}
Let us consider the case $n=3$, with distinct arguments of $\lambda_2-\lambda_3$, $\lambda_1-\lambda_2$ and $\lambda_1-\lambda_3$. Equation (\ref{E: Tij recurs}) gives
\begin{equation}\label{E:t en fonction de s}
\begin{array}{lll}
(\hat{T}_R)_{12}(1-\hat{\Delta}_{12,R})=(\hat{C}_R)_{12}, \\
(\hat{T}_R)_{13}(1-\hat{\Delta}_{13,R})=(\hat{C}_R)_{13}+(\hat{C}_R)_{12} (\hat{C}_R)_{23}\left(\frac{\hat{\Delta}_{23,R}}{1-\hat{\Delta}_{23,R}}\right),\\
(\hat{T}_R)_{23}(1-\hat{\Delta}_{23,R})=(\hat{C}_R)_{23},\\
\\(\hat{T}_L)_{21}(1-\hat{\Delta}_{21,L})=(\hat{C}_L)_{21},\\
(\hat{T}_L)_{31}(1-\hat{\Delta}_{31,L})=(\hat{C}_L)_{31}+(\hat{C}_L)_{21}(\hat{C}_L)_{32}\left(\frac{\hat{\Delta}_{21,L}}{1-\hat{\Delta}_{21,L}}\right),\\
(\hat{T}_L)_{32}(1-\hat{\Delta}_{32,L})=(\hat{C}_L)_{32}.
\end{array}
\end{equation}
Decreasing values of $\hat{\epsilon}$ such that $\hat{\mu}_{1,R}-\hat{\mu}_{3,R} \in \mathbb{N}^*$ and $\hat{\mu}_{3,L}-\hat{\mu}_{1,L} \in \mathbb{N}^*$ are approaching the ray $\arg(\sqrt{\epsilon})=\arg(\lambda_{3,0}-\lambda_{1,0})$. The following comes from the inequalities $\arg(\lambda_{1,0}-\lambda_{2,0})<\arg(\lambda_{1,0}-\lambda_{3,0})<\arg(\lambda_{2,0}-\lambda_{3,0})$. When $\hat{\epsilon} \to 0$ on resonance values
\begin{itemize}
\item
$\hat{\mu}_{1,R}-\hat{\mu}_{3,R} \in \mathbb{N}^*$ making $\hat{\Delta}_{13,R}=1$, we have $\Im(\hat{\mu}_{2,R}-\hat{\mu}_{3,R})>0$ and then $\left(\frac{\hat{\Delta}_{23,R}}{1-\hat{\Delta}_{23,R}}\right)=\left(\frac{1}{\hat{\Delta}_{32,R}-1}\right)$ tends to $-1$, since $\hat{\Delta}_{32,R} \to 0$;
\item $\hat{\mu}_{3,L}-\hat{\hat{\mu}}_{1,L} \in \mathbb{N}^*$ making $\hat{\Delta}_{31,L}=1$, we have $\Im(\hat{\mu}_{1,L}-\hat{\mu}_{2,L})>0$ and then $\left(\frac{\hat{\Delta}_{21,L}}{1-\hat{\Delta}_{21,L}}\right)=\left(\frac{1}{\hat{\Delta}_{12,L}-1}\right)$ tends to $-1$, since $\hat{\Delta}_{12,L} \to 0$.
\end{itemize}
These limits imply that the right hand side of the equations (\ref{E:t en fonction de s}) at the resonance is minus an entry of the inverse of the unfolded Stokes matrices. We immediately see that
\begin{itemize}
\item if $(C_R)_{12} \ne 0$, $(\hat{T}_R)_{12}$ (and hence $\hat{w}_{2,R}(x)$) does not have a limit at the resonances $\hat{\mu}_{1,R}-\hat{\mu}_{2,R} \in \mathbb{N}^*$ making $\hat{\Delta}_{12,R}=1$;
\item if $(C_R)_{23} \ne 0$, $(\hat{T}_R)_{23}$ (and hence $\hat{w}_{3,R}(x)$) does not have a limit at the resonances $\hat{\mu}_{2,R}-\hat{\mu}_{3,R} \in \mathbb{N}^*$ making $\hat{\Delta}_{23,R}=1$;
\item if $(C_R)_{13}-(C_R)_{12}(C_R)_{23} \ne 0$, $(\hat{T}_R)_{13}$ (and hence $\hat{w}_{3,R}(x)$) does not have a limit at the resonances $\hat{\mu}_{1,R}-\hat{\mu}_{3,R} \in \mathbb{N}^*$ making $\hat{\Delta}_{13,R}=1$;
\item if $(C_L)_{21} \ne 0$, $(\hat{T}_L)_{21}$ (and hence $\hat{w}_{1,L}(x)$) does not have a limit at the resonances $\hat{\mu}_{2,L}-\hat{\mu}_{1,L} \in \mathbb{N}^*$ making $\hat{\Delta}_{21,L}=1$;
\item if $(C_L)_{32} \ne 0$, $(\hat{T}_L)_{32}$ (and hence $\hat{w}_{2,L}(x)$) does not have a limit at the resonances $\hat{\mu}_{3,L}-\hat{\mu}_{2,L} \in \mathbb{N}^*$ making $\hat{\Delta}_{32,L}=1$;
\item if $(C_L)_{31}-(C_L)_{21}(C_L)_{32} \ne 0$, $(\hat{T}_L)_{31}$ (and hence $\hat{w}_{1,L}(x)$) does not have a limit at the resonances $\hat{\mu}_{3,L}-\hat{\mu}_{1,L} \in \mathbb{N}^*$ making $\hat{\Delta}_{31,L}=1$.
\end{itemize}
\end{example}

\subsection{Stokes matrices and monodromy in the Riccati systems}\label{S:monodromy Riccati}

In this section, we give a meaning to the unfolded Stokes matrices in the corresponding Riccati systems. This allows an interpretation of the Stokes matrices at $\epsilon=0$. This section is not prerequisite to state the complete system of analytic invariants of the systems (\ref{E:prenormal system k1}) .

We will look at the monodromy of first integrals in the Riccati systems. These first integrals are obtained from the basis of the linear system.

\begin{proposition}
For $x \in V^{\hat{\epsilon}}$, the $j^{th}$ Riccati system has first integrals $\mathcal{H}^j_q$, for $q \in \{1,2,...,n\} \backslash \{j\}$, that can be written as
\begin{equation}\label{E:definition monod first integ}
\mathcal{H}^j_q(\hat{\epsilon},x,[y]_j)=(-1)^{q-j}\frac{ \left|\begin{matrix}b_1^j(\hat{\epsilon},x,[y]_j)  &...&\widehat{b_q^j(\hat{\epsilon},x,[y]_j)} &... &b_n^j(\hat{\epsilon},x,[y]_j)\end{matrix}\right|}{\left|\begin{matrix}b_1^j(\hat{\epsilon},x,[y]_j)  &...&\widehat{b_j^j(\hat{\epsilon},x,[y]_j)} &... &b_n^j(\hat{\epsilon},x,[y]_j)\end{matrix}\right|},
\end{equation}
with
\begin{equation}\label{E:bij}
b_i^j(\hat{\epsilon},x,[y]_j)=(-1)^{i-j}(w_i(\hat{\epsilon},x))_j \left( [y]_j-[w_i]_j \right),
\end{equation}
and $w_i(\hat{\epsilon},x)$ the $i^{th}$ column of the fundamental matrix of solutions $W_V(\hat{\epsilon},x)$ given by (\ref{E:basis union}) (for $[w_i]_j$, see Notation \ref{N:projective not}). $\mathcal{H}^j_q$ has values in $\mathbb{CP}^1$ for $q \ne j$.
\end{proposition}

\begin{proof}
Let $w_i(\hat{\epsilon},x)$ be the columns of the fundamental matrix of solutions $W_V(\hat{\epsilon},x)$ given by (\ref{E:basis union}). The general solution of a linear system (\ref{E:prenormal system k1}) may be expressed as a linear combination $y=\sum_{q=1}^n k_q w_q(\hat{\epsilon},x)$ of the particular solution $w_q(\hat{\epsilon},x)$, with $k_q \in \mathbb{C}$. In particular, the $j^{th}$ component of this general solution $y$ satisfies
\begin{equation}
(y)_j=\sum_{q=1}^n k_q (w_q(\hat{\epsilon},x))_j,
\end{equation}
so
\begin{equation}
\sum_{q=1}^n k_q (w_q(\hat{\epsilon},x))_j \frac{y}{(y)_j}=\sum_{q=1}^n k_q w_q(\hat{\epsilon},x),
\end{equation}
and
\begin{equation}\label{E:equation b et 0}
\sum_{q=1}^n \frac{k_q}{k_j} \left(w_q(\hat{\epsilon},x)-(w_q(\hat{\epsilon},x))_j \frac{y}{(y)_j}\right)=0.
\end{equation}
Solving for $\frac{k_q}{k_j}$, $q \ne j$, and using Notation \ref{N:projective not} and (\ref{E:bij}) gives (\ref{E:definition monod first integ}).
\end{proof}

As detailed in the next theorem, entries of the inverse of the unfolded Stokes matrices appear in the expression of the monodromy of the first integrals $\mathcal{H}_{q}^j$ around $x=\hat{x}_l$.

\begin{theorem}\label{T:monodromy first integral}
The monodromy of a first integral $\mathcal{H}_q^{j}$ around $x=\hat{x}_l$ may be written as the composition of
\begin{itemize}
\item a wild part (i.e. of the form $e^{\frac{2 \pi i}{\alpha}}$ with $\alpha \in \mathbb{C}$, $\alpha \to 0$) depending on the formal invariants,
\item a map depending on the entries of the inverse of the unfolded Stokes matrices and having a limit for $\epsilon=0$.
\end{itemize}
More precisely, with $\mathcal{H}_j^j=1$, the monodromy of the first integrals may be expressed as
\begin{equation}\label{E:monodromy of first integ 1}
M_{\hat{x}_R}(\mathcal{H}^j_q) =\hat{\Delta}_{jq,R}\frac{\mathcal{H}^j_q+\sum_{p=q+1}^{n}(\hat{C}_R^{-1})_{qp} \mathcal{H}^j_p}{1+\sum_{p=j+1}^n(\hat{C}_R^{-1})_{jp}\mathcal{H}^j_p},
\end{equation}
and
\begin{equation}\label{E:monodromy of first integ 2}
M_{\hat{x}_L}(\mathcal{H}^j_q) =\hat{\Delta}_{jq,L}\frac{\mathcal{H}^j_q+\sum_{p=1}^{q-1}(\hat{C}_L^{-1})_{qp} \mathcal{H}^j_p}{1+\sum_{p=1}^{j-1}(\hat{C}_L^{-1})_{jp}\mathcal{H}^j_p}.
\end{equation}
Denoting
\begin{equation}\label{E: def Hj}
\mathcal{H}^j=(\mathcal{H}_1^j,...,\mathcal{H}_n^j)^T,
\end{equation}
this is equivalent to
\begin{equation}\label{E:monodromy of first integ vec}
M_{\hat{x}_l}(\mathcal{H}^j) =diag\{\hat{\Delta}_{j1,l},...,\hat{\Delta}_{jn,l}\}\frac{\hat{C}_l^{-1} \mathcal{H}^j}{[(\hat{C}_l^{-1})_{j1},...,(\hat{C}_l^{-1})_{jn}]\mathcal{H}^j},
\end{equation}
with $\hat{\Delta}_{jq,l}$ as defined by (\ref{E:Delta}).
\end{theorem}
\begin{proof}
In order to calculate the monodromy of the first integrals given by (\ref{E:definition monod first integ}), we need to compute the monodromy of
\begin{equation}
B^{j}(\hat{\epsilon},x,[y]_j)=[\begin{matrix}b_1^j(\hat{\epsilon},x,[y]_j) &... &b_n^j(\hat{\epsilon},x,[y]_j)\end{matrix}],
\end{equation}
with $b_i^j(\hat{\epsilon},x,[y]_j)$ given by (\ref{E:bij}).
Since the monodromy of $w_q(\hat{\epsilon},x)$ is given by Proposition \ref{P:Stokes and monod}, we have
\begin{equation}\label{E:monod of B}
 M_{\hat{x}_l}(B^{j}(\hat{\epsilon},x,[y]_j))=B^{j} (\hat{\epsilon},x,[y]_j)\hat{m}_l,
\end{equation}
with $\hat{m}_l$ given by (\ref{E:Mxl in terms of S}). With $\mathcal{H}^j$ defined in (\ref{E: def Hj}), relation (\ref{E:equation b et 0}) implies
\begin{equation}\label{E:systemB}
B^{j}(\hat{\epsilon},x,[y]_j) \mathcal{H}^j=0,
\end{equation}
and thus, using (\ref{E:monod of B}),
\begin{equation}\label{E:systemB mono}
B^{j}(\hat{\epsilon},x,[y]_j)\hat{m}_l M_{\hat{x}_l}(\mathcal{H}^j)=0.
\end{equation}
Equations (\ref{E:systemB}) and (\ref{E:systemB mono}) imply that
\begin{equation}
M_{\hat{x}_l}(\mathcal{H}^j_q) =\frac{(\hat{m}_l^{-1}\mathcal{H}^j )_q}{(\hat{m}_l^{-1}\mathcal{H}^j )_j},
\end{equation}
leading to the equations of the theorem, using (\ref{E:Mxl in terms of S}).
\end{proof}

Theorem \ref{T:monodromy first integral} yields the following interpretation of the Stokes matrices at $\epsilon=0$:

\begin{corollary}\label{C:monodromy integral}
The first integral $\mathcal{H}_q^j$ is an eigenvector of the monodromy operator around a singular point $x=\hat{x}_l$ (by this we means $M_{\hat{x}_l}\mathcal{H}_q^j=\hat{\Delta}_{jq,l}\mathcal{H}_q^j$ ) if and only if the rows $j$ and $q$ in the inverse of the unfolded Stokes matrix $\hat{C}_l$ are trivial (see Notation \ref{N:trivial row column}). Hence, a nontrivial $i^{th}$ row in the inverse of the right (respectively left) Stokes matrix at $\epsilon=0$ is an obstruction for the first integrals $\mathcal{H}^i_k$ to be eigenvectors of the monodromy operator around the right (respectively left) singular point, for $k \in \{1,...,n\} \backslash \{i\}$.
\end{corollary}

\begin{proof}
This is immediate from equations (\ref{E:monodromy of first integ 1}) and (\ref{E:monodromy of first integ 2}).
\end{proof}

The wild part in the monodromy of the first integrals of the Riccati system is due to the definition of the fundamental matrix of solutions of the model system over the considered domain and is not a consequence of the Stokes phenomenon:

\begin{remark}
If we compare first integrals over the intersections of the sectorial domains $\Omega_U^{\hat{\epsilon}}$ and $\Omega_D^{\hat{\epsilon}}$ instead of over the auto-intersection of $V^{\hat{\epsilon}}$ (thus taking Notation \ref{N: def F_s} for $F_s(\hat{\epsilon},x)$ over $\Omega_s^{\hat{\epsilon}}$ instead of Notation \ref{N:Fv} for $F_V(\hat{\epsilon},x)$ over $V^{\hat{\epsilon}}$), the wild part is only present in the comparison over $\Omega_C^{\hat{\epsilon}}$ (which does not exist at $\epsilon=0$). When we compare the first integrals over $\Omega_R^{\hat{\epsilon}}$ and $\Omega_L^{\hat{\epsilon}}$, there is no wild part in equations corresponding to (\ref{E:monodromy of first integ 1}), (\ref{E:monodromy of first integ 2}) and (\ref{E:monodromy of first integ vec}).
\end{remark}

\subsection{Auto-intersection relation and $\frac{1}{2}$-summable representative of the equivalence class of unfolded Stokes matrices}\label{S:relation d'autointersection}
In this section, we compare the two points of view that we have on $S_\cap$, the auto-intersection of $S$. This will yield a relation that is satisfied for all $\epsilon \in S_\cap$. We call it the auto-intersection relation. It allows to prove the existence of a representative of the equivalence class of unfolded Stokes matrices which is $\frac{1}{2}$-summable in $\epsilon$. Further, it will be a necessary and sufficient condition for the realization of the complete system of analytic invariants.

For $\bar{\epsilon}$ and $\tilde{\epsilon}=\bar{\epsilon} e^{2 \pi i}$ in $S_\cap$ (Figure \ref{fig:Art2 16}), we have two different presentations of the dynamics of the same linear differential system. On $S_\cap$, since there is no resonance value, there always exists transition matrices between fundamental matrices of solutions composed of eigenvectors of the monodromy operators around the singular points. We will use these transition matrices to compare the two presentations on $S_\cap$. First, let us take the monodromy operators with the base point taken on the upper (respectively lower) sectorial domain when the corresponding loop surrounds the upper (respectively lower) singular point.

\begin{notation}\label{N:monodo base point}
In Notation \ref{N:monodromy operator}, we defined the monodromy operators $M_{\hat{x}_R}$ and $M_{\hat{x}_L}$, for $\hat{\epsilon} \in S$. Over $S_{\cap}$, let us denote
\begin{itemize}
\item $M^*_{\tilde{x}_L}= M_{\tilde{x}_L}$,
\item $M^*_{\bar{x}_R}=M_{\bar{x}_R}$,
\item $M^*_{\bar{x}_L} =M_{\bar{x}_L}^{-1}$,
\item $M^*_{\tilde{x}_R}=M_{\tilde{x}_R}^{-1}$.
\end{itemize}
Hence, the base points of $M^*_{\tilde{x}_L}$ and $M^*_{\bar{x}_R}$ belongs to $\Omega_D^{\bar{\epsilon}} \cap \Omega_D^{\tilde{\epsilon}}$, whereas the base points of $M^*_{\bar{x}_L}$ (respectively $M^*_{\tilde{x}_R}$) are taken on $\Omega_U^{\bar{\epsilon}} \cap \Omega_U^{\tilde{\epsilon}}$ (Figures \ref{fig:Art2 V} and \ref{fig:Art2 Y}).
\end{notation}

\begin{figure}[h!]
\begin{center}
{\psfrag{B}{$\Omega^{\tilde{\epsilon}}_D$}
\psfrag{C}{$\Omega^{\tilde{\epsilon}}_U$}
\psfrag{E}{\small{$\tilde{x}_L$}}
\psfrag{F}{\small{$\tilde{x}_R$}}
\psfrag{X}{$\tilde{\epsilon} \in S_{\cap}$}
\psfrag{Y}{$\bar{\epsilon} \in S_{\cap}$}
\psfrag{K}{$\Omega^{\bar{\epsilon}}_D$}
\psfrag{J}{$\Omega^{\bar{\epsilon}}_U$}
\psfrag{G}{\small{$\bar{x}_L$}}
\psfrag{H}{\small{$\bar{x}_R$}}
\includegraphics[width=11cm]{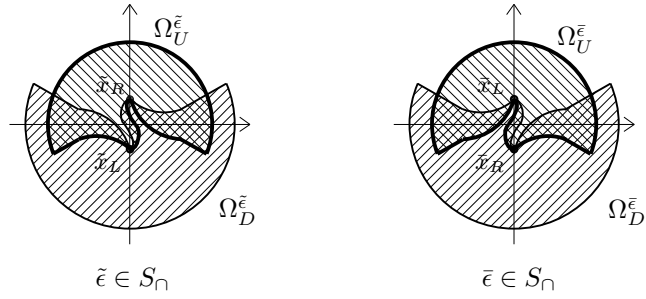}}
    \caption{Sectorial domains in the $x$-variable for $\hat{\epsilon} \in S_{\cap}$.}
    \label{fig:Art2 V}
\end{center}
\end{figure}

\begin{figure}[h!]
\begin{center}
{\psfrag{B}{\small{$\Omega^{\tilde{\epsilon}}_D$}}
\psfrag{C}{\small{$\Omega^{\tilde{\epsilon}}_U$}}
\psfrag{E}{\small{$M^*_{\tilde{x}_L}$}}
\psfrag{F}{\small{$M^*_{\tilde{x}_R}$}}
\psfrag{X}{$\tilde{\epsilon} \in S_{\cap}$}
\psfrag{Y}{$\bar{\epsilon} \in S_{\cap}$}
\psfrag{K}{\small{$\Omega^{\bar{\epsilon}}_D$}}
\psfrag{J}{\small{$\Omega^{\bar{\epsilon}}_U$}}
\psfrag{G}{\small{$M^*_{\bar{x}_L}$}}
\psfrag{H}{\small{$M^*_{\bar{x}_R}$}}
\includegraphics[width=10cm]{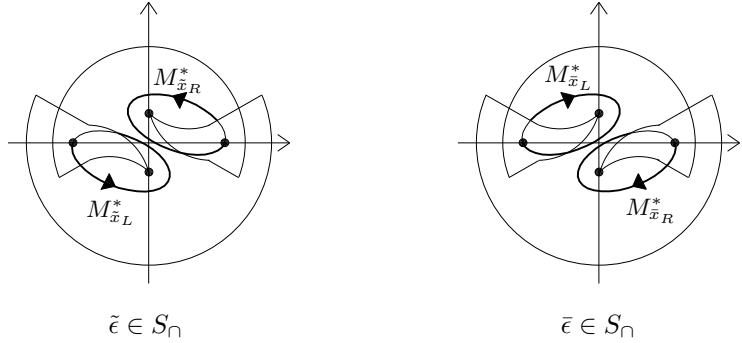}}
    \caption{Illustration of the definition of the monodromy operators $M^*_{\tilde{x}_L}$, $M^*_{\bar{x}_R}$, $M^*_{\bar{x}_L}$ and $M^*_{\tilde{x}_R}$.}
    \label{fig:Art2 Y}
\end{center}
\end{figure}

\begin{definition}\label{D:def matrix transition}
For $l=L,R$, let us take $W_{\hat{x}_l}(x)$ a fundamental matrix of solutions of (\ref{E:prenormal system k1}) composed of eigenvectors of the monodromy operator $M^*_{\hat{x}_l}$, depending analytically on $\hat{\epsilon} \in S_\cap$ and converging uniformly over compact sets of $\Omega_s^0$ when $\hat{\epsilon} \to 0$ (and $\hat{\epsilon} \in S_\cap$) to $W_s(0,x)$ defined by (\ref{E:confluent solution}), with $s=D$ if $\Im(\hat{x}_l)<0$ and $s=U$ otherwise. Let $E_{L,\hat{x}_L \to \hat{x}_R}$ be the matrix such that, over a fixed compact set of $\Omega_L^0$ sufficiently far from the singular points,
\begin{equation}\label{E:def transit matri L}
E_{L,\hat{x}_L \to \hat{x}_R}=(W_{\hat{x}_L}(x))^{-1}W_{\hat{x}_R}(x).
 \end{equation}
Let $E_{\hat{x}_L \to \hat{x}_R}$ be the matrix such that, over a fixed compact set of $\Omega_R^0$ sufficiently far from the singular points,
\begin{equation}\label{E:def transit matri R}
E_{R,\hat{x}_L \to \hat{x}_R}=(W_{\hat{x}_L}(x))^{-1}W_{\hat{x}_R}(x).
\end{equation}
We call $E_{L,\hat{x}_L \to \hat{x}_R}$ (respectively $E_{\hat{x}_L \to \hat{x}_R}$) the \emph{left} (respectively \emph{right}) \emph{transition matrix} from $\hat{x}_L$ to $\hat{x}_R$. These transition matrices are unique up to multiplication on each side by nonsingular diagonal matrices depending analytically on $\hat{\epsilon} \in S_\cap$, with a nonsingular limit at $\epsilon=0$ (coming from the normalization of the chosen fundamental matrices of solutions).
\end{definition}

The following proposition is implicit from the paper \cite{aG99} of A.~Glutsyuk. The proof will be useful later.

\begin{proposition}\label{P:transit inv}
Let us take two families of systems
\begin{equation}\label{E:system 1 2}
(x^2-\hat{\epsilon})y_{i}'=B_i(\hat{\epsilon},x)y_{i}, \quad i=1,2,
\end{equation}
having the form (\ref{E:prenormal system k1}) with the same model system and depending on $\hat{\epsilon} \in S_\cap$. Let
\begin{equation}
x_U=\bar{x}_L=\tilde{x}_R, \qquad x_D=\bar{x}_R=\tilde{x}_L.
\end{equation}
Let us take for each family of systems a right transition matrix from $x_D$ to $x_U$, i.e. $E^i_{R,x_D \to x_U}$ (Definition \ref{D:def matrix transition}). The two family of systems (\ref{E:system 1 2}) are analytically equivalent, the equivalence depending analytically on $(\epsilon,x) \in S_\cap \times \mathbb{D}_r$ and converging uniformly on compact sets of $\mathbb{D}_r$ when $\epsilon \to 0$, if and only if there exist $\hat{Q}_U$ and $\hat{Q}_D$ nonsingular diagonal matrices depending analytically on $\hat{\epsilon} \in S_\cap$, with a nonsingular limit at $\epsilon=0$, and such that
\begin{equation}\label{E:trans inv eq ab 2}
E^1_{R,x_D \to x_U}\hat{Q}_U=\hat{Q}_D E^2_{R,x_D \to x_U}.
\end{equation}
\end{proposition}

\begin{proof}
Let us denote by $W^i_{\hat{x}_l}(x)$, $l=L,R$, the fundamental matrix of solutions taken to calculate the right transition matrices $E^i_{R,x_D \to x_U}$, $i=1,2$. Let us take two domains $\mathcal{G}^{\hat{\epsilon}}_{U}$ and $\mathcal{G}^{\hat{\epsilon}}_{D}$ covering $\mathbb{D}_r$ (Figure \ref{fig:Art2 15}), such that $\mathcal{G}^{\hat{\epsilon}}_{U}$ (respectively $\mathcal{G}^{\hat{\epsilon}}_{D}$) contains $x_U$ but not $x_D$ (respectively $x_D$ but not $x_U$) and has the limit $\Omega_U^0$ (respectively $\Omega_D^0$) when $\hat{\epsilon} \to 0$ in $S_\cap$.

\begin{figure}[h!]
\begin{center}
{\psfrag{E}{\small{$x_D$}}
\psfrag{D}{\small{$x_U$}}
\psfrag{B}{$\mathcal{G}^{\hat{\epsilon}}_{U}$}
\psfrag{C}{$\mathcal{G}^{\hat{\epsilon}}_{D}$}
\includegraphics[width=5cm]{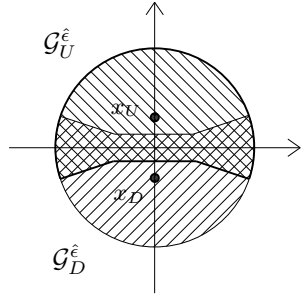}}
    \caption{Domains $\mathcal{G}^{\hat{\epsilon}}_{U}$ and $\mathcal{G}^{\hat{\epsilon}}_{D}$ and their intersection.}
    \label{fig:Art2 15}
\end{center}
\end{figure}

Let us suppose that (\ref{E:trans inv eq ab 2}) is satisfied. The transformation $y_1=P_{\hat{\epsilon}}(x) y_2$, with
\begin{equation}
P_{\hat{\epsilon}}(x)=\begin{cases}W^1_{x_U}(x)\hat{Q}_U(W^2_{x_U}(x))^{-1}, \quad \mbox{on }\mathcal{G}^{\hat{\epsilon}}_{U}, \\
W^1_{x_D}(x)\hat{Q}_D(W^2_{x_D}(x))^{-1}, \quad \mbox{on }\mathcal{G}^{\hat{\epsilon}}_{D},
\end{cases}
\end{equation}
is well-defined on $\mathbb{D}_r$ because of (\ref{E:trans inv eq ab 2}), for any $\hat{\epsilon} \in S_\cap \cup \{0\}$. It conjugates the two systems, depends analytically on $(\hat{\epsilon},x) \in S_\cap \times \mathbb{D}_r$ and converges uniformly on compact sets of $\mathbb{D}_r$ when $\hat{\epsilon} \to 0$.

On the other hand, let us suppose that the change $y_1=P_{\hat{\epsilon}}(x) y_2$ yields an analytic equivalence (as in the statement of the proposition) between the two systems. Then, by uniqueness (up to normalization) of $W^i_{\hat{x}_L}(x)$ and $W^i_{\hat{x}_R}(x)$, we must have
\begin{equation} \label{E:Pe eq U}
P_{\hat{\epsilon}}(x)W^2_{x_U}(x)=W^1_{x_U}(x)\hat{Q}_U, \quad \mbox{over } \mathcal{G}^{\hat{\epsilon}}_{U},
\end{equation}
and
\begin{equation}\label{E:Pe eq D}
P_{\hat{\epsilon}}(x)W^2_{x_D}(x)=W^1_{x_D}(x)\hat{Q}_D, \quad \mbox{over }\mathcal{G}^{\hat{\epsilon}}_{D},
\end{equation}
with $\hat{Q}_U$ and $\hat{Q}_D$ nonsingular diagonal matrices depending analytically on $\hat{\epsilon} \in S_\cap$ and having a nonsingular limit at $\epsilon=0$. Isolating $P_{\hat{\epsilon}}(x)$ in (\ref{E:Pe eq U}) and (\ref{E:Pe eq D}), we get, over a compact in $\Omega_{R}^{0}$ and for $\hat{\epsilon}$ sufficiently small,
\begin{equation}
W^1_{x_D}(x)\hat{Q}_D(W^2_{x_D}(x))^{-1}=W^1_{x_U}(x)\hat{Q}_U(W^2_{x_U}(x))^{-1},
\end{equation}
which is equivalent to (\ref{E:trans inv eq ab 2}), using (\ref{E:def transit matri R}). \end{proof}

\begin{remark}
Taking the left transition matrices instead of the right transition matrices in Proposition \ref{P:transit inv} (and taking in the proof a compact set on $\Omega_L^0$ instead of $\Omega_R^0$) yields similar result.
\end{remark}

When taking a system of the form (\ref{E:prenormal system k1}), the right transition matrices $E_{R,\tilde{x}_L \to \tilde{x}_R}$ and $E_{R,\bar{x}_R \to \bar{x}_L}$ both correspond to the transition from the lower point to the upper point. By Proposition \ref{P:transit inv}, we know they satisfy a relation like (\ref{E:trans inv eq ab 2}). We formulate it more precisely in Proposition \ref{P:relation compat prop}.

\begin{definition}\label{D:def expon close}
For $r(\epsilon)$ analytic in $\epsilon \in S_\cap$, we say that $r(\epsilon)$ is exponentially close to $0$ in $\sqrt{\epsilon}$ if it satisfies $|r(\epsilon)|<b e^{-\frac{a}{\sqrt{|\epsilon|}}}$ for some $a,b \in \mathbb{R}^*_+$.
\end{definition}

\begin{lemma}\label{L:rela delta tilde et bar}
Following its definition given by (\ref{E:Delta}),
\begin{equation}\label{E:rela delta tilde et bar}
\tilde{\Delta}_{sj,l}=(\tilde{D}_l)_{ss}(\tilde{D}_l^{-1})_{jj}, \quad s < j, \,l=L,R\\
\end{equation}
is exponentially close to $0$ in $\sqrt{\epsilon}$ (to prove it, use (\ref{E:theta1 satisfy})). We also have
\begin{equation}(\hat{\Delta}_{sj,l})^{-1}=\hat{\Delta}_{js,l}
\end{equation}
and
\begin{equation}
\hat{\Delta}_{sj,l}\hat{\Delta}_{ji,l}=\hat{\Delta}_{si,l}.
\end{equation}
By (\ref{E:xLR}), (\ref{E:Dx1}) and (\ref{E:def tilde et bar eps}), we obtain
\begin{equation}\label{E:dl tilde et bar}
\tilde{D}_L=\bar{D}_R^{-1}, \quad \tilde{D}_R=\bar{D}_L^{-1}
\end{equation}
and, by (\ref{E:Delta}),
\begin{equation}
\begin{array}{lll}
\bar{\Delta}_{sj,R}=\tilde{\Delta}_{js,L},\\
\bar{\Delta}_{sj,L}=\tilde{\Delta}_{js,R}.
\end{array}
\end{equation}
Hence, $\bar{\Delta}_{sj,l}$ is exponentially close to $0$ in $\sqrt{\epsilon}$ for $s > j$ and $l=L,R$. \end{lemma}

\begin{lemma}\label{L:exp t}
On $S_\cap$, entries from the following matrices are exponentially close to $0$ in $\sqrt{\epsilon}$ in the sense of Definition \ref{D:def expon close}:
\begin{equation}\label{E:mat exp small 2}
\begin{array}{lll}
&\bar{C}_L-\bar{T}_L,  \quad &I-\tilde{T}_L,  \\
&\tilde{C}_R-\tilde{T}_R, \quad &I-\bar{T}_R.
\end{array}
\end{equation}
\begin{equation}\label{E:mat exp small 3}
\begin{array}{lll}
&\bar{C}_L^{-1}-\bar{T}_L^{-1}, \quad &I-\tilde{T}_L^{-1},  \\
&\tilde{C}_R^{-1}-\tilde{T}_R^{-1}, \quad &I-\bar{T}_R^{-1},
\end{array}
\end{equation}
\begin{equation}\label{E:mat exp small}
\begin{array}{lll}
\tilde{C}_L-\tilde{D}_L \tilde{T}_L^{-1} \tilde{D}_L^{-1},  \\
\bar{C}_R-\bar{D}_R \bar{T}_R^{-1} \bar{D}_R^{-1},
\end{array}
\end{equation}
and
\begin{equation}\label{E:mat exp small 4}
\begin{array}{lll}
I-\bar{D}_R \bar{T}_L\bar{D}_R^{-1}, \quad &I-\bar{D}_L \bar{T}_L\bar{D}_L^{-1}, \quad &I-\bar{D}_R \bar{T}_L^{-1}\bar{D}_R^{-1}\\
I-\tilde{D}_R \tilde{T}_R \tilde{D}_R^{-1}  \quad &I-\tilde{D}_L \tilde{T}_R \tilde{D}_L^{-1} \quad &I-\tilde{D}_L \tilde{T}_R^{-1} \tilde{D}_L^{-1}.
\end{array}
\end{equation}
\end{lemma}

\begin{proof}
The proof follows from Lemma \ref{L:rela delta tilde et bar} and (\ref{E:Stokes and Txl}). Relation (\ref{E: Tij recurs}) is used to obtain (\ref{E:mat exp small 2}). Since $\hat{T}_l^{-1}=\hat{D}_l \hat{T}_l^{-1} \hat{D}_l^{-1}\hat{C}_{l}^{-1}$, we have, for $i \ne j$,
\begin{equation}\label{E:tij recr 3}
(\hat{T}_l^{-1})_{ij}(1-\Delta_{ij,l})=(\hat{C}_l^{-1})_{ij}+\sum_{\begin{subarray}{lll}i<k <j, \, l=R\\ j<k<i, \, l=L\end{subarray}} (\hat{T}_l^{-1})_{ik}(\hat{C}_l^{-1})_{kj}\Delta_{ik,l} ,
\end{equation}
Relation (\ref{E:tij recr 3}) leads to (\ref{E:mat exp small 3}). Since $\hat{D}_l \hat{T}_l^{-1} \hat{D}_l^{-1}= \hat{T}_l^{-1}\hat{C}_{l}$, we have, for $i \ne j$,
\begin{equation}\label{E:tij recr 2}
(\hat{T}_l^{-1})_{ij}(\Delta_{ij,l}-1)=(\hat{C}_l)_{ij}+\sum_{\begin{subarray}{lll}i<k <j, \, l=R\\ j<k<i, \, l=L\end{subarray}} (\hat{T}_l^{-1})_{ik}(\hat{C}_l)_{kj},
\end{equation}
Relations (\ref{E:tij recr 2}) and (\ref{E:mat exp small 3}) yield (\ref{E:mat exp small}). Finally, (\ref{E:mat exp small 4}) follows from (\ref{E:mat exp small 2}) and (\ref{E:mat exp small 3}), using (\ref{E:prod DrDl}) if necessary.
\end{proof}

\begin{definition}\label{D:Nl tilde}
Let the formal invariants be given. Let $\{\hat{C}_R, \hat{C}_L \}$ be a representative of the equivalence class of the unfolded Stokes collections. Let $\hat{T}_l$ be obtained by (\ref{E:monodromy matrix}). We define
\begin{equation}
\begin{array}{lll}
\tilde{N}_L=\tilde{D}_L\tilde{T}_L^{-1}\tilde{T}_R\tilde{D}_L^{-1}, \qquad &\tilde{N}_R=\tilde{T}_L^{-1}\tilde{T}_R,\\
 \bar{N}_L=\bar{T}_R^{-1}\bar{T}_L, \qquad &\bar{N}_R=\bar{D}_R\bar{T}_R^{-1}\bar{T}_L\bar{D}_R^{-1}.
\end{array}
\end{equation}
We call the matrix $\hat{N}_L$ (respectively $\hat{N}_R$) the \emph{left} (respectively \emph{right}) \emph{transition invariant}. Two transition invariants $\hat{N}_l$ and $\hat{N}'_l$, with $l=L,R$, are equivalent if
\begin{equation}\label{E:equiv e N}
\hat{N}'_l =\hat{K} \hat{N}_l \hat{K}^{-1}
\end{equation}
with $\hat{K}$ as in Definition \ref{D:equiv}.
 \end{definition}

\begin{corollary}[of Lemma \ref{L:exp t}]\label{C:limit matrices}
On $S_\cap$, the difference between a left (respectively right) transition invariant and a left (respectively right) unfolded Stokes matrix is exponentially close to $0$ in $\sqrt{\epsilon}$ in the sense of Definition \ref{D:def expon close}, i.e.
\begin{equation}
\tilde{N}_R-\tilde{C}_R, \quad \bar{N}_L-\bar{C}_L, \quad \tilde{N}_L-\tilde{C}_L, \quad \bar{N}_R-\bar{C}_R.
\end{equation}
\end{corollary}

\begin{remark}\label{R:diag elem of M}
From Corollary \ref{C:limit matrices}, the diagonal entries of the transition invariants $\hat{N}_l$, $l=L,R$, tend to $1$ when $\hat{\epsilon} \to 0$ in $S_\cap$. They are thus always different from zero if the radius $\rho$ of the sector $S$ is sufficiently small.
\end{remark}

\begin{definition}\label{D:auto-intersection}
Let the unfolded Stokes matrices and the formal invariants be given. Let $\hat{T}_l$ as obtained by (\ref{E:monodromy matrix}). We say that the \emph{auto-intersection relation} is satisfied if there exist $\bar{Q}_U$ and $\bar{Q}_D$ nonsingular diagonal matrices depending analytically on $\bar{\epsilon} \in S_\cap$, with a nonsingular limit at $\epsilon=0$, such that
\begin{equation}\label{E:asymp behav Q}
|\bar{Q}_s-I|<c_s|\bar{\epsilon}|, \quad c_s \in \mathbb{R}, \, \bar{\epsilon} \in S_\cap, \, s=D,U,
\end{equation}
and
\begin{equation}\label{E:vraie rela compat}
\bar{Q}_D\bar{D}_R \bar{T}_R^{-1} \bar{T}_L \bar{D}_R^{-1}=\tilde{T}_L^{-1}
\tilde{T}_R \bar{Q}_U,
\end{equation}
which is equivalent to
\begin{equation}\label{E:rel comp L P p}
\bar{Q}_D\bar{N}_l  =  \tilde{N}_l\bar{Q}_U, \quad l=L,R.
\end{equation}
because of (\ref{E:dl tilde et bar}) and Definition \ref{D:Nl tilde}.
\end{definition}

\begin{proposition}\label{P:relation compat prop}
The auto-intersection relation (\ref{E:rel comp L P p}) for the family (\ref{E:prenormal system k1}) is satisfied.
\end{proposition}

\begin{proof}
We proceed similarly as the proof of Proposition \ref{P:transit inv}, taking
\renewcommand{\theenumi}{\alph{enumi}}
\begin{enumerate}
\item \label{I:premier item}$W_U(\bar{\epsilon},x)\bar{D}_R \bar{T}_L \bar{D}_R^{-1}$  and $W_U(\tilde{\epsilon},x)\tilde{D}_R\tilde{T}_R\tilde{D}_R^{-1}$ as the fundamental matrices of solutions composed of eigenvectors of $M^*_{\bar{x}_L}$ (to verify, use (\ref{E:F ramif}), (\ref{E: def C_Re}) and (\ref{E:Stokes and Txl})),
\item \label{I:deuxieme item}$W_D(\bar{\epsilon},x)\bar{T}_R $  and $W_D(\tilde{\epsilon},x)\tilde{T}_L$ as the fundamental matrices of solutions composed of eigenvectors of $M^*_{\bar{x}_R}$,
\end{enumerate}
with $W_s(\epsilon,x)$ given by (\ref{E:equation W,H,F}). By Lemma \ref{L:exp t} and Corollary \ref{C:uniform convergence}, these solutions converge uniformly to $W_s(0,x)$ (defined by (\ref{E:confluent solution})) on compact sets of $\Omega_s^0$ when $\bar{\epsilon} \to 0$, $\bar{\epsilon} \in S_\cap$, for $s=D$ or $s=U$. The corresponding transition matrices are here given by
\begin{equation}
 E_{L,\tilde{x}_L \to \tilde{x}_R}=\tilde{N}_Le^{2\pi i \Lambda_1(\epsilon)}, \qquad E_{R,\tilde{x}_L \to \tilde{x}_R}=\tilde{N}_R,
 \end{equation}
\begin{equation}
 E_{L,\bar{x}_R \to \bar{x}_L}=\bar{N}_Le^{2\pi i \Lambda_1(\epsilon)},\qquad E_{R,\bar{x}_R \to \bar{x}_L}=\bar{N}_R,
\end{equation}
leading to (\ref{E:rel comp L P p}).

Let us now prove (\ref{E:asymp behav Q}) for $s=D$ (the case $s=U$ is similar). We have obtained the existence of nonsingular diagonal matrices $\bar{Q}_U$ and $\bar{Q}_D$ depending analytically on $\bar{\epsilon} \in S_\cap$, with a nonsingular limit at $\epsilon=0$, such that
\begin{equation}
W_U(\bar{\epsilon},x)\bar{D}_R \bar{T}_L \bar{D}_R^{-1}=W_U(\tilde{\epsilon},x)\tilde{D}_R\tilde{T}_R\tilde{D}_R^{-1} \bar{Q}_U
\end{equation}
and
\begin{equation}
W_D(\bar{\epsilon},x)\bar{T}_R=W_D(\tilde{\epsilon},x)\tilde{T}_L \bar{Q}_D.
\end{equation}
Extending the solution $W_D(\bar{\epsilon},x)\bar{T}_R$ (respectively $W_D(\tilde{\epsilon},x)\tilde{T}_L$) to $x=0$ along a path in $\Omega_D^{\bar{\epsilon}}$ (respectively $\Omega_D^{\tilde{\epsilon}}$), we obtain
\begin{equation}
W_D(\bar{\epsilon},0) \bar{T}_R =W_D(\tilde{\epsilon},0)\tilde{T}_L \bar{Q}_D \bar{D}_R
\end{equation}
or equivalently, because of (\ref{E:equation W,H,F}),
\begin{equation}
H_D(\bar{\epsilon},0)F_D(\bar{\epsilon},0) \bar{T}_R =H_D(\tilde{\epsilon},0)F_D(\tilde{\epsilon},0)\tilde{T}_L \bar{Q}_D \bar{D}_R.
\end{equation}
Since $F_D(\bar{\epsilon},0)=F_D(\tilde{\epsilon},0)\bar{D}_R$, we have
\begin{equation}
H_D(\bar{\epsilon},0)F_D(\bar{\epsilon},0) \bar{T}_R F_D(\bar{\epsilon},0)^{-1}=H_D(\tilde{\epsilon},0)F_D(\tilde{\epsilon},0)\tilde{T}_L F_D(\tilde{\epsilon},0)^{-1}  \bar{Q}_D.
\end{equation}
Using (\ref{E:mat exp small 2}), we have that $F_D(\bar{\epsilon},0) \bar{T}_R F_D(\bar{\epsilon},0)^{-1}$ and $F_D(\tilde{\epsilon},0)\tilde{T}_L F_D(\tilde{\epsilon},0)^{-1}$ are exponentially close in $\sqrt{\epsilon}$ to $I$. We use (\ref{E:diff Hs0bartilde}) and (\ref{E:Hs0bounded}) in order to obtain (\ref{E:asymp behav Q}) for $s=D$.
\end{proof}

We now present an important consequence to the auto-intersection relation:

\begin{theorem}\label{T:summ}
There exists a representative of the equivalence class of unfolded Stokes matrices which is $\frac{1}{2}$-summable in $\epsilon$. It is defined over $S$ with a slight reduction of the opening and radius.
\end{theorem}

\begin{proof}
The strategy consists in using the Ramis-Sibuya Theorem (see for instance \cite{jR04}): if $C(\hat{\epsilon})$ depends analytically on $\hat{\epsilon}$ on a ramified sector around the origin and if the difference on the auto-intersection of the sector is exponentially close to $0$ in $\sqrt{\epsilon}$, i.e. $|C(\bar{\epsilon})-C(\tilde{\epsilon})|<B e^{-\frac{A}{\sqrt{|\epsilon|}}}$ for some positive $A$ and $B$, then $C(\epsilon)$ is $\frac{1}{2}$-summable in $\epsilon$.

By Proposition \ref{P:relation compat prop}, the auto-intersection relation (\ref{E:rel comp L P p}) is satisfied. Hence,
\begin{equation}\label{E:rel comp L P}
\bar{Q}_D\bar{N}_l  =  \tilde{N}_l\bar{Q}_D\bar{Q},
\end{equation}
with $\bar{Q}=\bar{Q}_D^{-1}\bar{Q}_U$. We then have
\begin{equation}\label{E:rel P 2}
(\bar{N}_l)_{ii}  =  (\tilde{N}_l)_{ii} (\bar{Q})_{ii},
\end{equation}
Corollary \ref{C:limit matrices} says that $\bar{N}_l$ (respectively $\tilde{N}_l$) is exponentially close in $\sqrt{\epsilon}$ to $\bar{C}_l$ (respectively $\tilde{C}_l$). Since the unfolded Stokes matrices has $1$'s on the diagonal, relation (\ref{E:rel P 2}) implies that $\bar{Q}$ is exponentially close (in $\sqrt{\epsilon}$) to $I$. If the opening and the radius of $S$ are reduced slightly, there exists $\hat{K}$ a nonsingular diagonal matrix depending analytically on $\hat{\epsilon} \in S$, with a nonsingular limit at $\epsilon=0$, such that $\tilde{K}^{-1}\bar{K}=\bar{Q}_D$ (recall that (\ref{E:asymp behav Q}) is satisfied). Relation (\ref{E:rel comp L P}) becomes
\begin{equation}\label{E:rela N K Q}
\bar{K}\bar{N}_l \bar{K}^{-1}=  \tilde{K}\tilde{N}_l \tilde{K}^{-1}\bar{Q},
\end{equation}
since $\bar{Q}$ is diagonal, and hence commutes with $\bar{K}$. Let us take the representative of the equivalence class of unfolded Stokes matrices $\hat{C}'_l=\hat{K}\hat{C}_l\hat{K}^{-1}$. Using Corollary \ref{C:limit matrices} with \begin{equation}\hat{N}'_l=\hat{K}\hat{N}_l\hat{K}^{-1},\end{equation} we obtain that $\bar{N}_l'$ (respectively $\tilde{N}_l'$) is exponentially close to $\bar{C}'_l$ (respectively $\tilde{C}'_l$). On the other hand, relation (\ref{E:rela N K Q}) implies
\begin{equation}\label{E:rela N K Q 2}
\bar{N}_l'= \tilde{N}_l'\bar{Q}
\end{equation}
with $\bar{Q}$ exponentially close in $\sqrt{\epsilon}$ to $I$. The difference between the representatives $\bar{C}'_l$ and $\tilde{C}'_l$ is hence exponentially close to $0$ in $\sqrt{\epsilon}$, for $l=L,R$.
\end{proof}

\begin{remark}
In dimension $n=2$, it is always possible to choose an analytic representative of the equivalence classes of unfolded Stokes matrices. All the cases have been enumerated in \cite{cRcC}. Indeed, in the case of nonvanishing elements $(\hat{C}_L)_{21}$ and $(\hat{C}_R)_{12}$, the auto-intersection relation is equivalent to the analyticity of the product $(\hat{C}_L)_{21}(\hat{C}_R)_{12}$. Preliminary investigation in the case $n=3$ shows that this could not be the case generically. We study this in more details in \cite{cLR3}.
\end{remark}

\subsection{Unfolded Stokes matrices reducible in block diagonal form}
We will now state a sufficient condition for the decomposition of a system (\ref{E:prenormal system k1}) in dimension $n$ as the direct product of irreducible systems of lower dimension (this may require a permutation), using the following lemma.

\begin{lemma}\label{L:bloc diagonal form}
For $\hat{\epsilon} \in S_\cap$, the matrix $P^{-1}\hat{C}_LP$ (respectively $P^{-1}\hat{C}_RP$), with a permutation matrix $P$, is lower (respectively upper) triangular, unipotent and in a block diagonal form if and only if $P^{-1}\hat{T}_L P$ (respectively $P^{-1}\hat{T}_R P$) has the same form.

$\bar{C}_R$ and $\bar{C}_L$ have a common block diagonal form with the same permutation matrix $P$ (when staying triangular) if and only if $\tilde{C}_R$ and $\tilde{C}_L$ have the same block diagonal form with the same permutation matrix $P$ (and stay triangular).
\end{lemma}

\begin{proof}
The first assertion comes from the fact that columns of $P^{-1}\hat{T}_lP$ are eigenvectors of $P^{-1}\hat{C}_l\hat{D}_lP$ (note that there are no resonances for $\hat{\epsilon}$ in $S_\cap$). Let us prove the converse. $P^{-1}\hat{T}_lP$ is unipotent, triangular and in a block diagonal form if and only if $P^{-1}\hat{D}_l\hat{T}_l\hat{D}_l^{-1}P$ has the same structure with the same permutation matrix $P$. Then, the product $(P^{-1}\hat{T}_lP) (P^{-1}\hat{D}_l\hat{T}_l\hat{D}_l^{-1}P)^{-1}=P^{-1}\hat{C}_lP$ (by (\ref{E:Stokes and Txl})) has the desired property.

The second assertion follows directly from (\ref{E:vraie rela compat}) and from the first assertion.
\end{proof}

\begin{theorem}\label{T:decomp sum systems}
Let us take any family of systems (\ref{E:prenormal system k1}) with both unfolded Stokes matrices admitting, after conjugation (if necessary) by the same permutation matrix $P$ preserving their triangular form, the same decomposition in diagonal blocks for all $\hat{\epsilon} \in S$ (i.e. $P^{-1}\hat{C}_lP=\hat{c}^l_{n_1} \oplus \hat{c}^l_{n_2} \oplus ... \oplus \hat{c}^l_{n_k}$ for $l=L,R$, with $n_1+n_2+...+n_k=n$). This family of systems is analytically equivalent (with permutation $P$) to the direct product of families of systems.
\end{theorem}

\begin{proof}
First, let us take a system (\ref{E:prenormal system k1}) which has unfolded Stokes matrices in block diagonal form with the same positions of the blocks: $\hat{C}_l=\hat{c}^l_{n_1} \oplus \hat{c}^l_{n_2} \oplus ... \oplus \hat{c}^l_{n_k}$ for $l=L,R$, with $n_1+n_2+...+n_k=n$. We will prove that this system  is analytically equivalent to a direct product of smaller systems of dimensions $n_1$, ..., $n_k$. Looking at (\ref{E: def C_Re}), we notice that these relations would still hold if we replace by zero each element $(H_s(\hat{\epsilon},x))_{ij}$ such that the position $(i,j)$ is outside the diagonal blocks of  $\hat{C}_l$. This leads us to define $J_s(\hat{\epsilon},x)$, for $x \in \Omega_s^{\hat{\epsilon}}$, by
\begin{equation}
(J_s(\hat{\epsilon},x))_{ij}=\begin{cases}\begin{array}{lll}0, \quad &\mbox{if $(i,j)$ lies outside the diagonal blocks}, \\(H_s(\hat{\epsilon},x))_{ij}, \quad &\mbox{otherwise}.\end{array}
\end{cases}
\end{equation}
$J_s(\hat{\epsilon},x)$ is in block diagonal form $J_{s,n_1}(\hat{\epsilon},x) \oplus J_{s,n_2}(\hat{\epsilon},x) \oplus ... \oplus J_{s,n_k}(\hat{\epsilon},x)$ and it follows from (\ref{E:limit Hs}) that it is invertible. From (\ref{E: def C_Re}), we have
\begin{equation}\label{E:J et Stokes}
J_D(\hat{\epsilon}, x)^{-1}J_U(\hat{\epsilon}, x)=\begin{cases}\begin{array}{lll}F_D(\hat{\epsilon},x) \hat{C}_R(F_D(\hat{\epsilon},x))^{-1}, &\quad  \mbox{\rm{ on} } \Omega_R^{\hat{\epsilon}},
\\ F_D(\hat{\epsilon},x) \hat{C}_L(F_D(\hat{\epsilon},x))^{-1}, &\quad  \mbox{\rm{ on} } \Omega_L^{\hat{\epsilon}},\\I, &\quad \mbox{\rm{ on} } \Omega_C^{\hat{\epsilon}}.\end{array}\end{cases}
\end{equation}
These relations imply that the transformation
\begin{equation}
\mathcal{Q}(\hat{\epsilon},x)=\begin{cases}J_D(\hat{\epsilon},x)H_D(\hat{\epsilon},x)^{-1}, \quad x \in \Omega_D^{\hat{\epsilon}},  \\J_U(\hat{\epsilon},x)H_U(\hat{\epsilon},x)^{-1}, \quad x \in \Omega_U^{\hat{\epsilon}},
\end{cases}
\end{equation}
is well-defined on the intersections of the domains and is an analytic function of $x$ in a whole neighborhood of $x=0$, including the points $\hat{x}_R$ and $\hat{x}_L$. We will now prove that $\mathcal{Q}(\hat{\epsilon},x)$ is unramified in $\epsilon$. Since it is bounded at $\epsilon=0$, this will imply the analyticity of $\mathcal{Q}(\epsilon,x)$ at $\epsilon=0$. To prove that
\begin{equation}\label{E: Q non ramif}
\mathcal{Q}(\tilde{\epsilon},x)=\mathcal{Q}(\bar{\epsilon},x),
\end{equation}
i.e.
\begin{equation}\label{E:Q non ram}
J_s(\tilde{\epsilon},x)^{-1}J_s(\bar{\epsilon},x)=H_s(\tilde{\epsilon},x)^{-1}H_s(\bar{\epsilon},x), \quad s \in \{1,2\},
\end{equation}
we will consider $x \in \Omega_C^{\tilde{\epsilon}} \cap \Omega_C^{\bar{\epsilon}}$. In this region, we have $J_U(\hat{\epsilon}, x)=J_D(\hat{\epsilon}, x)$ and $H_U(\hat{\epsilon}, x)=H_D(\hat{\epsilon}, x)$. By uniqueness of the Floquet solutions (Theorem \ref{T:tout sur bases vp}), we have
\begin{equation}\label{E: proportion vecteur propres}
H_D(\bar{\epsilon},x)F_D(\bar{\epsilon},x)\bar{T}_RK=H_D(\tilde{\epsilon},x)F_D(\tilde{\epsilon},x)\tilde{T}_{L}
\end{equation}
with $K$ a nonsingular diagonal matrix. Hence,
\begin{equation}\label{E:eq Z}
H_D(\bar{\epsilon},x)F_D(\bar{\epsilon},x)=H_D(\tilde{\epsilon},x)F_D(\tilde{\epsilon},x)Z,
\end{equation}
with $Z=\tilde{T}_{L}K^{-1}\bar{T}_R^{-1}$. By Lemma \ref{L:bloc diagonal form}, $Z$ is in the block diagonal form $Z_{n_1} \oplus Z_{n_2} \oplus ... \oplus Z_{n_k}$. By definition of $J_D(\hat{\epsilon},x)$, we have
\begin{equation}\label{E:eq Z 2}
J_D(\bar{\epsilon},x)F_D(\bar{\epsilon},x)=J_D(\tilde{\epsilon},x)F_D(\tilde{\epsilon},x)Z.
\end{equation}
Relations (\ref{E:eq Z}) and (\ref{E:eq Z 2}) yield (\ref{E:Q non ram}). Finally, $\lim_{\epsilon \to 0}\mathcal{Q}(\epsilon,x)$ is bounded, so $\mathcal{Q}(\epsilon,x)$ is an analytic function of $(\epsilon,x)$ in a whole neighborhood of $(0,0)$. The transformation $v=\mathcal{Q}(\epsilon,x)y$ gives a system with the fundamental matrix of solutions $J_s(\hat{\epsilon},x)F_s(\hat{\epsilon},x)$ on $\Omega_s^{\hat{\epsilon}} $, and hence with the matrix in block diagonal form $B(\epsilon,x)=B_{n_1}(\epsilon,x) \oplus B_{n_2}(\epsilon,x) \oplus ... \oplus B_{n_k}(\epsilon,x)$.

Finally, let us take a system (\ref{E:prenormal system k1}) in which the unfolded Stokes matrices conjugated by a permutation matrix have the same decomposition in diagonal blocks. We apply the previous result to the system transformed by $y \mapsto P y$.
\end{proof}

\subsection{Unfolded Stokes matrices with trivial rows or column}\label{S:trivial rows columns}
We include here the study of the cases when both unfolded Stokes matrices have a trivial row or column (see Notation \ref{N:trivial row column}). When this happens, the system is analytically equivalent to a simpler one. This section is not a prerequisite to obtain the complete system of invariants of the systems (\ref{E:prenormal system k1}).

\begin{lemma}\label{L:trivital columns}
For $\hat{\epsilon}$ in $S_\cap$ and $j \in \{1,2,...,n\}$, the following properties are equivalent, and they are satisfied for $\bar{\epsilon}$ if and only if they are satisfied for $\tilde{\epsilon}$:

\begin{enumerate}
\item the $j^{th}$ solution that is eigenvector of the monodromy operator around $x=\hat{x}_R$ is a multiple of the $j^{th}$ solution that is eigenvector of the monodromy operator around $x=\hat{x}_L$;

\item the $j^{th}$ column of the transition invariants $\hat{N}_L$ and $\hat{N}_R$ (Definition \ref{D:Nl tilde}) is trivial;

\item the $j^{th}$ columns of $\hat{T}_R$ and $\hat{T}_L$ are trivial;

\item the $j^{th}$ columns of $\hat{C}_R$ and $\hat{C}_L$ are trivial;

\item the solution $w_j(\hat{\epsilon},x)$, corresponding to the $j^{th}$ column of $W_V(\hat{\epsilon},x)$ given by (\ref{E:basis union}), is eigenvector of the monodromy around both singular points.

\end{enumerate}
\end{lemma}

\begin{proof}
It follows from (\ref{E:rel comp L P p}).
\end{proof}

\begin{theorem}\label{T:trivial columns theorem}
A family of systems (\ref{E:prenormal system k1}) with both unfolded Stokes matrices having the $j^{th}$ column  trivial for all $\hat{\epsilon} \in S$ is analytically equivalent to a family of system (\ref{E:prenormal system k1}) with an invariant subsystem formed by the equations $i \ne j$ (i.e. the $(i,j)$ entries are null for all $i \ne j$).
\end{theorem}

\begin{proof}
We follow the same steps as in the proof of Theorem \ref{T:decomp sum systems}, considering the $j^{th}$ column with nondiagonal entries null (instead of null entries outside diagonal blocks in Theorem \ref{T:decomp sum systems}), and taking a different definition of $J_s(\hat{\epsilon},x)$. We take $J_s(\hat{\epsilon},x)=\mathcal{Q}_s(\hat{\epsilon},x)H_s(\hat{\epsilon},x)$, with
\begin{equation}
(\mathcal{Q}_s(\hat{\epsilon},x))_{ik}=\begin{cases}\begin{array}{lll}1, \quad &\mbox{if } i=k, \\\frac{-(H_s(\hat{\epsilon},x))_{ij}}{(H_s(\hat{\epsilon},x))_{jj}}, \quad &\mbox{if } k=j, \, i \ne k,\\0, \quad &\mbox{otherwise}.\end{array}
\end{cases}
\end{equation}

The $j^{th}$ column of $J_s(\hat{\epsilon},x)$ then has zero nondiagonal entries. The rest follows as in the proof of Theorem \ref{T:decomp sum systems}, using Lemma \ref{L:trivital columns} instead of Lemma \ref{L:bloc diagonal form} (and forgetting about the last part of the proof about the permutation of the $y$-coordinates).
\end{proof}

\begin{lemma}\label{L:trivital rows}
For $\hat{\epsilon}$ in $S_\cap$ and $j \in \{1,2,...,n\}$, the following properties are equivalent, and they are satisfied for $\bar{\epsilon}$ if and only if they are also for $\tilde{\epsilon}$:

\begin{enumerate}
\item \label{I:1a} the $j^{th}$ row of the transition invariants $\hat{N}_L$ and $\hat{N}_R$ is trivial;

\item \label{I:2a} the $j^{th}$ rows of $\hat{T}_R$ and $\hat{T}_L$ are trivial;

\item \label{I:3a} the $j^{th}$ rows of $\hat{C}_R$ and $\hat{C}_L$ are trivial.

\end{enumerate}
Hence, in the $j^{th}$ Riccati system, the property of a first integral $\mathcal{H}_q^j$ to be an eigenvector of the monodromy around both singular points is conserved in both points of view $\bar{\epsilon}$ and $\tilde{\epsilon}$.
\end{lemma}

\begin{proof}
The first part follows from (\ref{E:rel comp L P p}). The last part comes from Corollary \ref{C:monodromy integral}: a first integral $\mathcal{H}_q^j$ is eigenvector of the monodromy around both singular points if and only if rows $q$ and $j$ of the inverse of two unfolded Stokes matrices are trivial.
\end{proof}

\begin{theorem}\label{T:trivial rows theorem}
A family of systems (\ref{E:prenormal system k1}) with both unfolded Stokes matrices having the $j^{th}$ row trivial for all $\hat{\epsilon} \in S$ is analytically equivalent to a family of system (\ref{E:prenormal system k1}) where the $j^{th}$ equation is independent of the others, hence integrable (i.e. the $(j,i)$ entries are null for all $i \ne j$).
\end{theorem}

\begin{proof}
The proof of the analytic equivalence (to a system having $(j,i)$ entries null for all $i \ne j$ with $j$ fixed) is very similar to the proof of Theorem \ref{T:decomp sum systems}, considering the $j^{th}$ row with nondiagonal entries null (instead of null entries outside diagonal blocks in Theorem \ref{T:decomp sum systems}), and taking a different definition of $J_s(\hat{\epsilon},x)$, namely
\begin{equation}
(J_s(\hat{\epsilon},x))_{ik}=\begin{cases}\begin{array}{lll}(H_s(\hat{\epsilon},x))_{ik}, \quad &i \ne j,\\
0, \quad &i=j,\, \ne j \\
1, \quad &i=j=k.
\end{array}
\end{cases}
\end{equation}
We then follow the proof of Theorem \ref{T:decomp sum systems}, using Lemma \ref{L:trivital rows} instead of Lemma \ref{L:bloc diagonal form} and forgetting about the last section of the proof that concerns permutation.
\end{proof}

\subsection{Analytic invariants}\label{S:analytic invariant}
We now have the tools to prove that the equivalent unfolded Stokes collections are analytic invariants for the classification of the systems (\ref{E:prenormal system k1}).

\begin{theorem}\label{T:analy equiv}
Two families of systems of the form (\ref{E:prenormal system k1}) with the same model system (\ref{E:model system k1}) are analytically equivalent if and only if their unfolded Stokes collections are equivalent. In particular, a family (\ref{E:prenormal system k1}) is analytically equivalent to its model if and only if the unfolded Stokes matrices are the identity.
\end{theorem}

\begin{proof}
We consider two systems of the form (\ref{E:prenormal system k1}):
\begin{equation}\label{E:systems to classify}
(x^2-\epsilon)y_{i}'=B_i(\epsilon,x)y_{i},
\end{equation}
with
\begin{equation}
B_i(\epsilon,x)=\Lambda(\epsilon,x)+(x^2-\epsilon)R_{i}(\epsilon,x), \quad i=1,2,
\end{equation}
and $\Lambda(\epsilon,x)$ given by (\ref{E: C(e,x) model k1}). We choose a neighborhood of the origin $\mathbb{D}_r$ common to the two systems for which the modulus is defined. We denote the fundamental matrix of solutions of (\ref{E:systems to classify}) given by Theorem \ref{T:fundamental matrix e} as $H_{i,s}(\hat{\epsilon},x)F_s(\hat{\epsilon},x)$ (for $(\hat{\epsilon},x) \in S \times \Omega_s^{\hat{\epsilon}}$, $s=D,U$).

Let us suppose that these two systems are analytically equivalent via a transformation $y_{2}=\mathcal{Q}(\epsilon,x)y_{1}$. By Proposition \ref{P:fundamental matrix}, we must have
\begin{equation}
H_{2,s}(\hat{\epsilon} ,x)= \mathcal{Q}(\epsilon ,x)H_{1,s}(\hat{\epsilon} ,x)\hat{K} \quad \mbox{on } \Omega_s^{\hat{\epsilon}}, \quad s=D,U,
\end{equation}
with $\hat{K}$ a nonsingular diagonal matrix depending analytically on $\hat{\epsilon} \in S$ with a nonsingular limit at $\epsilon=0$ and such that (\ref{E:condition Ktildebar}) is satisfied. Then, on the intersections of $\Omega_D^{\hat{\epsilon}}$ and $\Omega_U^{\hat{\epsilon}}$, we have
\begin{equation}
(H_{2,D}(\hat{\epsilon} ,x))^{-1}H_{2,U}(\hat{\epsilon} ,x)=\hat{K}^{-1}(H_{1,D}(\hat{\epsilon} ,x))^{-1}H_{1,U}(\hat{\epsilon} ,x)\hat{K}.
\end{equation}
This implies that the unfolded Stokes collections given by (\ref{E: def C_Re}) are equivalent.

Let us prove the other direction. Let us suppose that the two systems above have equivalent Stokes collections $\{\hat{C}^i_R,\hat{C}^i_L\}$ with a matrix $\hat{K}$ as in Definition \ref{D:equiv}, i.e.
\begin{equation}\label{E:relation Stokes 1 et 2 p}
\hat{C}^{2}_l =\hat{K} \hat{C}^{1}_l \hat{K}^{-1}, \quad l=L,R.
\end{equation}
By taking, for the second system, an adequate normalization of the fundamental matrix of solutions (namely changing from $H_{2,s}(\hat{\epsilon},x)F_s(\hat{\epsilon},x)$ to $H_{2,s}(\hat{\epsilon},x)F_s(\hat{\epsilon},x)\hat{K}$, $s=D,U$), we can, without loss of generality, suppose that
\begin{equation}\label{E:relation Stokes 1 et 2}
\hat{C}^{2}_l =\hat{C}^{1}_l, \quad l=L,R.
\end{equation}
First, let us suppose that the unfolded Stokes matrices $\hat{C}^{i}_R$ and $\hat{C}^{i}_L$ cannot have a block diagonal form (for all $\hat{\epsilon} \in S$) with the same positions of the blocks, neither after conjugation of each of them by the same permutation matrix (that keeps their triangular form). We take
\begin{equation}\label{E:definition G(e,x)}
\mathcal{Q}(\hat{\epsilon} ,x)= \begin{cases}H_{2,D}(\hat{\epsilon} ,x)(H_{1,D}(\hat{\epsilon} ,x))^{-1}, \quad \mbox{on } \Omega_D^{\hat{\epsilon}}, \\
H_{2,U}(\hat{\epsilon} ,x)(H_{1,U}(\hat{\epsilon} ,x))^{-1}, \quad \mbox{on } \Omega_U^{\hat{\epsilon}},
                \end{cases}
\end{equation}
which is well-defined because of (\ref{E:relation Stokes 1 et 2}). Since $\lim_{x \to \hat{x}_l}\mathcal{Q}(\hat{\epsilon},x)$ is bounded, invertible and independent of $s$ for $l=L,R$ (see (\ref{E:limit Hs})), $\mathcal{Q}(\hat{\epsilon} ,x)$ is an analytic function of $x$ on the whole neighborhood $\mathbb{D}_r$ of $x=0$ which includes the points $\hat{x}_R$ and $\hat{x}_L$, for $\hat{\epsilon} \in S$. We will now choose carefully $\hat{\eta}$ such that $\hat{\eta}\mathcal{Q}(\hat{\epsilon} ,x)$ becomes analytic at $\epsilon=0$. We will prove that $\hat{\eta}\mathcal{Q}(\hat{\epsilon} ,x)$ is uniform in $\epsilon$ and bounded near $\epsilon=0$. The transformation $\mathcal{Q}(\bar{\epsilon} ,x)\mathcal{Q}(\tilde{\epsilon} ,x)^{-1}$ is an automorphism of the second family of systems (\ref{E:systems to classify}). Hence, over each domain $\Omega_s^{\bar{\epsilon}}$, $s=D,U$, we have the following automorphism of the model
\begin{equation}\label{E:automo}
(H_{2,s}(\bar{\epsilon},x))^{-1}\mathcal{Q}(\bar{\epsilon} ,x)\mathcal{Q}(\tilde{\epsilon} ,x)^{-1}H_{2,s}(\bar{\epsilon},x)=\bar{D}_s,
\end{equation}
giving $\bar{D}_s$ a diagonal matrix depending on $\bar{\epsilon}$. With relations (\ref{E: def C_Re}) applied to the second system, (\ref{E:automo}) leads to
\begin{equation}\label{E:relation C D}
\bar{C}^2_l\bar{D}_U=\bar{D}_D\bar{C}^2_l, \quad l \in \{R,L \}.
\end{equation}
As the diagonal entries of $\bar{C}_l$ are $1$'s, we have $\bar{D}_U=\bar{D}_D$. The hypothesis that the Stokes matrices have no common reduction to block diagonal form (neither after conjugation by a permutation matrix that keeps their triangular form) implies that this relation can only be satisfied for $\bar{D}_U=\bar{\mu}I$ for some $\bar{\mu}$ analytic function over $S_\cap$. Relation (\ref{E:automo}) becomes
\begin{equation}\label{E:automo D 2}
\mathcal{Q}(\bar{\epsilon} ,x)\mathcal{Q}(\tilde{\epsilon} ,x)^{-1}=\bar{\mu}I.
\end{equation}
In particular,
\begin{equation}\label{E:automo D 20}
\mathcal{Q}(\bar{\epsilon} ,0)\mathcal{Q}(\tilde{\epsilon} ,0)^{-1}=\bar{\mu}I.
\end{equation}
Using properties (\ref{E:diff Hs0bartilde}) and (\ref{E:Hs0bounded})(which remained valid when we modified $H_{2,s}(\hat{\epsilon},x)$ to $H_{2,s}(\hat{\epsilon},x)\hat{K}$, $s=D,U$), the definition (\ref{E:definition G(e,x)}) implies there exists $C \in \mathbb{R}_+$ such that
\begin{equation}\label{E:automo D 3}
|\mathcal{Q}(\bar{\epsilon} ,0)\mathcal{Q}(\tilde{\epsilon} ,0)^{-1}-I| \leq C|\bar{\epsilon}|, \quad \bar{\epsilon} \in S_\cap.
\end{equation}
Relation (\ref{E:automo D 20}) and (\ref{E:automo D 3}) imply there exists $c \in \mathbb{R}_+$ such that
\begin{equation}\label{E:mu asympto}
|\bar{\mu}-1|\leq c|\bar{\epsilon}|, \quad \bar{\epsilon} \in S_\cap.
\end{equation}
Reducing slightly the radius $\rho$ of $S$ and its opening, let $\hat{\eta}$ be an analytic function of $\hat{\epsilon}$ on $S$ satisfying
\begin{equation}\label{E: eta mu}
\bar{\eta}^{-1}\tilde{\eta}=\bar{\mu}.
\end{equation}
Of course, such a function can be found with $\lim_{\hat{\epsilon} \to 0}\hat{\eta}=1$. Let $\mathcal{Q}^*(\hat{\epsilon} ,x)=\hat{\eta}\mathcal{Q}(\hat{\epsilon},x)$. From (\ref{E:automo D 2}) and (\ref{E: eta mu}), we get
\begin{equation}
\mathcal{Q}^*(\bar{\epsilon} ,x)=\mathcal{Q}^*(\tilde{\epsilon} ,x).
\end{equation}
Then,
\begin{equation}
\lim_{\epsilon \to 0}\mathcal{Q}^*(\epsilon ,x)=H_{2,s}(0,x)(H_{1,s}(0,x))^{-1}, \quad x \in \Omega_s^0, \, s=D,U,
\end{equation}
which is finite, so $\mathcal{Q}^*(\epsilon ,x)$ is analytic in $\epsilon$ at $\epsilon=0$. Hence, $\mathcal{Q}^*(\epsilon,x)$ analytically conjugates the two systems.

Finally, let us suppose that both unfolded Stokes matrices of each system admit, after conjugation if necessary by the same permutation matrix that keeps their triangular form, the same maximal decomposition in diagonal blocks for all $\hat{\epsilon} \in S$. By Theorem \ref{T:decomp sum systems}, each system is analytically equivalent (with permutation $P$) to a system decomposed in smaller indecomposable systems. The decomposed systems have equivalent unfolded Stokes collections and the smaller indecomposable systems too. By applying the former argument to each pair of indecomposable systems, we find that they are analytically equivalent. Hence, the two decomposed systems are analytically equivalent, and so are the initial systems.
\end{proof}

\section{Realization of the analytic invariants} \label{S:realization}

By Section \ref{S:Complete System}, the complete system of analytic invariants for the systems (\ref{E:prenormal system k1}) consists of the formal invariants (the model system) and an equivalence class of unfolded Stokes matrices. In this section, we give the realization theorem for these invariants by proceeding in two steps. First, we consider the local realization :
\begin{theorem}\label{T:reala loca}
Let a complete system of analytic invariants be given:
\begin{itemize}
\item a model system (i.e. formal invariants $\lambda_{j,q}(\epsilon)$, $j=1,2,...,n$, $q=0,1$, depending analytically on $\epsilon$ at the origin),
\item an equivalence class (see Definition \ref{D:equiv}) of unfolded Stokes matrices $\hat{C}_R$ and $\hat{C}_L$, which are respectively an upper triangular and a lower triangular unipotent matrix depending analytically on $\hat{\epsilon} \in S$ and having a bounded limit when $\hat{\epsilon} \to 0$ on $S$ (the sector $S$ of radius $\rho_0$ and of opening $2 \pi + \gamma_0$ is chosen from the formal invariants as in Section \ref{S:sectors e}, and $\rho_0$ can obviously be chosen smaller to ensure the analyticity, over $S$, of the entries of $\hat{C}_R$ and $\hat{C}_L$).
\end{itemize}
Then, there exist $r>0$, a radius $\rho < \min\{\rho_0,\frac{r²}{2}\}$ of $S$ and a system $(x^2-\epsilon)y'=A(\hat{\epsilon},x)y$ ($y \in \mathbb{C}^n$) characterized by these analytic invariants, with $A(\hat{\epsilon},x)$ analytic over $S \times \mathbb{D}_r$. The limit of $A(\hat{\epsilon},x)$ when $\hat{\epsilon} \to 0$ ($\hat{\epsilon} \in S$) is analytic in $x$ over $\mathbb{D}_r$.
\end{theorem}
We prove Theorem \ref{T:reala loca} from Sections \ref{S:th local deb} to \ref{S:const H}. Then, we show that the auto-intersection relation (\ref{E:vraie rela compat}) is sufficient for the global realization of the analytic invariants, i.e.:
\begin{theorem}\label{T:reala globa}
Let a complete system of analytic invariants as described in Theorem \ref{T:reala loca} be given and satisfying the auto-intersection relation (\ref{E:vraie rela compat}). Then, there exist $r>0$, a radius $\rho < \min\{\rho_0,\frac{r²}{2}\}$ of $S$ and a system $(x^2-\epsilon)y'=B(\epsilon,x)y$ ($y \in \mathbb{C}^n$) characterized by these analytic invariants, with $B(\epsilon,x)$ analytic over $\mathbb{D}_\rho \times \mathbb{D}_r$.
\end{theorem}
The proof of Theorem \ref{T:reala globa} is presented from Sections \ref{S:th global deb} to \ref{S:const J}. It uses the  ramified system constructed in the proof of Theorem \ref{T:reala loca}. The auto-intersection relation (\ref{E:vraie rela compat}) will be the key ingredient to prove Theorem \ref{T:reala globa}, namely to correct the family to a uniform family. It will guarantee the triviality of the abstract vector bundle realizing the family of Stokes matrices.

\subsection{Introduction to the proof of Theorem \ref{T:reala loca}}\label{S:th local deb}

Considering $\hat{\epsilon}$ fixed, we realize the invariants on an abstract vector bundle which we then show to be trivial. For this, using ideas from the proof of the realization theorem at $\epsilon=0$ in \cite{yS90} (p.~150) and from the proof of Cartan's Lemma in \cite{GuRo} (p.~199), we will prove that, for $s=D,U$ and sufficiently small radii $\rho$ of $S$ and $r$ of $\Omega_s^{\hat{\epsilon}}$, there exist matrices $H_s(\hat{\epsilon},x)$ depending analytically on $(\hat{\epsilon},x) \in S \times \Omega_s^{\hat{\epsilon}}$, having a limit when $\hat{\epsilon} \to 0$ in $S$ that is analytic in $x$ over $\Omega_s^{0}$, and such that, for $\hat{\epsilon} \in S \cup \{0\}$,
\begin{equation}\label{E:relation H to be sat}
H_D(\hat{\epsilon}, x)^{-1}H_U(\hat{\epsilon}, x)=I+Z(\hat{\epsilon},x), \quad x \in \Omega_U^{\hat{\epsilon}} \cap \Omega_D^{\hat{\epsilon}},\end{equation}
where
\begin{equation}\label{E:def Z1}
Z(\hat{\epsilon},x)=\begin{cases}\begin{array}{lll}F_D(\hat{\epsilon},x) \hat{C}_RF_D(\hat{\epsilon},x)^{-1}-I  \quad  &\mbox{\rm{ on} } \Omega_R^{\hat{\epsilon}},\\F_D(\hat{\epsilon},x) \hat{C}_LF_D(\hat{\epsilon},x)^{-1} -I \quad  &\mbox{\rm{ on} } \Omega_L^{\hat{\epsilon}},\\0 \quad &\mbox{\rm{ on} } \Omega_C^{\hat{\epsilon}},\end{array}\end{cases}
\end{equation}
with $F_s(\hat{\epsilon},x)$ a fundamental matrix of solutions of the model system (as in Notation \ref{N: def F_s}) which is completely determined by the given formal invariants.

Then, we consider
\begin{equation}\label{E:def W pour A}
W_s(\hat{\epsilon},x)=H_s(\hat{\epsilon}, x)F_s(\hat{\epsilon},x), \quad (\hat{\epsilon},x) \in (S\cup \{0\}) \times \Omega_s^{\hat{\epsilon}}, \, s=D,U.
\end{equation}
Relations (\ref{E:relation H to be sat}) implies that
\begin{equation}
W'_D(\hat{\epsilon},x)W_D(\hat{\epsilon},x)^{-1}=W'_U(\hat{\epsilon},x)W_U(\hat{\epsilon},x)^{-1}, \quad \mbox{on } \Omega_\cap^{\hat{\epsilon}},  \,\hat{\epsilon}\in (S\cup \{0\}),
\end{equation}
so that
\begin{equation}
A(\hat{\epsilon},x)=\begin{cases}(x^2-\epsilon)W'_D(\hat{\epsilon},x)W_D(\hat{\epsilon},x)^{-1}, \quad \mbox{on } \Omega_D^{\hat{\epsilon}}, \\
(x^2-\epsilon)W'_U(\hat{\epsilon},x)W_U(\hat{\epsilon},x)^{-1}, \quad \mbox{on } \Omega_U^{\hat{\epsilon}},
\end{cases}
\end{equation}
is well-defined and hence analytic over $(S\cup \{0\}) \times (\mathbb{D}_r \backslash \{ \hat{x}_R,\hat{x}_L\})$.

We will prove the boundedness of $H_s(\hat{\epsilon},x)$, $H_s(\hat{\epsilon},x)^{-1}$ and $H'_s(\hat{\epsilon},x)$ near $x=\hat{x}_l$, for $\hat{\epsilon}\in (S\cup \{0\})$, $s=D,U$ and $l=L,R$. This implies that $A(\hat{\epsilon},x)$ is analytic over $S \times \mathbb{D}_{r}$ and has a limit when $\hat{\epsilon} \to 0$ (with $\hat{\epsilon} \in S$) that is analytic in $x$ over $\mathbb{D}_r$, since
\begin{equation}\begin{array}{lll}
(x^2-\epsilon)W'_s(\hat{\epsilon},x)W_s(\hat{\epsilon},x)^{-1}\\
\quad =(x^2-\epsilon)H'_s(\hat{\epsilon},x)H_s(\hat{\epsilon},x)^{-1}+H_s(\hat{\epsilon},x)\Lambda(\epsilon,x)H_s(\hat{\epsilon},x)^{-1}.\end{array}
\end{equation}

$H_s(\hat{\epsilon},x)$ will be obtained in Section \ref{S:const H} from a specific sequence of matrices constructed in Section \ref{S:construc seq}. This proof needs adequate choice of radii $r$ of $\Omega_s^{\hat{\epsilon}}$ and $\rho$ of $S$.

\subsection{Choice of the radius $r$ for the domains in the $x$-variable}
First, we choose $r$ by considering the case $\epsilon=0$. For $r >0$, let us take $\Omega_D$ and $\Omega_U$ as in Definition \ref{D:defsectors} (Figure \ref{fig:Art2 4}) and let
\begin{equation}\label{E:defsectors beta}
\begin{array}{lll}
\Omega_{D,\beta}=\{x \in \mathbb{C} : |x|<r, -(\pi+\delta+\beta) < \arg(x)< \delta +\beta \},\\
\Omega_{U,\beta}=\{x \in \mathbb{C} : |x|<r, -(\delta+\beta) < \arg(x)< \pi+ \delta+\beta \},
\end{array}
\end{equation}
with $\beta>0$ sufficiently small so that the closure of $\Omega_{s,\beta}$ does not contain more separation rays (Definition \ref{D:separation rays}) than $\Omega_s$, $s=D,U$ (Figure \ref{fig:Art2 19}). From these domains, we define domains having their part of the boundary other than the part $\{|x|=r\}$ included in some solution curves of the system $\dot{x}=x^2-\epsilon$ allowing complex time. The procedure explained in Section \ref{S:sectors x} yields $\Omega^0_s$ (respectively $\Omega^0_{s,\beta}$) included in $\Omega_s$ (respectively $\Omega_{s,\beta}$), for $s=D,U$ (Figure \ref{fig:Art2 19}). In the course of the proof, for domains denoted by the letter $\Omega$, we use the notation
\begin{equation}
\Omega_\cap=\Omega_U \cap \Omega_D=\Omega_L \cup \Omega_C \cup \Omega_R.
\end{equation}
\begin{figure}[h!]
\begin{center}
{\psfrag{A}{$\Omega_D$}
\psfrag{D}{$\Omega_{D,\beta}$}
\psfrag{B}{$\Omega^{0}_D$}
\psfrag{C}{$\Omega^{0}_{D,\beta}$}
\includegraphics[width=11cm]{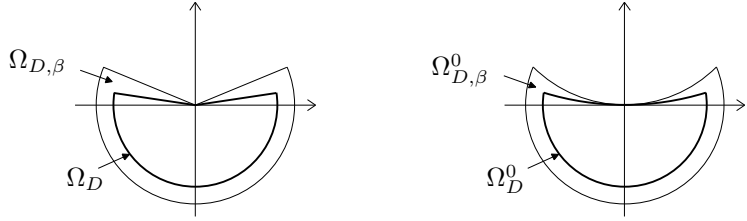}}
    \caption{Sectorial domains $\Omega_D$, $\Omega_{D,\beta}$, $\Omega^0_D$ and $\Omega^0_{D,\beta}$.}
    \label{fig:Art2 19}
\end{center}
\end{figure}

We now define domains $\Omega^0_s(\nu)$ included in $\Omega^0_{s,\beta}$ and converging when $\nu \to \infty$ to $\Omega^0_s$. In the $t$-variable (see Section \ref{S:sectors x}), let us define the neighborhoods $\Gamma^0_s(\nu)$ (Figure \ref{fig:Art2 X}) of $\Gamma^0_s$ (which is the domain corresponding to $\Omega^0_s$ in the $t$-variable):
\begin{equation}
\Gamma^0_s(\nu)= \{z : \exists t \in \Gamma_s^0 \mbox{ s.t. } |z-t|\frac{|z|}{|t|}< 2^{-\nu} \theta \}, \quad \nu \geq 1, \, s=D,U.
\end{equation}
\begin{figure}[h!]
\begin{center}
{\psfrag{J}{$\Gamma^{0}_D$}
\psfrag{I}{$\Gamma^{0}_U$}
\psfrag{K}{$\Gamma^{0}_D(\nu)$}
\psfrag{H}{$\Gamma^{0}_U(\nu)$}
\includegraphics[width=12cm]{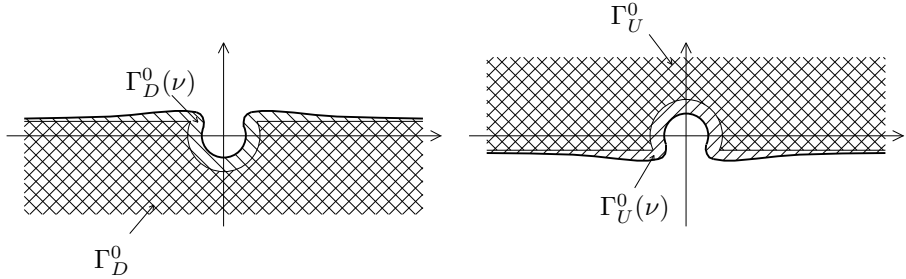}}
    \caption{A neighborhood $\Gamma^0_s(\nu)$ of $\Gamma^0_s$, $s=D,U$.}
    \label{fig:Art2 X}
\end{center}
\end{figure}
We choose $\theta>0$ such that $\Gamma^0_s(1)$ is included in $\Gamma^0_{s,\beta}$ (which is the domain corresponding to $\Omega^0_{s,\beta}$ in the $t$-variable). In the $x$-variable, the domains $\Gamma^0_s(\nu)$ correspond to
\begin{equation}
\Omega^0_s(\nu)=\{y : \exists x \in \Omega^0_s \mbox{ s.t. } |y-x|< 2^{-\nu} \theta |y|^2 \}, \quad \nu \geq 0, \, s=D,U.
\end{equation}

As illustrated in Figure \ref{fig:Art2 18}, we write the boundary of $\Omega^0_\cap(\nu)=\Omega^0_U(\nu) \cap \Omega^0_D(\nu)$ as
\begin{equation}
\partial\Omega^0_\cap(\nu)=\gamma^0_{\nu,U}\cup\gamma^0_{\nu,D},
\end{equation}
denoting $\gamma^0_{\nu,s} \subset \partial \Omega^0_\cap(\nu)$ the path included in the boundary of $\Omega_s^{0}(\nu) $, $s=D,U$ starting at $x=-r$ and ending at $x=r$.
\begin{figure}[h!]
\begin{center}
{\psfrag{E}{$\gamma_{\nu,D}^{0}$}
\psfrag{F}{$\gamma_{\nu,U}^{0}$}
\psfrag{B}{$\Omega_D^{0}(\nu)$}
\psfrag{C}{$\Omega_U^{0}(\nu)$}
\includegraphics[width=7cm]{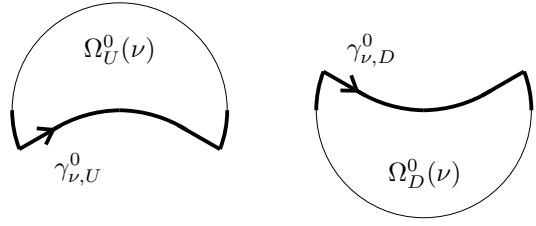}}
    \caption{Integration path $\gamma^0_{\nu,s} \subset \partial \Omega^0_s(\nu)$, $s=D,U$.}
    \label{fig:Art2 18}
\end{center}
\end{figure}

Asymptotic properties of $Z(0,x)$ imply that $ \forall N \in \mathbb{N}^*$ there exists $K_N^0 \in \mathbb{R}_+$ such that
\begin{equation}
|Z(0,x)| \leq K_N^0 |x|^N ,\quad  x \in \Omega^0_l(\theta), \, l=L,R.
\end{equation}
We take $r$ sufficiently small so that the length of each path $\gamma_{\nu,s}^0$ is bounded by a constant $c_s^0 $ such that
\begin{equation}
\int_{\gamma_{\nu,s}^0}|dh| < c_s^0 < \min\left\{\frac{\pi \theta}{2^4 K_2^0},\frac{\pi}{K_1^0}\right\}, \quad \nu \geq 1,\,s=D,U.
\end{equation}

\subsection{Choice of radius $\rho$ of $S$ and sequence of spiraling domains}\label{S:choice radius}
First, let us take the radius $\rho>0$ for $S$ such that $\rho < \min\{\rho_0,\frac{r²}{2}\}$. Restricting $\rho$ if necessary, we construct, as in Section \ref{S:sectors x}, sectorial domains $\Omega^{\hat{\epsilon}}_s$ (respectively $\Omega^{\hat{\epsilon}}_{s,\beta}$) that differ from $\Omega^0_s$ (respectively $\Omega^0_{s,\beta}$) mainly inside a small disk. $\Omega^{\hat{\epsilon}}_{s,\beta}$ is a neighborhood of $\Omega^{\hat{\epsilon}}_s$ (see Figure \ref{fig:Art2 20}). As in Figure \ref{fig:Art2 2}, these sectorial domains may spiral around the singular points, depending on the value of $\hat{\epsilon}$. Nevertheless, $\Omega^{\hat{\epsilon}}_s$ always stay inside $\Omega^{\hat{\epsilon}}_{s,\beta}$.
\begin{figure}[h!]
\begin{center}
{\psfrag{B}{$\Omega_D^{\hat{\epsilon}}$}
\psfrag{C}{$\Omega_{D,\beta}^{\hat{\epsilon}}$}
\psfrag{H}{\small{$\hat{x}_R$}}
\psfrag{G}{\small{$\hat{x}_L$}}
\includegraphics[width=9cm]{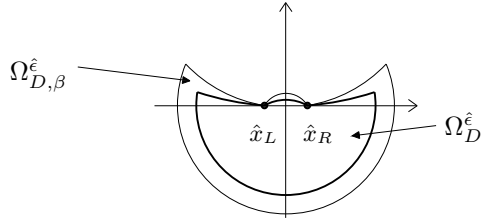}}
    \caption{Sectorial domains $\Omega^{\hat{\epsilon}}_D$ and $\Omega_{D,\beta}^{\hat{\epsilon}}$, case $\sqrt{\hat{\epsilon}}\in \mathbb{R}^*_-$.}
    \label{fig:Art2 20}
\end{center}
\end{figure}

For $\nu \geq 1$, we define the spiraling domains $\Omega^{\hat{\epsilon}}_s(\nu)$ which converge when $\nu \to \infty$ to $\Omega^{\hat{\epsilon}}_s$ and are included in $\Omega^{\hat{\epsilon}}_{s,\beta}$ for $\rho$ sufficiently small:
\begin{equation}
\Omega^{\hat{\epsilon}}_s(\nu)=\Omega^{\hat{\epsilon}}_s \cup_{l=L,R} \{y :  \exists x \in \Omega^{\hat{\epsilon}}_l \mbox{ s.t. } |y-x|< 2^{-\nu} \theta |y-\hat{x}_l|^2\}, \quad \hat{\epsilon} \in S\cup \{0\},\,s=D,U.
 \end{equation}
The spirals of $\Omega^{\hat{\epsilon}}_s(\nu)$ near $x=\hat{x}_l$ are approximately logarithmic.

As illustrated in Figure \ref{fig:Art2 17}, we denote as $\gamma^{\hat{\epsilon}}_{\nu,s}= \gamma^{\hat{\epsilon}}_{\nu,s,L} \cup \gamma^{\hat{\epsilon}}_{\nu,s,R}$ the broken path included in the boundary of $\Omega_s^{\hat{\epsilon}}(\nu) $, $s=D,U$. The path $\gamma^{\hat{\epsilon}}_{\nu,s,L}$ starts at $x=-r$ and ends at $x=\hat{x}_L$, whereas $\gamma^{\hat{\epsilon}}_{\nu,s,R}$ starts at $x=\hat{x}_R$ and ends at $x=r$. Remember that they may spiral near the singular points.
\begin{figure}[h!]
\begin{center}
{\psfrag{I}{$\gamma_{\nu,U,L}^{\hat{\epsilon}}$}
\psfrag{D}{$\gamma_{\nu,D,L}^{\hat{\epsilon}}$}
\psfrag{E}{$\gamma_{\nu,D,R}^{\hat{\epsilon}}$}
\psfrag{F}{$\gamma_{\nu,U,R}^{\hat{\epsilon}}$}
\psfrag{B}{$\Omega_D^{\hat{\epsilon}}(\nu)$}
\psfrag{C}{$\Omega_U^{\hat{\epsilon}}(\nu)$}
\psfrag{G}{\small{$\hat{x}_L$}}
\psfrag{H}{\small{$\hat{x}_R$}}
\includegraphics[width=10cm]{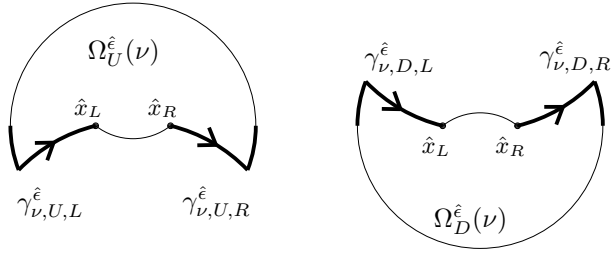}}
    \caption{Integration path $\gamma^{\hat{\epsilon}}_{\nu,s}= \gamma^{\hat{\epsilon}}_{\nu,s,L} \cup \gamma^{\hat{\epsilon}}_{\nu,s,R}$, $s=D,U$, case $\sqrt{\hat{\epsilon}} \in \mathbb{R}_{-}^*$.}
    \label{fig:Art2 17}
\end{center}
\end{figure}

Reducing $\rho$ if necessary, properties of $Z(\hat{\epsilon},x)$ (from (\ref{E:def Z1})) on $\Omega^{\hat{\epsilon}}_{L,\beta}$ and $\Omega^{\hat{\epsilon}}_{R,\beta}$ imply that, for $N=1,2,3,4$, there exists $K_N \in \mathbb{R}_+$ ($K_N \geq K_N^0$) such that
\begin{equation}\label{E:prop Z}
|Z(\hat{\epsilon},x)| \leq K_N |x-\hat{x}_l|^N,\quad (\hat{\epsilon},x) \in (S\cup\{0\}) \times \Omega^{\hat{\epsilon}}_l(1), \, l=L,R.
\end{equation}
Also,
\begin{equation}\label{E:centre Z1}
Z(\hat{\epsilon},x)=0 \quad (\hat{\epsilon},x) \in S \times \Omega_C^{\hat{\epsilon}}(1).
\end{equation}

We reduce $\rho$ in order to have
\begin{equation}\label{E:longueur int}
\int_{\gamma^{\hat{\epsilon}}_{\nu,s}}|dh| =c_s \leq \min \left\{\frac{\pi \theta}{2^4 K_2},\frac{\pi}{2^2 K_1}\right\}, \quad \nu \geq 1, \, \hat{\epsilon} \in S\cup \{0\}, \, s=D,U,
\end{equation}
(since the spirals are logarithmic, they have finite length).

\subsection{Construction of a specific sequence $Z^\nu$, $Z^\nu_{U}$ and $Z^\nu_{D}$}\label{S:construc seq}
In this section, starting from $Z^1=Z(\hat{\epsilon},x)$, we construct, for $\nu=2,3,...$, a sequence of matrices $Z^\nu$, $Z^\nu_{U}$ and $Z^\nu_{D}$ such that the following four conditions are satisfied:
\renewcommand{\theenumi}{\Roman{enumi}}
\begin{enumerate}
\item \label{decomposition}$Z^{\nu-1}=Z^\nu_{U}-Z^\nu_D$, for $(\hat{\epsilon},x) \in (S\cup \{0\}) \times \Omega_\cap^{\hat{\epsilon}}(\nu-1)$;
\item \label{info Zs} for $s=D,U$,
\begin{itemize} \item $Z^\nu_s(\hat{\epsilon},x)$ is analytic on $S \times \Omega_s^{\hat{\epsilon}}(\nu-1)$,\\
\item $Z^\nu_s(0,x)$ is analytic for $x \in \Omega_s^0(\nu-1)$,\\
\item $|Z^\nu_s| \leq 2^{-(\nu+1)}$ for $(\hat{\epsilon},x) \in (S\cup \{0\}) \times \Omega_s^{\hat{\epsilon}}(\nu)$ \\
\end{itemize}
\item \label{induction} $I+Z^{\nu}=(I+Z^\nu_D)(I+Z^{\nu-1})(I+Z^\nu_U)^{-1}$, $(\hat{\epsilon},x) \in (S\cup \{0\}) \times \Omega_\cap^{\hat{\epsilon}}(\nu-\delta)$ for some $0<\delta<1$;
\item \label{info Z} \begin{itemize}
\item $Z^\nu(0,x)$ is analytic over $\Omega_\cap^0(\nu-\delta)$,\\
\item $Z^\nu(\hat{\epsilon},x)=0$ on $S \times \Omega_C^{\hat{\epsilon}}(\nu-\delta)$,\\
\item $Z^\nu(\hat{\epsilon},x)$ is analytic on $S \times \Omega_\cap^{\hat{\epsilon}}(\nu-\delta)$ and satisfies, for $N=1,2,3,4$,\\
$|Z^\nu| \leq 2^{-2(\nu-1)} K_N |x-\hat{x}_l|^N $ for $(\hat{\epsilon},x) \in (S\cup \{0\}) \times \Omega_l^{\hat{\epsilon}}(\nu)$, $l=L,R$.
\end{itemize}
 \end{enumerate}

In order to obtain condition (\ref{decomposition}), we define the matrices $Z_D^\nu (\hat{\epsilon},x)$ and $Z_U^\nu (\hat{\epsilon},x)$ for $\nu=2,3,...$ by
\begin{equation}\label{E:def Z nu s}
Z^\nu_s(\hat{\epsilon},x)=\frac{1}{2 \pi i}\int_{\gamma^{\hat{\epsilon}}_{\nu-1,s}}\frac{Z^{\nu-1}(\hat{\epsilon},h)}{h-x}dh, \quad (\hat{\epsilon},x) \in (S\cup \{0\}) \times \Omega_s^{\hat{\epsilon}}(\nu-1), \, s=D,U.
\end{equation}
Condition (\ref{decomposition}) is satisfied since, for $(\hat{\epsilon},x) \in (S\cup \{0\}) \times \Omega_\cap^{\hat{\epsilon}}(\nu-1)$,
\begin{equation}
Z_U^\nu (\hat{\epsilon},x)-Z_D^\nu (\hat{\epsilon},x)=\frac{1}{2 \pi i}\int_{\gamma_{\nu-1}^{\hat{\epsilon}}}\frac{Z^{\nu-1}(\hat{\epsilon},h)}{h-x}dh =Z^{\nu-1}(\hat{\epsilon},x),
\end{equation}
where $\gamma_{\nu-1}^{\hat{\epsilon}}$ (Figure \ref{fig:Art2 W}) is a union of two paths surrounding $\Omega_L^{\hat{\epsilon}}(\nu-1)$ and $\Omega_R^{\hat{\epsilon}}(\nu-1)$: \begin{equation}\gamma_{\nu-1}^{\hat{\epsilon}}=\gamma^{\hat{\epsilon}}_{\nu-1,U,L}(\gamma^{\hat{\epsilon}}_{\nu-1,D,L})^{-1} \cup \gamma^{\hat{\epsilon}}_{\nu-1,U,R}(\gamma^{\hat{\epsilon}}_{\nu-1,D,R})^{-1}.\end{equation}
\begin{figure}[h!]
\begin{center}
{\psfrag{B}{$\Omega_D^{\hat{\epsilon}}(\nu-1)$}
\psfrag{C}{$\Omega_U^{\hat{\epsilon}}(\nu-1)$}
\psfrag{E}{$\Omega_L^{\hat{\epsilon}}(\nu-1)$}
\psfrag{F}{$\Omega_R^{\hat{\epsilon}}(\nu-1)$}
\psfrag{G}{\small{$\hat{x}_L$}}
\psfrag{H}{\small{$\hat{x}_R$}}
\includegraphics[width=8cm]{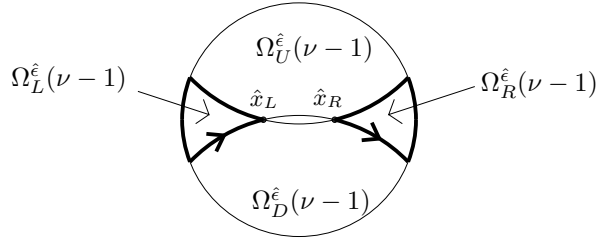}}
    \caption{Integration path $\gamma_{\nu-1}^{\hat{\epsilon}}$, case $\sqrt{\hat{\epsilon}} \in \mathbb{R}_{-}^*$.}
    \label{fig:Art2 W}
\end{center}
\end{figure}

Let us now prove (\ref{info Zs}) for $\nu \geq 2$, taking into account that (\ref{info Z}) is satisfied (it is indeed for $\nu=1$). When integrating in (\ref{E:def Z nu s}), we have
\begin{equation}\label{E:hx diff}
|h-x|\geq 2^{-\nu}\theta  |h-\hat{x}_l|^2, \quad h \in \gamma^{\hat{\epsilon}}_{\nu-1,s}, \, x \in \Omega_s^{\hat{\epsilon}}(\nu), \, \hat{\epsilon} \in S\cup \{0\}, \, s=D,U, \, l=L,R.
\end{equation}
Then, using (\ref{info Z}) as well as relations (\ref{E:longueur int}) and (\ref{E:hx diff}), we have, for $s=D,U$,
\begin{equation}\label{E:bound Znus}
\begin{array}{lll}
|Z^\nu_s(\hat{\epsilon},x)| &\leq \quad \frac{1}{2 \pi} \int_{\gamma^{\hat{\epsilon}}_{\nu-1,s}}\frac{|Z^{\nu-1}(\hat{\epsilon},h)|}{|h-x|}|dh|, \quad &(\hat{\epsilon},x) \in (S\cup \{0\}) \times \Omega_s^{\hat{\epsilon}}(\nu), \\
&\leq \frac{2^{-2(\nu-2)} K_2 c_s}{2\pi 2^{-\nu}\theta}\leq 2^{-(\nu+1)}, &(\hat{\epsilon},x) \in (S\cup \{0\}) \times \Omega_s^{\hat{\epsilon}}(\nu).
\end{array}
\end{equation}

Let us now prove condition (\ref{info Z}), taking $Z^{\nu}$ defined by relation (\ref{induction}) (there exists some $0<\delta<1$ such that $(I+Z^\nu_U)$ is invertible for $(\hat{\epsilon},x) \in (S\cup \{0\}) \times \Omega_\cap^{\hat{\epsilon}}(\nu-\delta)$). On each side of (\ref{induction}), multiplying by $(I+Z_U^{\nu})$ on the right yields
\begin{equation}
Z^\nu_U+Z^{\nu}(I+Z^\nu_U)=Z^{\nu-1}+Z^\nu_D+Z^\nu_D Z^{\nu-1}, \quad (\hat{\epsilon},x) \in (S\cup \{0\}) \times \Omega_\cap^{\hat{\epsilon}}(\nu).
\end{equation}
Using condition (\ref{decomposition}), it yields
\begin{equation}\label{E:rel znu rec}
Z^{\nu}(I+Z^\nu_U)= Z^\nu_D Z^{\nu-1}, \quad (\hat{\epsilon},x) \in (S\cup \{0\}) \times \Omega_\cap^{\hat{\epsilon}}(\nu).
\end{equation}
Hence,\begin{equation}\label{E:borne z nu fct autres}
\begin{array}{lll}
|Z^{\nu}|&\leq |Z^\nu_D| |Z^{\nu-1}| |(I+Z^\nu_U)^{-1}| \quad & (\hat{\epsilon},x) \in (S\cup \{0\}) \times \Omega_\cap^{\hat{\epsilon}}(\nu)\\
&\leq |Z^\nu_D| |Z^{\nu-1}| \frac{1}{1-|Z^\nu_U|} &(\hat{\epsilon},x) \in (S\cup \{0\}) \times \Omega_\cap^{\hat{\epsilon}}(\nu),
\end{array}
\end{equation}
the last inequality obtained since $|Z^\nu_U| < \frac{1}{2}$. Because of (\ref{E:centre Z1}), we have
\begin{equation}
Z^\nu(\hat{\epsilon},x)=0 \mbox{ on } S \times \Omega_C^{\hat{\epsilon}}(\nu).
\end{equation}
Finally, we finish the proof of (\ref{info Z}) from condition (\ref{info Zs}) and the induction hypothesis into (\ref{E:borne z nu fct autres}): for $N \leq 4$ and $l=L,R$, we have
\begin{equation}
\begin{array}{lll}
|Z^{\nu}|&\leq 2^{-(\nu+1)} (2^{-2(\nu-2)} K_N |x-\hat{x}_l|^N) (\frac{1}{1-2^{-(\nu+1)}}),  \,\, &(\hat{\epsilon},x) \in (S\cup \{0\}) \times \Omega_l^{\hat{\epsilon}}(\nu),\\
&\leq 2^{-2(\nu-1)} K_N |x-\hat{x}_l|^N, \,\,  &(\hat{\epsilon},x) \in (S\cup \{0\}) \times \Omega_l^{\hat{\epsilon}}(\nu).
\end{array}
\end{equation}

\subsection{Construction of $H_D(\hat{\epsilon},x)$ and $H_U(\hat{\epsilon},x)$}\label{S:const H}
The sequence of matrices $Z^\nu$, $Z^\nu_{U}$ and $Z^\nu_{D}$ constructed in Section \ref{S:construc seq} satisfies condition (\ref{induction}) and hence, for $(\hat{\epsilon},x) \in (S\cup \{0\}) \times \Omega_\cap^{\hat{\epsilon}}(\nu)$,
\begin{equation}\label{E: eq induction}
\begin{array}{lll}
I+Z(\hat{\epsilon},x)=I+Z^1=\\
\qquad \left[(I+Z^\nu_{D})...(I+Z^3_{D})(I+Z^2_{D}) \right]^{-1} (I+Z^{\nu})\left[(I+Z^\nu_{U})...(I+Z^3_{U})(I+Z^2_{U})\right].
\end{array}
\end{equation}

Since
\begin{equation}
\lim_{\nu \to \infty}Z^{\nu}=0, \quad (\hat{\epsilon},x) \in (S\cup \{0\}) \times \Omega_\cap^{\hat{\epsilon}}(\nu),
\end{equation}
and
\begin{equation}
\prod_{\nu=2}^{\infty} |1+Z^\nu_{s}| \leq \prod_{\nu=2}^{\infty} (1+2^{-(\nu+1)}), \quad (\hat{\epsilon},x) \in (S\cup \{0\}) \times \Omega_s^{\hat{\epsilon}}(\nu), \, s=D,U,
\end{equation}
the products in brackets are convergent in (\ref{E: eq induction}) when $\nu \to \infty$ and we are led to matrices satisfying (\ref{E:relation H to be sat}) (details in Lemma $4$ from the proof of Cartan's Lemma in \cite{GuRo} p.~195):
\begin{equation}\label{E:sol H fct Z}
H_s(\hat{\epsilon},x)=\lim_{\nu \to \infty} (I+Z_s^\nu)...(I+Z_s^3)(I+Z_s^2), \quad (\hat{\epsilon},x) \in (S\cup \{0\}) \times \Omega_s^{\hat{\epsilon}}(\nu),\, s=D,U.
\end{equation}

The boundedness of $H_s(\hat{\epsilon},x)$ and $H_s(\hat{\epsilon},x)^{-1}$ when $x \to \hat{x}_l$, $x \in \Omega_s^{\hat{\epsilon}}$, $\hat{\epsilon} \in S\cup \{0\}$, $s=D,U$, $l=L,R$, is obtained from (\ref{info Z}) and from the fact that the limit of the products in brackets in (\ref{E: eq induction}) are invertible and convergent when $\nu \to \infty$.

Let us now prove that $H'_s(\hat{\epsilon},x)=\frac{\partial H_s(\hat{\epsilon},x)}{\partial x}$ is bounded when $x \to \hat{x}_l$, $x \in \Omega_s^{\hat{\epsilon}}$, $\hat{\epsilon} \in S\cup \{0\}$, $l=L,R$ and $s=D,U$, by proving there exists $K \in \mathbb{R}_+$ such that
\begin{equation}\label{E:diff Hxlh}
|H_s(\hat{\epsilon},\hat{x}_l+t)-H_s(\hat{\epsilon},\hat{x}_l)|\leq K|t|, \quad \hat{x}_l+t \in \Omega_s^{\hat{\epsilon}}.
\end{equation}
First, let us prove there exists $k \in \mathbb{R}_+$ such that
\begin{equation}\label{E:diff Zxh}
|Z_s^\nu(\hat{\epsilon},\hat{x}_l+t)-Z_s^\nu(\hat{\epsilon},\hat{x}_l)|\leq  2^{-\nu}k |t|, \quad \hat{x}_l+t \in \Omega_s^{\hat{\epsilon}}(\nu).
\end{equation}
Using (\ref{E:def Z nu s}), (\ref{info Z}), (\ref{E:hx diff}) and (\ref{E:longueur int}), we have, for $t$ such that $\hat{x}_l+t \in \Omega_s^{\hat{\epsilon}}(\nu)$,
\begin{equation}\label{E:diff Zxh 2}
\begin{array}{lll}
|Z_s^\nu(\hat{\epsilon},\hat{x}_l+t)-Z_s^\nu(\hat{\epsilon},\hat{x}_l)|&=\frac{1}{2 \pi} \left| \int_{\gamma^{\hat{\epsilon}}_{\nu-1,s}}Z^{\nu-1}(\hat{\epsilon},h)\left(\frac{1}{h-(\hat{x}_l+t)}-\frac{1}{h-\hat{x}_l}\right) dh \right|\\
& \leq \frac{|t|}{2 \pi} \left| \int_{\gamma^{\hat{\epsilon}}_{\nu-1,s}}\frac{|Z^{\nu-1}(\hat{\epsilon},h)|}{|h-(\hat{x}_l+t)||h-\hat{x}_l|} |dh| \right|\\
&\leq \frac{|t|}{2 \pi} \frac{K_3 2^\nu c_s}{ 2^{2(\nu-2)}\theta}\\
&\leq  |t| \frac{K_3 }{2^{\nu+1}  K_2 },
\end{array}
\end{equation}
thus proving (\ref{E:diff Zxh}) with $k=\frac{K_3}{2K_2}$. To obtain (\ref{E:diff Hxlh}) from (\ref{E:diff Zxh}), let us denote shortly
\begin{equation}
Z_s^\nu(\hat{\epsilon},\hat{x}_l)=\hat{Z}_{s,l}^\nu, \quad Z_s^\nu(\hat{\epsilon},\hat{x}_l+t)=\hat{Z}_{s,t}^\nu.
\end{equation}
From (\ref{E:sol H fct Z}), we have
\begin{equation}\label{E:diff Hxh}
\begin{array}{lll}
|H_s(\hat{\epsilon},\hat{x}_l+t)-H_s(\hat{\epsilon},\hat{x}_l)|\\
\quad =\lim_{\nu \to \infty} |(I+\hat{Z}_{s,t}^\nu)...(I+\hat{Z}_{s,t}^3)(I+\hat{Z}_{s,t}^2)-(I+\hat{Z}_{s,l}^\nu)...(I+\hat{Z}_{s,l}^3)(I+\hat{Z}_{s,l}^2)|.
\end{array}
\end{equation}
Using (\ref{E:diff Zxh}) and (\ref{info Zs}), we can bound (\ref{E:diff Hxh}) and obtain (\ref{E:diff Hxlh}) from:
\begin{equation}
\begin{array}{lll}
|H_s(\hat{\epsilon},\hat{x}_l+t)-H_s(\hat{\epsilon},\hat{x}_l)|\\
\quad \leq \lim_{\nu \to \infty} \sum_{i=2}^\nu  |\hat{Z}_{s,t}^i-\hat{Z}_{s,l}^i| \prod_{q=2}^{i-1}|I+\hat{Z}_{s,t}^q|\prod_{p=i+1}^\nu |I+\hat{Z}_{s,l}^p|\\
\quad \leq \lim_{\nu \to \infty} \sum_{i=2}^\nu \frac{|I+\hat{Z}_{s,l}^i|}{1-2^{-(i +1)}}   |\hat{Z}_{s,t}^i-\hat{Z}_{s,l}^i| \prod_{q=2}^{i-1}|I+\hat{Z}_{s,t}^q|\prod_{p=i+1}^\nu |I+\hat{Z}_{s,l}^p| \\
\quad \leq \lim_{\nu \to \infty}  \sum_{i=2}^\nu  \frac{|\hat{Z}_{s,t}^i-\hat{Z}_{s,l}^i| }{1-2^{-(i +1)}}\prod_{p=2}^\nu (1+2^{-(p +1)})\\
\quad \leq \lim_{\nu \to \infty} \sum_{i=2}^\nu  \frac{k |t| }{ 2^{i}(1-2^{-(i +1)})} \prod_{p=2}^\nu (1+2^{-(p +1)}) \\
\quad \leq \lim_{\nu \to \infty} k |t| \sum_{i=2}^\nu  2^{-(i-1)} \prod_{p=2}^\nu (1+2^{-(p +1)}) .
\end{array}
\end{equation}

This section concludes the proof of Theorem \ref{T:reala loca}.
\hfill $\Box$

\subsection{Introduction to the proof of Theorem \ref{T:reala globa}}\label{S:th global deb}

From now on and until the end of Section \ref{S:realization}, we present the proof of Theorem \ref{T:reala globa}, using the proof of Theorem \ref{T:reala loca}.

Since the given system of invariants satisfy the auto-intersection relation (\ref{E:vraie rela compat}), Theorem \ref{T:summ} allows us to take, without loss of generality, the unfolded Stokes matrices as $\frac{1}{2}$-summable in $\epsilon$ and then, by (\ref{E:rela N K Q 2}), the corresponding matrices $\tilde{N}_R$ and $\bar{N}_R$ (Definition \ref{D:Nl tilde}) satisfy
\begin{equation}\label{E: cond with T'}
\bar{N}_R=\tilde{N}_R \bar{Q},
\end{equation}
with $\bar{Q}$ a nonsingular diagonal matrix exponentially close to $I$ in $\sqrt{\epsilon}$. Let
\begin{equation}\label{E:ramified system}
(x^2-\epsilon)v'=A(\hat{\epsilon},x)v
\end{equation}
be the system constructed in the proof of Theorem \ref{T:reala loca} by using the $\frac{1}{2}$-summable unfolded Stokes matrices. We will correct the system (\ref{E:ramified system}) by a transformation $y=J(\hat{\epsilon},x)v$ (defined for $(\hat{\epsilon},x) \in S \times \mathbb{D}_r$) to obtain a system $(x^2 -\epsilon)y'=B(\epsilon,x)y$ with $B(\epsilon,x)$ analytic in $\epsilon$ at $\epsilon=0$. The condition (\ref{E: cond with T'}) will be used in the correction of the family.

\subsection{The correction to a uniform family}
Let $\bar{\epsilon}$ and $\tilde{\epsilon}=\bar{\epsilon}e^{2\pi i}$ in $S_\cap$. Similarly as in Proposition \ref{P:relation compat prop},  $\bar{N}_R$ (respectively $\tilde{N}_R$) is the transition matrix $E_{R,\bar{x}_R \to \bar{x}_L}$ (respectively $E_{R,\tilde{x}_L \to \tilde{x}_R}$) between $H_D(\bar{\epsilon},x)F_D(\bar{\epsilon},x) \bar{T}_R$ and $H_U(\bar{\epsilon},x)F_U(\bar{\epsilon},x)\bar{D}_R \bar{T}_L \bar{D}_R^{-1}$ (respectively $H_D(\tilde{\epsilon},x)F_D(\tilde{\epsilon},x)\tilde{T}_L$ and $H_U(\tilde{\epsilon},x)F_U(\tilde{\epsilon},x)\tilde{D}_R\tilde{T}_R\tilde{D}_R^{-1}$). Because the transition matrices satisfy (\ref{E: cond with T'}), Proposition \ref{P:transit inv} implies that there exists an invertible transformation $P(\bar{\epsilon},x)$ analytic in $(\bar{\epsilon},x) \in S_\cap \times \mathbb{D}_r$ and conjugating the systems $(x^2-\epsilon)v'=A(\bar{\epsilon},x)v$ and $(x^2-\epsilon)v'=A(\tilde{\epsilon},x)v$, i.e.
 \begin{equation}\label{E:conjug tilde bar}
A(\tilde{\epsilon},x)=P(\bar{\epsilon},x)A(\bar{\epsilon},x)P(\bar{\epsilon},x)^{-1}+(x^2-\epsilon)P(\bar{\epsilon},x)'P(\bar{\epsilon},x)^{-1}.
\end{equation}
We need to go inside the details of the construction of $P(\bar{\epsilon},x)$ to estimate its growth. $P(\bar{\epsilon},x)$ is as follows:
\begin{equation}\label{E:def Pebar}
\begin{array}{lll}
P(\bar{\epsilon},x)=\begin{cases}H_U(\tilde{\epsilon},x)F_U(\tilde{\epsilon},x)\tilde{D}_R\tilde{T}_R\tilde{D}_R^{-1} \bar{Q} &\\
\qquad \times \left(H_U(\bar{\epsilon},x)F_U(\bar{\epsilon},x)[\bar{D}_R \bar{T}_L \bar{D}_R^{-1}] \right)^{-1}, &x \in \Omega_U^{\bar{\epsilon}} \cap \Omega_U^{\tilde{\epsilon}},\\
H_D(\tilde{\epsilon},x)F_D(\tilde{\epsilon},x)\tilde{T}_L  \left(H_D(\bar{\epsilon},x)F_D(\bar{\epsilon},x) \bar{T}_R \right)^{-1}, & x \in \Omega_D^{\bar{\epsilon}} \cap \Omega_D^{\tilde{\epsilon}}.
\end{cases}
\end{array}
\end{equation}
$P(\bar{\epsilon},x)$ is well-defined (to verify, use (\ref{E:relation H to be sat}), (\ref{E:F ramif}) and (\ref{E:Stokes and Txl})) and can be analytically extended to $\mathbb{D}_r$. It satisfies $P(0,x)=I$ (see Lemma \ref{L:exp t}).

In Section \ref{S:cons P}, we will show that there exists $\mathcal{K}_1 \in \mathbb{R}_+$ such that
\begin{equation}\label{E:diff Ptildebar}
|P(\bar{\epsilon},x)-I|\leq \mathcal{K}_1 |\bar{\epsilon}|, \quad(\bar{\epsilon},x) \in (S_\cap \cup \{0\})  \times \mathbb{D}_r.
\end{equation}
This leads to the proof, sketched in Section \ref{S:const J}, of the existence of $J(\hat{\epsilon},x)$, a nonsingular matrix depending analytically on $(\hat{\epsilon},x) \in S \times B_r$ such that
\begin{equation}\label{E:def P}
J(\tilde{\epsilon},x)^{-1}J(\bar{\epsilon},x)=P(\bar{\epsilon},x)
\end{equation}
on $S_\cap$ and such that $J(\hat{\epsilon},x)$, $J'(\hat{\epsilon},x)$ and $J(\hat{\epsilon},x)^{-1}$ have a bounded limit at $\epsilon=0$ (this proof requires slight reductions of the radius and opening of $S$).

Let $(x^2-\epsilon)y'=B(\hat{\epsilon},x)y$ be the system obtained by the change $y=J(\hat{\epsilon},x)v$ into (\ref{E:ramified system}). We have
\begin{equation}
B(\hat{\epsilon},x)=J(\hat{\epsilon},x)A(\hat{\epsilon},x)J(\hat{\epsilon},x)^{-1}+(x^2-\epsilon)J(\hat{\epsilon},x)'J(\hat{\epsilon},x)^{-1}.
\end{equation}
Replacing (\ref{E:def P}) into (\ref{E:conjug tilde bar}), we get
\begin{equation}\begin{array}{lll}
&J(\tilde{\epsilon},x)A(\tilde{\epsilon},x)J(\tilde{\epsilon},x)^{-1}+(x^2-\epsilon)J(\tilde{\epsilon},x)'J(\tilde{\epsilon},x)^{-1}\\=&J(\bar{\epsilon},x)A(\bar{\epsilon},x)J(\bar{\epsilon},x)^{-1}+(x^2-\epsilon)J(\bar{\epsilon},x)'J(\bar{\epsilon},x)^{-1},
\end{array}
\end{equation}
and hence we will have $B(\tilde{\epsilon},x)=B(\bar{\epsilon},x)$ on $S_\cap$ (for $x$ fixed). $B(\epsilon,x)$ will be analytic in $\epsilon$ because it will be unramified and because $\lim_{\epsilon \to 0}B(\epsilon,x)$ will exist.

In conclusion, once (\ref{E:diff Ptildebar}) and the existence of the desired $J(\hat{\epsilon},x)$ are proved (in Sections \ref{S:cons P} and \ref{S:const J}), we will have constructed an analytic family of systems with the given complete system of analytic invariants.

\subsection{Properties of $P(\bar{\epsilon},x)$ near $\bar{\epsilon}=0$}\label{S:cons P}
In this section, we show that the conjugating transformation $P(\bar{\epsilon},x)$ satisfies (\ref{E:diff Ptildebar}).

\subsubsection{Proof of (\ref{E:diff Ptildebar})}
Let us detail how to obtain (\ref{E:diff Ptildebar}) for $\bar{\epsilon} \ne 0$ from the construction of $P(\bar{\epsilon},x)$ given by (\ref{E:def Pebar}). We will prove that (\ref{E:diff Ptildebar}) is satisfied for $x \in (\Omega_U^{\bar{\epsilon}} \cap \Omega_U^{\tilde{\epsilon}}) \cup (\Omega_D^{\bar{\epsilon}} \cap \Omega_D^{\tilde{\epsilon}})$. By the Maximum Modulus Theorem, this implies that (\ref{E:diff Ptildebar}) is satisfied for $x \in \mathbb{D}_r$.

With the shorter notations
\begin{equation}
\hat{H}_D=H_D(\hat{\epsilon},x) \quad \mbox{and} \quad \hat{F}_D=F_D(\hat{\epsilon},x),
\end{equation}
we have, for $x \in \Omega_D^{\bar{\epsilon}} \cap \Omega_D^{\tilde{\epsilon}}$,
\begin{equation}\label{E:borne P}
\begin{array}{lll}
|P(\bar{\epsilon},x)-I|&=|\tilde{H}_D \tilde{F}_D \tilde{T}_L \bar{T}_R^{-1}\bar{F}_D^{-1} \bar{H}_D^{-1}-I| \\
&=|\tilde{H}_D \tilde{F}_D (\tilde{T}_L \bar{T}_R^{-1}-I)\bar{F}_D^{-1} \bar{H}_D^{-1}+(\tilde{H}_D \bar{H}_D^{-1}-I)| \\
&\leq |\bar{H}_D^{-1}| |\tilde{H}_D| |\tilde{F}_D ||\bar{F}_D^{-1}| |\tilde{T}_L \bar{T}_R^{-1}-I| +|\bar{H}_D^{-1}|  |\tilde{H}_D -\bar{H}_D|\\
&\leq |\bar{H}_D^{-1}|  |\tilde{H}_D| |\tilde{F}_D ||\bar{F}_D^{-1}| (|\tilde{T}_L -I|+|\bar{T}_R^{-1}-I| +|\tilde{T}_L -I| | \bar{T}_R^{-1}-I| ) \\& \quad +|\bar{H}_D^{-1}|  |\tilde{H}_D -\bar{H}_D|,
\end{array}
\end{equation}
as well as a similar relation on $\Omega_U^{\bar{\epsilon}} \cap \Omega_U^{\tilde{\epsilon}}$. From Lemma \ref{L:exp t} (and using (\ref{E:prod DrDl})), the following matrices appearing in (\ref{E:borne P}) and in the similar relation on $\Omega_U^{\bar{\epsilon}} \cap \Omega_U^{\tilde{\epsilon}}$ are exponentially close in $\sqrt{\epsilon}$ to $I$:
\begin{equation}
\begin{array}{llllll}
\tilde{D}_R\tilde{T}_R\tilde{D}_R^{-1}, \quad &\bar{D}_R \bar{T}_L^{-1} \bar{D}_R^{-1}, \quad & \tilde{T}_L, \quad &\bar{T}_R^{-1} .
\end{array}
\end{equation}
Hence, in order to obtain the relation (\ref{E:diff Ptildebar}) for $x \in (\Omega_U^{\bar{\epsilon}} \cap \Omega_U^{\tilde{\epsilon}}) \cup (\Omega_D^{\bar{\epsilon}} \cap \Omega_D^{\tilde{\epsilon}})$, it suffices to bound $|H_s(\tilde{\epsilon},x)-H_s(\bar{\epsilon},x)|$ . From (\ref{E:sol H fct Z}), we have
\begin{equation}\label{E:lim prod Z}
\begin{array}{lll}
|H_s(\tilde{\epsilon},x)-H_s(\bar{\epsilon},x)|\\
\quad =\lim_{\nu \to \infty} |(I+\tilde{Z}_s^\nu)...(I+\tilde{Z}_s^3)(I+\tilde{Z}_s^2)-(I+\bar{Z}_s^\nu)...(I+\bar{Z}_s^3)(I+\bar{Z}_s^2)|.
\end{array}
\end{equation}
We will prove in Section \ref{S:proof of Znutildebar} that there exists $k_{1} \in \mathbb{R}_+$ such that, for $\nu \geq 2$,
\begin{equation}\label{E:diff Znutildebar 1}
|\tilde{Z}_s^\nu-\bar{Z}_s^\nu|\leq  2^{-\nu}k_{1} |\bar{\epsilon}|,  \quad(\bar{\epsilon},x) \in S_\cap \times (\Omega_s^{\bar{\epsilon}}(\nu) \cap \Omega_s^{\tilde{\epsilon}}(\nu)),  \, s=D,U.
\end{equation}
Using (\ref{E:lim prod Z}), (\ref{E:diff Znutildebar 1}) and condition (\ref{info Zs}) in Section \ref{S:construc seq}, we then have
\begin{equation}\label{E:lim prod Z 2}
\begin{array}{lll}
|H_s(\tilde{\epsilon},x)-H_s(\bar{\epsilon},x)|\\
\quad \leq \lim_{\nu \to \infty} \sum_{i=2}^\nu  |\tilde{Z}_s^i-\bar{Z}_s^i| \prod_{q=2}^{i-1}|I+\tilde{Z}_s^q|\prod_{p=i+1}^\nu |I+\bar{Z}_s^p|\\
\quad \leq \lim_{\nu \to \infty} \sum_{i=2}^\nu  |\tilde{Z}_s^i-\bar{Z}_s^i|\prod_{q=2}^{i-1}|I+\tilde{Z}_s^q| \prod_{p=i+1}^\nu |I+\bar{Z}_s^p| \frac{|I+\tilde{Z}_s^i|}{1-2^{-(i +1)}}\\
\quad \leq \lim_{\nu \to \infty} \prod_{p=2}^\nu (1+2^{-(p +1)}) \sum_{i=2}^\nu  \frac{|\tilde{Z}_s^i-\bar{Z}_s^i| }{1-2^{-(i +1)}}\\
\quad \leq \lim_{\nu \to \infty} \prod_{p=2}^\nu (1+2^{-(p +1)}) \sum_{i=2}^\nu  \frac{k_{1} |\bar{\epsilon}| }{ 2^{i}(1-2^{-(i +1)})}\\
\quad \leq \lim_{\nu \to \infty} k_1 |\bar{\epsilon}| \prod_{p=2}^\nu (1+2^{-(p +1)}) \sum_{i=2}^\nu  2^{-(i-1)},
\end{array}
\end{equation}
yielding the existence of $\mathcal{K}^*_{1} \in \mathbb{R}_+$ such that
\begin{equation}\label{E:diff H s}
\begin{array}{lll}
|H_s(\tilde{\epsilon},x)-H_s(\bar{\epsilon},x)| \leq \mathcal{K}^*_{1} |\bar{\epsilon}|, \quad(\bar{\epsilon},x) \in S_\cap \times (\Omega_s^{\bar{\epsilon}} \cap \Omega_s^{\tilde{\epsilon}}),  \, s=D,U.
\end{array}
\end{equation}

\subsubsection{Property (\ref{E:diff Znutildebar 1}) of $Z_s^\nu$}\label{S:proof of Znutildebar}
Let us now prove (\ref{E:diff Znutildebar 1}), the remaining ingredient of the proof of (\ref{E:diff Ptildebar}) for $x \in (\Omega_U^{\bar{\epsilon}} \cap \Omega_U^{\tilde{\epsilon}}) \cup (\Omega_D^{\bar{\epsilon}} \cap \Omega_D^{\tilde{\epsilon}})$. From the definition of $\hat{Z}_s^\nu$ in (\ref{E:def Z nu s}), we have, for $(\hat{\epsilon},x) \in S \times (\Omega_s^{\bar{\epsilon}} (\nu)\cap \Omega_s^{\tilde{\epsilon}}(\nu))$ and $s=D,U$,
\begin{equation}\label{E:diff Znustildebar pp}
 \left| \tilde{Z}_s^\nu-\bar{Z}_s^\nu \right| = \left| \frac{1}{2 \pi i}\int_{\gamma^{\tilde{\epsilon}}_{\nu-1,s}}\frac{Z^{\nu-1}(\tilde{\epsilon},h)}{h-x}dh-\frac{1}{2 \pi i}\int_{\gamma^{\bar{\epsilon}}_{\nu-1,s}}\frac{Z^{\nu-1}(\bar{\epsilon},h)}{h-x}dh \right|
\end{equation}
The integration paths in (\ref{E:diff Znustildebar pp}) differ near the singular points but have a nonvoid common part. For $s=D,U$, we denote by $i^{\bar{\epsilon}}_{\nu,s}$ the common part of $\gamma^{\tilde{\epsilon}}_{\nu,s}$ and $\gamma^{\bar{\epsilon}}_{\nu,s}$, and by $r^{\tilde{\epsilon}}_{\nu,s}$ and $r^{\bar{\epsilon}}_{\nu,s}$ their respective remaining broken paths (i.e. we have $\gamma^{\hat{\epsilon}}_{\nu,s}=i^{\bar{\epsilon}}_{\nu,s}+r^{\hat{\epsilon}}_{\nu,s}$). Finally, as illustrated in Figure \ref{fig:Art2 21}, we separate the left and right parts of $r^{\hat{\epsilon}}_{\nu,s}$, denoting $r^{\hat{\epsilon}}_{\nu,s}=r^{\hat{\epsilon}}_{\nu,s,L} \cup r^{\hat{\epsilon}}_{\nu,s,R}$.
\begin{figure}[h!]
\begin{center}
{\psfrag{E}{{$r^{\tilde{\epsilon}}_{\nu,U,L} $}}
\psfrag{F}{{$r^{\bar{\epsilon}}_{\nu,U,L}$}}
\psfrag{D}{{$r^{\tilde{\epsilon}}_{\nu,U,R} $}}
\psfrag{B}{{$r^{\bar{\epsilon}}_{\nu,U,R}$}}
\psfrag{G}{{$i_{\nu,U}$}}
\psfrag{H}{{$r^{\bar{\epsilon}}_{\nu,D,R}$}}
\psfrag{I}{{$r^{\tilde{\epsilon}}_{\nu,D,R} $}}
\psfrag{L}{{$r^{\bar{\epsilon}}_{\nu,D,L}$}}
\psfrag{K}{{$r^{\tilde{\epsilon}}_{\nu,D,L} $}}
\psfrag{J}{{$i_{\nu,D}$}}
\psfrag{A}{\small{$\tilde{x}_R$}}
\psfrag{C}{\small{$\tilde{x}_L$}}
\subfigure[$s=U$]
    {\includegraphics[width=6cm]{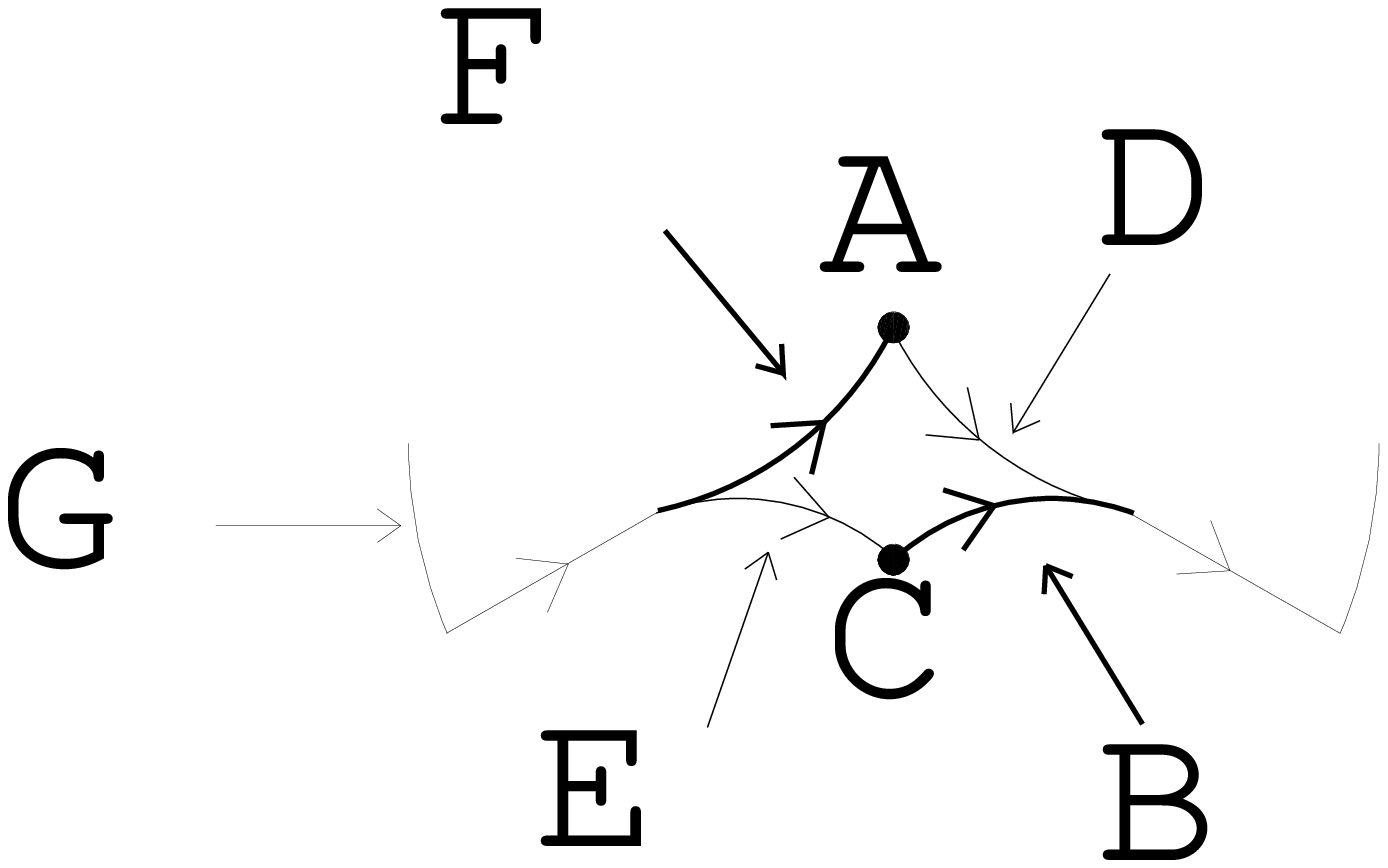}}\hfill
\subfigure[$s=D$]
    {\includegraphics[width=6cm]{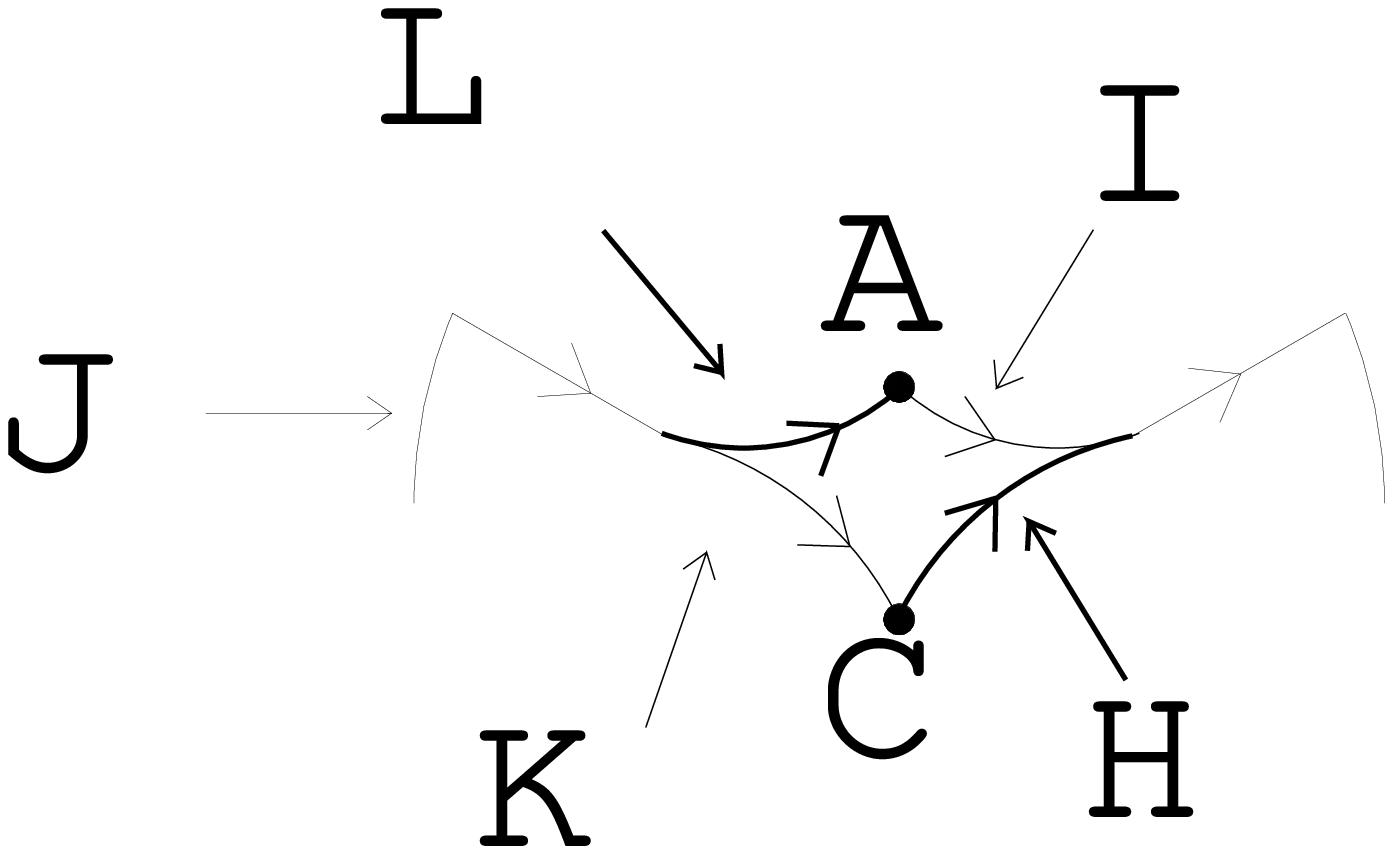}}}
\caption{Integration paths $i^{\bar{\epsilon}}_{\nu,s}$, $r^{\tilde{\epsilon}}_{\nu,s}=r^{\tilde{\epsilon}}_{\nu,s,L} \cup r^{\tilde{\epsilon}}_{\nu,s,R}$ and $r^{\bar{\epsilon}}_{\nu,s}=r^{\bar{\epsilon}}_{\nu,s,L} \cup r^{\bar{\epsilon}}_{\nu,s,R}$, $s=D,U$.}
\label{fig:Art2 21}
\end{center}
\end{figure}
With these notations, we can write (\ref{E:diff Znustildebar pp}) as
\begin{equation}\label{E:diff Znustildebar ppp}
\begin{array}{lll}
 \left| \tilde{Z}_s^\nu-\bar{Z}_s^\nu \right| &= \left| \frac{1}{2 \pi i}\int_{i^{\bar{\epsilon}}_{\nu-1,s}}\frac{Z^{\nu-1}(\tilde{\epsilon},h)-Z^{\nu-1}(\bar{\epsilon},h)}{h-x}dh \right. \\&\left. +\frac{1}{2 \pi i}\int_{r^{\tilde{\epsilon}}_{\nu-1,s}}\frac{Z^{\nu-1}(\tilde{\epsilon},h)}{h-x}dh-\frac{1}{2 \pi i}\int_{r^{\bar{\epsilon}}_{\nu-1,s} }\frac{Z^{\nu-1}(\bar{\epsilon},h)}{h-x}dh \right|,
\end{array}
\end{equation}
and hence
\begin{equation}\label{E:diff Znustildebar p}
\begin{array}{lll}
 \left| \tilde{Z}_s^\nu-\bar{Z}_s^\nu \right| &\leq \frac{1}{2 \pi} \int_{i^{\bar{\epsilon}}_{\nu-1,s}}\frac{|Z^{\nu-1}(\tilde{\epsilon},h)-Z^{\nu-1}(\bar{\epsilon},h)|}{|h-x|}|dh| \\&+\frac{1}{2 \pi}\left|\int_{r^{\tilde{\epsilon}}_{\nu-1,s}}\frac{Z^{\nu-1}(\tilde{\epsilon},h)}{h-x}dh\right|+\frac{1}{2 \pi }\left|\int_{r^{\bar{\epsilon}}_{\nu-1,s}}\frac{Z^{\nu-1}(\bar{\epsilon},h)}{h-x}dh\right|.
\end{array}
\end{equation}
In order to prove (\ref{E:diff Znutildebar 1}) from (\ref{E:diff Znustildebar p}), we will bound its last row, and then use induction.

By condition (\ref{info Z}) in Section \ref{S:construc seq}, we have
\begin{equation}
|Z^{\nu-1}(\hat{\epsilon},h)| \leq 2^{-2(\nu -2)} K_{4}|h-\hat{x}_l|^{4}, \quad (\hat{\epsilon},x) \in (S\cup \{0\}) \times \Omega_l^{\hat{\epsilon}}(\nu).
\end{equation}
Using (\ref{E:hx diff}), we thus have, for $x \in \Omega_l^{\bar{\epsilon}}(\nu) \cap \Omega_l^{\tilde{\epsilon}}(\nu)$ $s=D,U$, $l=L,R$ and $\hat{\epsilon} \in \{\tilde{\epsilon},\bar{\epsilon}\}$,
\begin{equation}
\begin{array}{lll}
\left| \int_{r^{\hat{\epsilon}}_{\nu-1,s,l} }\frac{Z^{\nu-1}(\hat{\epsilon},h)}{h-x}dh \right| &\leq \int_{r^{\hat{\epsilon}}_{\nu-1,s,l} } \frac{|Z^{\nu-1}(\hat{\epsilon},h)|}{|h-x|}|dh|\\
&\leq \int_{r^{\hat{\epsilon}}_{\nu-1,s,l}} \frac{ 2^{-2(\nu-2)}K_{4}|h-\hat{x}_l|^{2}}{2^{-\nu} \theta}|dh|.
\end{array}
\end{equation}
The integration paths $r^{\hat{\epsilon}}_{\nu,s}$ are located inside a disk of radius $c\sqrt{|\bar{\epsilon}|}$ for some $c \in \mathbb{R}_+^*$ (Section \ref{S:sectors x}), yielding
\begin{equation}
\begin{array}{lll}
\left| \int_{r^{\hat{\epsilon}}_{\nu-1,s,l} }\frac{Z^{\nu-1}(\hat{\epsilon},h)}{h-x}dh \right| &\leq \int_{r^{\hat{\epsilon}}_{\nu-1,s,l} }  \theta^{-1} 2^{4-\nu} K_{4}(|h|+\sqrt{|\epsilon |})^{2}|dh| \\ & \leq \int_{r^{\hat{\epsilon}}_{\nu-1,s,l} } \theta^{-1} 2^{4-\nu} K_{4}|\epsilon|(c+1)^{2}|dh|
\\ & =  \theta^{-1} 2^{4-\nu} K_{4}|\epsilon|(c+1)^{2} \int_{r^{\hat{\epsilon}}_{\nu-1,s,l}} |dh|.
\end{array}
\end{equation}
Thus, a bound for the last row of (\ref{E:diff Znustildebar p}) is, using (\ref{E:longueur int}) and the fact that the length of the path $r^{\hat{\epsilon}}_{\nu-1,s} $ is smaller than the length of the path $\gamma^{\hat{\epsilon}}_{\nu-1,s} $,
\begin{equation}\label{E: bound last row}
\begin{array}{lll}
\frac{1}{2 \pi }\left( \left| \int_{r^{\tilde{\epsilon}}_{\nu-1,s} }\frac{Z^{\nu-1}(\tilde{\epsilon},h)}{h-x}dh\right|+\left| \int_{r^{\bar{\epsilon}}_{\nu-1,s} }\frac{Z^{\nu-1}(\bar{\epsilon},h)}{h-x}dh \right| \right) \\
\quad \leq (2 \pi \theta)^{-1} 2^{4-\nu} K_{4}|\epsilon|(c+1)^{2} \left( \int_{r^{\tilde{\epsilon}}_{\nu-1,s} } |dh|+\int_{r^{\bar{\epsilon}}_{\nu-1,s}}|dh| \right)\\
\quad \leq (2 \pi \theta)^{-1} 2^{4-\nu} K_{4}|\epsilon|(c+1)^{2} \left( \int_{\gamma^{\tilde{\epsilon}}_{\nu-1,s} } |dh|+\int_{\gamma^{\bar{\epsilon}}_{\nu-1,s} }|dh| \right)\\
\quad \leq (2 \pi \theta)^{-1} 2^{4-\nu} K_{4}|\epsilon|(c+1)^{2} 2 c_s \\
\quad = \frac{k_1^*}{ 2^{\nu+5}}|\epsilon|,
\end{array}
\end{equation}
where
\begin{equation}\label{E:k nu N}
k_{1}^*=\frac{2^5 K_{4}(c+1)^{2}}{K_2}.
\end{equation}
Hence, (\ref{E:diff Znustildebar p}) becomes
\begin{equation}\label{E:diff Znustildebar}
\begin{array}{lll}
 \left| \tilde{Z}_s^\nu-\bar{Z}_s^\nu \right| \leq \frac{1}{2 \pi} \int_{i^{\bar{\epsilon}}_{\nu-1,s}}\frac{|Z^{\nu-1}(\tilde{\epsilon},h)-Z^{\nu-1}(\bar{\epsilon},h)|}{|h-x|}|dh|+\frac{k_{1}^*}{ 2^{\nu+5}}|\bar{\epsilon}|, \\
\hfill (\bar{\epsilon},x) \in S_\cap \times (\Omega_l^{\bar{\epsilon}}(\nu) \cap \Omega_l^{\tilde{\epsilon}}(\nu)).
\end{array}
\end{equation}
From (\ref{E:diff Znustildebar}), we will prove (\ref{E:diff Znutildebar 1}) for $\nu=2$, $\nu=3$ and $\nu > 3$.

Beginning with $\nu=2$, we have, from
\begin{equation}\label{E:f tilde et bar}
F_s(\bar{\epsilon},x)=F_s(\tilde{\epsilon},x), \quad x \in \Omega_s^{\bar{\epsilon}}\cap \Omega_s^{\tilde{\epsilon}}, \, s=D,U,
\end{equation}
and from (\ref{E:def Z1}),
\begin{equation}\label{E:diff Z1tildebar}
|\tilde{Z}^1-\bar{Z}^1|\leq \begin{cases}|F_D(\bar{\epsilon},x)\left(\tilde{C}_R-\bar{C}_R\right)F_D(\bar{\epsilon},x)^{-1}|,  \quad  &\mbox{\rm{ on} } \Omega_R^{\tilde{\epsilon}} \cap  \Omega_R^{\bar{\epsilon}},\\|F_D(\bar{\epsilon},x)\left(\tilde{C}_L-\bar{C}_L\right) F_D(\bar{\epsilon},x)^{-1}|, \quad  &\mbox{\rm{ on} } \Omega_L^{\tilde{\epsilon}}\cap  \Omega_L^{\bar{\epsilon}}.
\end{cases}
\end{equation}
By the $\frac{1}{2}$-summability of the unfolded Stokes matrices, $|\tilde{C}_l-\bar{C}_l|$ is exponentially close to $0$ in $\sqrt{\epsilon}$. Since
\begin{equation}
\left(F_D(\bar{\epsilon},x)\left(\tilde{C}_l-\bar{C}_l\right)F_D(\bar{\epsilon},x)^{-1}\right)_{ij}=\frac{f_i(\bar{\epsilon},x)}{f_j(\bar{\epsilon},x)}\left(\tilde{C}_l-\bar{C}_l\right)_{ij}, \quad l=L,R,
\end{equation}
and since there exists $k \in \mathbb{R}_+$ such that
\begin{equation}
\left|\frac{f_i(\bar{\epsilon},x)}{f_j(\bar{\epsilon},x)}\right| \leq k |x-\bar{x}_l|²,  \quad  \mbox{\rm{ on} } \Omega_l^{\tilde{\epsilon}}, \quad l=L,R, \quad i \ne j,
\end{equation}
relation (\ref{E:diff Z1tildebar}) implies that there exists $w_1 \in \mathbb{R}_+$ such that
\begin{equation}\label{E:diff Z1}
|\tilde{Z}^1-\bar{Z}^1|\leq \frac{w_1}{2^4} K_2|\bar{\epsilon}| |x-\bar{x}_l|^2, \quad l=L,R,\, (\bar{\epsilon},x) \in S_\cap \times (\Omega_l^{\bar{\epsilon}}(\nu) \cap \Omega_l^{\tilde{\epsilon}}(\nu)),
\end{equation}
with $K_2$ given by (\ref{E:prop Z}). Using relations (\ref{E:longueur int}) and (\ref{E:hx diff}) and the fact that the length of the path $i^{\bar{\epsilon}}_{\nu,s} $ is smaller than the length of the path $\gamma^{\hat{\epsilon}}_{\nu,s} $, the integral in (\ref{E:diff Znustildebar}) for $\nu=2$ is bounded by
\begin{equation}\label{E:bound first int 2}
\begin{array}{lll}
 \frac{1}{2 \pi} \int_{i_{1,s}}\frac{|Z^{1}(\tilde{\epsilon},h)-Z^{1}(\bar{\epsilon},h)|}{|h-x|}|dh|
&\leq \frac{1}{2 \pi} \int_{i_{1,s}}\frac{w_1 K_2|\bar{\epsilon}| |h-\bar{x}_l|^2}{2^4|h-x|}|dh|\\
&\leq \frac{1}{2 \pi} \int_{i_{1,s}}\frac{w_1 K_2|\bar{\epsilon}| |h-\bar{x}_l|^2}{2^42^{-2}\theta |h-\bar{x}_l|^2}|dh|\\
&\leq w_1 |\bar{\epsilon}| \frac{K_2}{2^3 \theta \pi}\int_{i_{1,s}}|dh|\\
&\leq w_1 |\bar{\epsilon}| \frac{K_2c_s }{2^3 \theta \pi}\\
&\leq \frac{1}{2^7} w_1|\bar{\epsilon}|.
\end{array}
\end{equation}
From (\ref{E:diff Znustildebar}) and (\ref{E:bound first int 2}), we have
\begin{equation}\label{E:diff Z 2 s}
|\tilde{Z}_s^2-\bar{Z}_s^2|\leq  \frac{1}{2^6} k_1|\bar{\epsilon}|,  \quad (\bar{\epsilon},x) \in S_\cap \times (\Omega_s^{\bar{\epsilon}}(2) \cap \Omega_s^{\tilde{\epsilon}})(2),  \, s=D,U,
\end{equation}
with
\begin{equation}\label{E:def kn}
k_1=\max\{k_{1}^*,w_1\}.
\end{equation}
Relation (\ref{E:diff Znutildebar 1}) is thus satisfied for $\nu=2$.

Now, let us study $|\tilde{Z}^{\nu-1}- \bar{Z}^{\nu-1}|$ in order to bound (\ref{E:diff Znustildebar}) for $\nu \geq 3$. From the equality
\begin{equation}
\tilde{A} \tilde{B} \tilde{C}^{-1}- \bar{A} \bar{B} \bar{C}^{-1}=\left( (\tilde{A}-\bar{A})\tilde{B} +\bar{A}(\tilde{B}-\bar{B})+ \bar{A} \bar{B} \bar{C}^{-1}(\bar{C}-\tilde{C})\right) \tilde{C}^{-1},
\end{equation}
applied to relation $Z^{\nu-1}= Z^{\nu-1}_D Z^{\nu-2}(I+Z^{\nu-1}_U)^{-1}$ coming from (\ref{E:rel znu rec}), we have (taking $Z^{\nu-1}=A B C^{-1}$, $Z^{\nu-1}_D=A$, $Z^{\nu-2}=B$ and $(I+Z^{\nu-1}_U)=C$)
\begin{equation}\label{E:ind borne e}
\begin{array}{lll}
|\tilde{Z}^{\nu-1}- \bar{Z}^{\nu-1}|\\
\quad \leq \left(|\tilde{Z}_D^{\nu-1}-\bar{Z}_D^{\nu-1}| |\tilde{Z}^{\nu-2}|+|\bar{Z}_D^{\nu-1}||\tilde{Z}^{\nu-2}-\bar{Z}^{\nu-2}|+ |\bar{Z}^{\nu-1}| |\bar{Z}_U^{\nu-1}-\tilde{Z}_U^{\nu-1}|\right)\\
\qquad \times |(I+\tilde{Z}_U^{\nu-1})^{-1}|.
\end{array}
\end{equation}
Let us remark that, because of (\ref{E:f tilde et bar}), we have
\begin{equation}\label{E:borne }
|\hat{Z}^\nu| \leq 2^{-2(\nu-1)} K_1 |x-\bar{x}_l|, \quad (\hat{\epsilon},x) \in S_\cap  \times \Omega_l^{\bar{\epsilon}}(\nu) \cap \Omega_l^{\tilde{\epsilon}}(\nu), l=L,R,
\end{equation}
coming from condition (\ref{info Z}) of Section \ref{S:construc seq}.

For $\nu=3$, equation (\ref{E:ind borne e}) yields, with the use of (\ref{E:diff Z1}), (\ref{E:diff Z 2 s}), (\ref{E:def kn}), (\ref{E:borne }) and $|Z^{\nu-1}_s| \leq 2^{-\nu}$ (from (\ref{info Zs}) ),
\begin{equation}\label{E:borne Z2}
\begin{array}{lll}
|\tilde{Z}^2-\bar{Z}^2|&\leq k_1 |\bar{\epsilon}| K_2|x-\bar{x}_l|^2 \left( \frac{1}{2^6} +\frac{1}{2^3 2^4}+\frac{1}{2^22^6}\right)\left(\frac{1}{1-2^{-3}}\right), \quad &l=L,R,\\
&\leq k_1 |\bar{\epsilon}| K_2 |x-\bar{x}_l|^2 \frac{1}{2^4}, & l=L,R,
\end{array}
\end{equation}
for $(\bar{\epsilon},x) \in S_\cap \times (\Omega_s^{\bar{\epsilon}}(2) \cap \Omega_s^{\tilde{\epsilon}})(2)$. In the same way as when we bounded (\ref{E:bound first int 2}), we use (\ref{E:borne Z2}) to bound the integral in (\ref{E:diff Znustildebar}) for $\nu=3$ :
\begin{equation}\label{E:borne int 3}
\begin{array}{lll}
 \frac{1}{2 \pi}\int_{i_{2,s}} \frac{|Z^{2}(\tilde{\epsilon},h)-Z^{2}(\bar{\epsilon},h)|}{|h-x|}|dh| &\leq  \frac{1}{2 \pi} \int_{i_{2,s}}\frac{k_1  |\bar{\epsilon}| K_2 |h-\bar{x}_l|^2}{2^{4}2^{-3}\theta  |h-\bar{x}_l|^2}|dh|\\ &\leq  \frac{k_1 |\bar{\epsilon}| K_2}{2^2 \pi \theta } \int_{i_{2,s}}|dh|\\&\leq  \frac{k_1 |\bar{\epsilon}| K_2 }{2^2  \pi \theta } \int_{\gamma^{\bar{\epsilon}}_{2,s}}|dh|\\&\leq \frac{1}{2^2}  k_1 |\bar{\epsilon}|  \frac{K_2c_s
}{\pi \theta }\\
&\leq \frac{1}{2^6}k_1 |\bar{\epsilon}|.
\end{array}
\end{equation}
Then, (\ref{E:borne int 3}) into (\ref{E:diff Znustildebar}) gives
\begin{equation}\label{E: bound Zs3}
|\tilde{Z}_s^3-\bar{Z}_s^3|\leq  \frac{1}{2^5} k_1 |\bar{\epsilon}|,  \quad\quad(\bar{\epsilon},x) \in S_\cap \times (\Omega_s^{\bar{\epsilon}}(3) \cap \Omega_s^{\tilde{\epsilon}}(3)),  \, s=D,U.
\end{equation}
Relation (\ref{E:diff Znutildebar 1}) is hence satisfied for $\nu=3$.

We are now ready to prove (\ref{E:diff Znutildebar 1}) for $\nu > 3$ by induction on $\nu$. Let us suppose that we have
\begin{equation}\label{E:borne Znu-1}
\begin{array}{lll}
|\tilde{Z}^{\nu-2}-\bar{Z}^{\nu-2}|\leq \frac{k_1}{2^{2(\nu-3)}} |\bar{\epsilon}| K_2 |x-\bar{x}_l|^2, \\ \qquad  (\bar{\epsilon},x) \in S_\cap \times (\Omega_s^{\bar{\epsilon}}(\nu-2) \cap \Omega_s^{\tilde{\epsilon}}(\nu-2)), \, l=L,R,
\end{array}
\end{equation}
and
\begin{equation}\label{E:borne Zsnu}
|\tilde{Z}_s^{\nu-1}-\bar{Z}_s^{\nu-1}|\leq  \frac{1}{2^{\nu-1}}k_1 |\bar{\epsilon}|,  \quad(\bar{\epsilon},x) \in S_\cap \times (\Omega_s^{\bar{\epsilon}}(\nu-1) \cap \Omega_s^{\tilde{\epsilon}}(\nu-1)),  \, s=D,U,
\end{equation}
(this is indeed satisfied for $\nu=4$ because of (\ref{E:borne Z2}) and (\ref{E: bound Zs3})). For $\nu > 3$, relation (\ref{E:ind borne e}) yields, using (\ref{E:borne Znu-1}), (\ref{E:borne Zsnu}), (\ref{E:borne }) and $|\hat{Z}^{\nu-1}_s| \leq 2^{-\nu}$ (from (\ref{info Zs})),
\begin{equation}\label{E:diff Znutildebar p}
\begin{array}{lll}
|\tilde{Z}^{\nu-1}-\bar{Z}^{\nu-1}|\leq & k_1 |\bar{\epsilon}| K_2 |x-\bar{x}_l|^2 \left( \frac{1}{2^{\nu-1}2^{2(\nu-3)}} +\frac{1}{2^\nu 2^{2(\nu-3)}}+\frac{1}{2^{2(\nu-2)}2^{\nu-1}}\right)\\
& \, \times \left(\frac{1}{1-2^{-\nu}}\right),
\end{array}
\end{equation}
and thus, for $(\bar{\epsilon},x) \in S_\cap \times (\Omega_s^{\bar{\epsilon}}(\nu-1) \cap \Omega_s^{\tilde{\epsilon}}(\nu-1))$ and $l=L,R$,
\begin{equation}\label{E:diff Znutildebar}
\begin{array}{lll}
|\tilde{Z}^{\nu-1}-\bar{Z}^{\nu-1}|\leq \frac{k_1}{2^{2(\nu-2)}} |\bar{\epsilon}| K_2 |x-\bar{x}_l|^2.\\
\end{array}
\end{equation}
In the same way as when we bounded (\ref{E:bound first int 2}) and (\ref{E:borne int 3}), we use (\ref{E:diff Znutildebar}) to bound the integral in (\ref{E:diff Znustildebar}) for $\nu > 3$:
\begin{equation}\label{E:borne int nu}
\begin{array}{lll}
 \frac{1}{2 \pi}  \int_{i^{\bar{\epsilon}}_{\nu-1,s}} \frac{|Z^{\nu-1}(\tilde{\epsilon},h)-Z^{\nu-1}(\bar{\epsilon},h)|}{|h-x|}|dh| &\leq\frac{1}{2 \pi} \int_{i^{\bar{\epsilon}}_{\nu-1,s}}\frac{k_1 K_2|\bar{\epsilon}|  |h-\bar{x}_l|^2}{2^{2(\nu-2)}|h-x|}|dh| \\
&\leq  \frac{1}{2 \pi} \int_{i^{\bar{\epsilon}}_{\nu-1,s}}\frac{k_1 K_2|\bar{\epsilon}|  |h-\bar{x}_l|^2}{2^{2(\nu-2)}2^{-\nu}\theta  |h-\bar{x}_l|^2}|dh|\\ &\leq  \frac{k_1 K_2|\bar{\epsilon}|  }{\pi 2^{\nu-3}\theta } \int_{i^{\bar{\epsilon}}_{\nu-1,s}}|dh|\\&\leq  \frac{k_1 K_2|\bar{\epsilon}|  }{\pi 2^{\nu-3}\theta } \int_{\gamma^{\hat{\epsilon}}_{\nu-1,s}}|dh|\\&\leq \frac{1}{2^{\nu-3}}  k_1 |\bar{\epsilon}|   \frac{K_2c_s
}{\pi \theta }\\
&\leq \frac{1}{2^{\nu+1}}k_1 |\bar{\epsilon}| .
\end{array}
\end{equation}
Then, (\ref{E:borne int nu}) and (\ref{E:diff Znustildebar}) gives  (\ref{E:diff Znutildebar 1}) for $\nu > 3$ (using (\ref{E:def kn})).

\subsection{Construction of $J(\hat{\epsilon},x)$}\label{S:const J}
For fixed $x$, the existence of $J(\hat{\epsilon},x)$ follows from the triviality of the vector bundle on the punctured disk in $\epsilon$-space. But, we need to show that $J(\hat{\epsilon},x)$ depends analytically on the "parameter" $x \in \mathbb{D}_r$ and also that we can fill the hole at $\epsilon=0$. So we need to go into the details of the construction of $J(\hat{\epsilon},x)$. We do this in a sketchy way since the details are completely similar (and simpler) to those we have done in Sections \ref{S:th local deb} to \ref{S:const H}.

$S$ has been taken previously with an opening $2 \pi + \gamma_0$. We reduce slightly the opening of $S$ to $2 \pi + \gamma$ with $0<\gamma<\gamma_0$, denoting the sector with the previous opening $S^{prev}$, such that, for some $\alpha>0$,
\begin{equation}
S(1)=S \cup \{\epsilon :  \exists \hat{\epsilon} \in S_\cap \mbox{ s.t. } |\epsilon-\hat{\epsilon}|< 2^{-1}\alpha|\epsilon| \} \subset S^{prev}.
 \end{equation}
We write $S$ as the union of two sectors $V_U$ and $V_D$
\begin{equation}
\begin{array}{lll}
V_D=\{\epsilon \in \mathbb{C} \, : \, 0<|\epsilon|<\rho, \, \arg(\epsilon) \in (\pi-\gamma,2\pi +\gamma ) \}, \\
V_U=\{\epsilon \in \mathbb{C} \, : \, 0<|\epsilon|<\rho, \, \arg(\epsilon) \in (2\pi-\gamma,3\pi +\gamma ) \}.
\end{array}
\end{equation}
We take the following domains converging when $\nu \to \infty$ to $V_s$ and included into $S^{prev}$:
\begin{equation}
V_s(\nu)=V_s \cup \{\epsilon :  \exists \hat{\epsilon} \in V_U \cap V_D \mbox{ s.t. } |\epsilon-\hat{\epsilon}|< 2^{-\nu}\alpha|\epsilon| \}, \quad \nu \geq 1, s=D,U.
\end{equation}
We separate the intersection of $V_U(\nu)$ and $V_D(\nu)$ into a left and a right domain:
\begin{equation}
V_\cap(\nu)=V_U(\nu) \cap V_D(\nu)=V_L(\nu) \cup V_R(\nu).
\end{equation}
We divide the boundary of $V_\cap(\nu)$ in two parts: as illustrated in Figure \ref{fig:Art2 22}, we denote $t_{\nu,s}=t_{\nu,s,L} \cup t_{\nu,s,R}$ the path included in the boundary of $V_s(\nu) $, $s=D,U$. The path $t_{\nu,s,L}$ begins at $\epsilon=-\rho$ and ends at $\epsilon=0$, whereas $t_{\nu,s,R}$ begins at $\epsilon=0$ and ends at $\epsilon=\rho$.
\begin{figure}[h!]
\begin{center}
{\psfrag{I}{$t_{\nu,U}$}
\psfrag{D}{$t_{\nu,D}$}
\psfrag{B}{$V_D(\nu)$}
\psfrag{C}{$V_U(\nu)$}
\psfrag{G}{\small{$0$}}
\includegraphics[width=9cm]{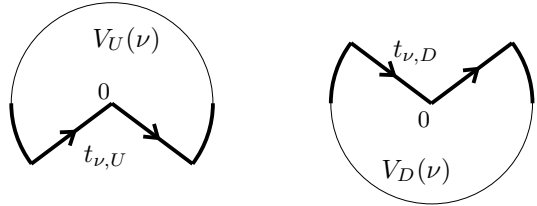}}
    \caption{Integration path $t_{\nu,s}=t_{\nu,s,L} \cup t_{\nu,s,R}$, $s=D,U$.}
    \label{fig:Art2 22}
\end{center}
\end{figure}

We reduce the radius $\rho$ of $S$ (and hence of $V_s$ and $V_s(\nu)$, $s=D,U$) a last time so that the length of each path $t_{\nu,s}$ is bounded as follows:
\begin{equation}
\int_{t_{\nu,s}}|dh| < l_s < \min \left\{\frac{\pi \alpha}{2^4 \mathcal{K}_1},\frac{\pi}{\mathcal{K}_1}\right\}, \quad s=D,U, \, \nu \geq 1,
\end{equation}
with $\mathcal{K}_1$ given by (\ref{E:diff Ptildebar}).

Starting from
\begin{equation}
Y^1=\begin{cases}\begin{array}{lll}P(\epsilon,x)-I, \quad &\epsilon \in V_L,\\
0, \quad &\epsilon \in V_R,
\end{array}
\end{cases}
\end{equation}
and using (\ref{E:diff Ptildebar}), we construct, for $\nu=2,3,...$, a sequence of matrices $Y^\nu$, $Y^\nu_{U}$ and $Y^\nu_{D}$ satisfying the conditions:
\renewcommand{\theenumi}{\roman{enumi}}
\begin{enumerate}
\item \label{decomposition y}$Y^{\nu-1}=Y^\nu_{U}-Y^\nu_D$, $(\epsilon,x) \in V_\cap(\nu-1) \times \mathbb{D}_r$;
 \item \label{info Zs y} for $s=D,U$, \begin{itemize}
 \item $Y^\nu_s$ is analytic for $(\epsilon,x) \in V_s(\nu-1) \times \mathbb{D}_r$,\\
 \item $|Y^\nu_s| \leq 2^{-(\nu+1)}$ for $(\epsilon,x) \in V_s(\nu) \times \mathbb{D}_r$;
 \end{itemize}
\item \label{induction y} For some $0<\delta<1$, \begin{itemize}\item $I+Y^{\nu}=(I+Y^\nu_D)(I+Y^{\nu-1})(I+Y^\nu_U)^{-1}$ for $(\epsilon,x) \in V_L(\nu-\delta) \times \mathbb{D}_r$, \item $Y^{\nu}=0$ on $V_R(\nu-\delta) \times  \mathbb{D}_r$;\end{itemize}
 \item \label{info Z y} \begin{itemize}\item $Y^\nu$ is analytic for $(\epsilon,x) \in V_L(\nu-\delta) \times \mathbb{D}_r$,\\
 \item $Y^\nu(0,x)=0$, \\
 \item $Y^\nu$ satisfies, with $\mathcal{K}_1$ given by (\ref{E:diff Ptildebar}),\\
 $|Y^\nu| \leq 2^{-2(\nu-1)} \mathcal{K}_1 |\epsilon|  $ for $(\hat{\epsilon},x) \in V_L(\nu) \times \mathbb{D}_r$.
  \end{itemize}
\end{enumerate}

We can prove that the properties (\ref{decomposition y}) to (\ref{info Z y}) are satisfied in a similar (and simpler) way as in Section \ref{S:construc seq}, by defining the matrices $Y_D^\nu (\epsilon,x)$ and $Y_U^\nu (\epsilon,x)$, for $\nu=2,3,...$, by
\begin{equation}\label{E:def y nu s}
Y^\nu_s(\epsilon,x)=\frac{1}{2 \pi i}\int_{t_{\nu-1,s}}\frac{Y^{\nu-1}(h,x)}{h-\epsilon}dh, \quad (\epsilon,x) \in V_s(\nu-1) \times \mathbb{D}_r, \, s=D,U.
\end{equation}
As in Section \ref{S:const H}, the desired $J(\hat{\epsilon},x)$ is given by
\begin{equation}
J(\hat{\epsilon},x)=\begin{cases}J_D(\hat{\epsilon},x), \quad \hat{\epsilon} \in V_D, \\J_U(\epsilon,x), \quad \hat{\epsilon} \in V_U,
\end{cases}
\end{equation}
with
\begin{equation}\label{E:sol J fct Z}
J_s(\epsilon,x)=\lim_{\nu \to \infty} (I+Y_s^\nu)...(I+Y_s^3)(I+Y_s^2), \quad s=D,U.
\end{equation}
By (\ref{info Zs y}), $J(\hat{\epsilon},x)^{-1}$ has a bounded limit at $\epsilon=0$. Since the family$\{J'(\hat{\epsilon},x)\}_{\hat{\epsilon} \in (S\cup \{0\})}$ is bounded, $J'(\hat{\epsilon},x)$ has a bounded limit at $\epsilon=0$.
This concludes the proof of Theorem \ref{T:reala globa}.
\hfill $\Box$

\section{Discussion and directions for further research}\label{S:Discussion}
The work presented in this paper brings a new light on the divergence of formal solutions near an irregular singular point of Poincaré rank $1$. It gives new perspectives, including a unified point of view in the understanding of the dynamics of the singularities by deformation. We have identified, interpreted and studied the realization of the complete system of analytic invariants of unfolded differential linear systems with an irregular singularity of Poincaré rank $1$ (nonresonant case). The meaning of the auto-intersection relation (which is the necessary and sufficient condition for the realization) is still obscure (in dimension $n \geq 3$). We will investigate it in more details in \cite{cLR3}.

One of the next steps in the large program of understanding singularities by unfolding is the study of analytic invariants of nonresonant linear differential equations with singularities of Poincaré rank $k$ higher than $1$. One difference is that there is no more a bijection between the $2k$ Stokes matrices and the $k+1$ singular points in the unfolded systems.

Another direction of research is the existence of universal families. Can we identify canonical representatives of the analytic equivalence classes of unfolded systems?

\section*{Acknowledgements} The authors are grateful to Claudine Mitschi and to the reviewer for helpful suggestions and comments. They also thank Sergei Yakovenko for stimulating discussions.

\end{document}